\documentclass[11pt]{amsart}

\usepackage{cancel}
\usepackage{xcolor}

\usepackage{soul, xcolor}
\setstcolor{red}

\usepackage{amsmath,amssymb,amsfonts,amsthm,enumerate}
\usepackage{eucal} 
\usepackage{xcolor}
\usepackage{youngtab}
\usepackage{young}
\usepackage{tikz}
\usepackage{tikz-cd}
\usepackage{mathabx,epsfig}
\usepackage[colorinlistoftodos]{todonotes}
\usepackage[markup=underlined]{changes}
\usepackage{mathrsfs}

\DeclareMathOperator{\blendpar}{bl}
\DeclareMathOperator{\hs}{hs}
\DeclareMathOperator{\hilbcomb}{Hilb}

\DeclareMathOperator{\pfhilbcomb}{PFHilb}

\DeclareMathOperator{\evmap}{ev}

\usepackage{mathtools}
\usepackage{mathabx}

\newcommand{\ferrers}{d}

\newcommand{\conduc}{\kappa}

\makeatletter
\newcommand{\Flatletter}{\hspace {2.3pt}
  \kern 0.4pt% Small left shift
  \tilde{\kern -2.3pt\mathcal{F}l\kern -0.4pt}%
  \kern 0.4pt% Small right shift
}
\makeatother

\makeatletter
\newcommand{\hbm}{\hspace {2.3pt}
  \kern 0.4pt% Small left shift
  \overline{\kern -2.3ptH\kern -0.4pt}%
  \kern 0.4pt% Small right shift
}
\makeatother

\makeatletter
\newcommand{\rbar}{\hspace {2.3pt}
  \kern 0.4pt% Small left shift
  \overline{\kern -2.3pt R\kern -0.4pt}%
  \kern 0.4pt% Small right shift
}

\makeatother

\DeclareMathOperator{\inideal}{in}
\newcommand{\torusrotelt}{z}
\newcommand{\flaglat}{L}
\newcommand{\spclass}{Y}
\newcommand{\coinvid}{\mathfrak{m}^{S_n}_+}

\newcommand{\lshuffs}[1]{\text{Sh}_{#1}^{<}}
\newcommand{\rshuffs}[1]{\text{Sh}_{#1}^{>}}
\newcommand{\shuffs}[1]{\text{Sh}_{#1}}
\newcommand{\cycdec}[1]{\mathcal{D}_{#1}}

\newcommand{\indt}{\mathbf{ind}}

\newcommand{\zptwo}{\mathbb{Z}_{\geq 0}^2}
\DeclareMathOperator{\im}{Im}
\DeclareMathOperator{\fltaucomb}{Z}

\newcommand{\reshess}{p}
\newcommand{\schedhess}{q}

\newcommand{\hf}{h}

\newcommand{\blendlam}{\boldsymbol{\lambda}}

\DeclareMathOperator{\ind}{ind}
\DeclareMathOperator{\indres}{ind}

\newcommand{\frob}{\mathcal{F}}
\newcommand{\aux}{\mathrm{aux}}
\DeclareMathOperator{\latl}{Lat}
\DeclareMathOperator{\flatl}{FLat}
\newcommand{\descpf}{\majt}

\newcommand{\antisym}{\mathcal{N}}
\newcommand{\Par}{\mathcal{P}}
\DeclareMathOperator{\signrep}{sgn}
\DeclareMathOperator{\grdim}{grdim}
\DeclareMathOperator{\mdeg}{mdeg}

\DeclareMathOperator{\pfl}{PF}
\newcommand{\lpfl}[1]{\pfl_{#1}^{<}}
\newcommand{\rpfl}[1]{\pfl_{#1}^{>}}

\DeclareMathOperator{\schedpfl}{SchedPF}
\newcommand{\lschedpfl}[1]{\schedpfl_{#1}^{<}}
\newcommand{\rschedpfl}[1]{\schedpfl_{#1}^{>}}

\newcommand{\minosp}{\Pi}
\DeclareMathOperator{\osp}{OSP}
\newcommand{\invt}{\mathbf{inv}}
\newcommand{\majt}{\mathbf{maj}}

\DeclareMathOperator{\readword}{word}
\newcommand{\areat}{\mathbf{area}}
\newcommand{\coareat}{\mathbf{coarea}}
\newcommand{\bx}{{\mathbf{x}}}
\newcommand{\by}{\mathbf{y}}
\newcommand{\bz}{\mathbf{z}}

\newcommand{\ba}{\mathbf{a}}

\newcommand{\bb}{\mathbf{b}}
\newcommand{\bc}{\mathbf{c}}
\newcommand{\bk}{\mathbf{k}}
\newcommand{\gf}{\mathbb{C}}

\newcommand{\C}{\mathbb{C}}

\newcommand{\Z}{\mathbb{Z}}

\newcommand{\cL}{\mathcal{L}}

\newcommand{\latflag}{E_\bullet}
\newcommand{\parflag}{\lambda^\bullet}

\newcommand{\affperms}{W}
\newcommand{\extperms}{\hat{W}}

\newcommand{\fltaugeo}{\mathcal{Z}}

\newcommand{\lang}{\left\langle}
\newcommand{\rang}{\right\rangle}

\newcommand{\grquot}[3]{\Grr_{#1}(#2/#3)}

\newcommand{\grsup}{{g}}
\newcommand{\flsup}{{f}}

\newcommand{\dyckpath}{\pi}

\newcommand{\eps}{\varpi}
\newcommand{\cA}{\mathcal{A}}

\newcommand{\fg}{\mathfrak{g}}
\newcommand{\fghat}{\widehat{\mathfrak{g}}}

\newcommand{\fthat}{\widehat{\mathfrak{t}}}

\newcommand{\fsl}{\mathfrak{sl}}
\newcommand{\fslhat}{\widehat{\mathfrak{sl}}}
\newcommand{\What}{\widehat{W}}
\newcommand{\That}{\widehat{T}}

\newcommand{\flag}{\mathcal{F}l}
\newcommand{\Fl}{\tilde{\mathcal{F}l}}
\newcommand{\Gr}{\tilde{\mathcal{G}}r}

\newcommand{\affgrass}{\tilde{\mathcal{G}}r}

\newcommand{\agrass}{\tilde{\mathcal{G}}\mathrm{r}}
\newcommand{\grr}{\mathcal{G}\mathrm{r}}
\newcommand{\Sp}{\tilde{\mathcal{S}}}

\newcommand{\tS}{\tilde{S}}
\newcommand{\tSn}{W}

\newcommand{\spres}{\mathfrak{u}}
\newcommand{\Tres}{U}

\newcommand{\sig}{\sigma}

\newcommand{\Aaf}{\mathbb{A}_{af}}
\newcommand{\aaf}{\Aaf}
\newcommand{\aafo}{\aaf^0}

\newcommand{\Saf}{\mathbb{S}_{af}}

\newcommand{\sktof}{q}

\newcommand{\affsptorus}{\spres}

\DeclareMathOperator{\hesscomb}{Hess}

\DeclareMathOperator{\sched}{sch}

\newcommand{\PF}{PF}
\newcommand{\awg}{W}

\newcommand{\rotelt}{\rho}
\newcommand{\shiftelt}{\psi}

\newcommand{\mplusxy}{\coinvid(\bx,\by)}
\newcommand{\runs}{\mathbf{r}}
\newcommand{\run}{r}
\DeclareMathOperator{\scheds}{Sched}
\newcommand{\hesspav}{{Y^\circ}}
\newcommand{\hessleq}{{Y}}
\newcommand{\ressched}{\beta}
\newcommand{\respf}{\alpha}

\DeclareMathOperator{\extcomb}{ext}
\DeclareMathOperator{\latcomb}{lat}
\newcommand{\extmap}{ext}

\DeclareMathOperator{\moda}{mod\ }
\DeclareMathOperator{\Ad}{Ad}

\DeclareMathOperator{\ldeg}{ord}

\DeclareMathOperator{\Ext}{Ext}

\DeclareMathOperator{\Lie}{Lie}

\DeclareMathOperator{\Sym}{Sym}

\DeclareMathOperator{\Res}{Res}
\newcommand{\lRes}[1]{\Res^{<}_{#1}}
\newcommand{\rRes}[1]{\Res^{>}_{#1}}
\DeclareMathOperator{\Stab}{Stab}

\DeclareMathOperator{\codim}{codim}
\DeclareMathOperator{\Hom}{Hom}

\DeclareMathOperator{\gr}{gr}
\DeclareMathOperator{\sort}{sort}
\DeclareMathOperator{\area}{area}

\DeclareMathOperator{\maj}{maj}
\DeclareMathOperator{\comaj}{comaj}
\DeclareMathOperator{\inv}{inv}
\DeclareMathOperator{\cars}{cars}

\DeclareMathOperator{\dinv}{dinv}

\DeclareMathOperator{\Spec}{Spec}

\newcommand{\ghilbgeo}{\mathcal{G}}
\newcommand{\hilbgeo}{\mathcal{H}\mathrm{ilb}}
\newcommand{\hessgeo}{\mathcal{H}\mathrm{ess}}
\newcommand{\philbgeo}{\hilbgeo_0}
\newcommand{\pfhilbgeo}{\mathcal{PFH}\mathrm{ilb}}
\newcommand{\ppfhilbgeo}{\pfhilbgeo_0}
\newcommand{\pfghilbgeo}{\mathcal{PFG}}

\newtheorem{thm}{Theorem}

\newtheorem{lemma}{Lemma}
\newtheorem{cor}[lemma]{Corollary}
\newtheorem{prop}{Proposition}
\theoremstyle{definition}
\newtheorem{defn}{Definition}
\newtheorem{remark}{Remark}
\newtheorem{ex}{Example}

\newtheorem*{thm1}{Theorem A}
\newtheorem*{thm2}{Theorem B}
\newtheorem*{thm3}{Theorem C}
\newtheorem*{cor1}{Corollary A}

\newtheorem*{shuffthm0}{Theorem 1}
\newtheorem*{shuffthm1}{Theorem 1$'$}
\newtheorem*{shuffthm2}{Theorem 1$''$}

\numberwithin{thm}{section}
\numberwithin{obs}{section}
\numberwithin{prop}{section}
\numberwithin{lemma}{section}
\numberwithin{conj}{section}
\numberwithin{defn}{section}
\numberwithin{remark}{section}
\numberwithin{equation}{section}

\newcommand{\calO}{\mathcal{O}}

\newcommand{\calK}{\mathcal{K}}
\newcommand{\ZZ}{\mathbb{Z}}
\newcommand{\CC}{\mathbb{C}}

\newcommand{\frg}{\mathfrak{g}} 
\newcommand{\grH}{\mathfrak{H}
^{\mathrm{gr}}}
\newcommand{\frt}{\mathfrak{t}}

\newcommand{\gext}{\underline{\mathrm{ext}}}

\newcommand{\gl}{\mathfrak{gl}}

\DeclareMathOperator{\mmod}{mod}

\newcommand{\Grr}{\mathcal{G}\mathrm{r}}
\newcommand{\gFl}{\mathcal{F}l}

\definechangesauthor[color=red]{EC}
\definechangesauthor[color=red]{AO}

\title{Affine Schubert calculus and double coinvariants}

\author{Erik Carlsson}
\address{Department of Mathematics\\
Mathematical Sciences Building \\
One Shields Ave.\\
University of California
Davis, CA 95616}
\email{ecarlsson@math.ucdavis.edu}
\author{Alexei Oblomkov}
\address{Department of Mathematics and Statistics\\
Lederle Graduate Research Tower\\
University of Massachusetts Amherst\\
710 N. Pleasant Street\\
Amherst, MA 01003-9305, USA}
\email{oblomkov@math.umass.edu}


\begin{document}

\begin{abstract}
  We define an action of the double coinvariant algebra
  $DR_n$ on the equivariant Borel-Moore homology of the affine flag variety
  $\Fl_n$ in type $A$, which has an explicit form in terms of
  the left and right action of the (extended) affine Weyl group and multiplication by Chern classes.
  Up to first order in the augmentation ideal, we show that it coincides with the
  action of the Cherednik algebra on the equivariant homology of
  the homogeneous affine Springer fiber $\Sp_{n,n+1} \subset \Fl_n$
  due to Yun and the second author \cite{OblomkovYun16}, and therefore
  preserves the non-equivariant Borel-Moore homology groups
$\hbm_*(\Sp_{n,n+1})\hookrightarrow \hbm_*(\Fl_n)$.
  We then define a geometric filtration 
  $F_{\ba} \hbm_*(\Sp_{n,n+1})=\hbm_*(\hessleq(\ba))$
  by closed subspaces $\hessleq(\ba)\subset \Sp_{n,n+1}$,
  %related to the 
  %\emph{Hessenberg paving} of \cite{goresky2003purity},
  which we prove recovers the Garsia-Stanton descent order on $DR_n$.  
We use this to deduce an explicit monomial basis of $DR_n$, as well as an independent proof of the (non-compositional) Shuffle Theorem \cite{haglund2005conjectured,CarlssonMellit12}.
\end{abstract}

\maketitle

\tableofcontents
\section{Introduction}
\label{sec:introduction}

The double (or diagonal) 
coinvariant algebra is the quotient
space of the polynomial algebra
$\gf[\mathbf{x},\mathbf{y}]=\gf[x_1,\dots,x_n,y_1,\dots,y_n]$
in $2n$ variables by the ideal generated by nonconstant
diagonally symmetric polynomials
\[DR_n=\gf[\bx,\by]/\mplusxy,\quad
  \mplusxy=\lang \sum_{k} x_k^i y_k^j:
  (i,j)\neq (0,0)\rang.\]
Since $\coinvid(\bx,\by)$ is doubly homogeneous, we find that $DR_n$ is a doubly graded vector space by the degree in the $x$ and $y$ variables respectively. Additionally, there is a \emph{diagonal} action of $S_n$ on $DR_n$ by
\[(\sigma f)(\bx,\by)= f(\bx_{\sigma},\by_{\sigma}),\]
where $\bx_\sigma=(x_{\sigma_1},...,x_{\sigma_n})$ and similarly for $\by_\sigma$.
This space was studied by Haiman,
who also proved that this space has dimension $(n+1)^{n-1}$ \cite{Haiman02}.

In \cite{haglund2005conjectured}, Haglund and Loehr conjectured the combinatorial formula for the bigraded Hilbert series
of $DR_n$ in terms of certain parking function statistics called $\area$ and $\dinv$, given in Definition~\ref{def:dinv} of Section \ref{sec:ratpark}:
\begin{equation}
    \label{eq:introhagloehr}
    \grdim_{q,t} DR_n=\sum_{P \in \pfl(n)} q^{\dinv(P)}t^{\area(P)}
=\sum_{\tau \in S_n} t^{\maj(\tau)} \prod_{i=1}^n [\sched_i(\tau)]_q.
\end{equation}
We also have the more general \emph{Shuffle conjecture}, which also encodes the character of the action of the symmetric group on $DR_n$.
The Shuffle conjecture was first proven 
by the first author and Mellit
in \cite{CarlssonMellit12}, as well as the more general ``rational'' case
by Mellit \cite{Mellit16}.
On the right hand side of \eqref{eq:introhagloehr}, the statistic $\maj(\tau)$ is the major index, $[k]_q=1+\cdots+q^{k-1}$ is the $q$-number, and
$\sched_i(\tau)$ are certain positive integers vectors known as ``schedules''
\cite{hicks2005thesis,haglund2008catalan}.
The schedules version of the Shuffle Theorem will be particularly relevant in this paper.

Separately, several articles due to Lusztig-Smelt \cite{LusztigSmelt91},
Gorsky-Mazin \cite{GorskyMazin2013,GorskyMazin2014}, Hikita \cite{hikita2014affine}, and Gorsky-Mazin-Vazirani \cite{gorsky2014rational} (see Section~\ref{sec:affine-permutations}), have connected the combinatorics of a rational extension of the Haglund-Loehr formula with a basis of the homology of certain affine Springer fiber in type $A$ denoted $\Sp_{n,m}\subset \Fl_n$. { Here $\Fl_n$ is the
\emph{affine flag variety} in type A, defined in Section
\ref{sec:geometryFl}}.

% This version,
% described the full shuffle case in \eqref{eq:shuffres}, depends on a 
% a pair of coprime positive integers $(n,m)$, and recovers the non-rational version by settings $m=n+1$.

{ In another direction}, 
the second author and Yun have shown that the cohomology of $\Sp_{n,m}$ is an irreducible module
\(\mathfrak{L}_{m/n}(triv)\)
over the rational Cherednik algebra \(\mathfrak{H}^{rat}_{m/n}\)
\cite{OblomkovYun16}, see Section~\ref{sec:cher-algebra-acti} for more details.
It was known from \cite{Gordon03} that
there is an isomorphism of \emph{singly} graded spaces $DR_n \cong  \mathfrak{L}_{(n+1)/n}(triv)$. The grading in the isomorphism is the
anti-diagonal grading: \(\deg(x_k)=-\deg(y_k)=1\) for all \(k\).
On the rational Cherednik algebra side this grading appears naturally from the \(\mathfrak{sl}_2\) action on \(\mathfrak{H}^{rat}_{m/n}\) \cite{BerestEtingofGinzburg03}.

% but
% this did not result in a proof of \eqref{eq:introhagloehr}.
%  , because Hikita's statistic could not be
 % connected to the grading by 
 % the $y$-variables in $DR_n$.

The action of $\mathfrak{H}^{rat}_{m/n}$ on $H^*(\Sp_{n,m})$ is closely related to
an action of the affine Weyl group
\begin{equation}
\label{eq:introaffweyl}    
W=\left\{w: \mathbb{Z}\rightarrow \mathbb{Z}:
w_{i+n}=w_i+n,\ \sum_{i=1}^n w_i=n(n+1)/2\right\}.
\end{equation}
which is essentially the Springer action
on $H^*(\Sp_{n,m})$.

The description of Kostant and Kummar \cite{kostant1986nil} of the cohomology of the affine flag variety $H^*(\Fl)$  (see Proposition \ref{prop:kostantkumar}) in terms of 
nill Hecke algebras yields
two actions of the group
\[\What=\left\{w: \mathbb{Z}\rightarrow \mathbb{Z}:
w_{i+n}=w_i+n \right\}.\]
on \(H^*(\Fl\times \ZZ)\), left and right.
The subgroup \(W\subset \What\) preserves \(H^*(\Fl\times 0)=H^*(\Fl)\)  and
the restriction map intertwines the above \(W\)-action  on \(H^*(\Sp_{n,m})\) with the right action of $W$ on $H^*(\Fl)$

{ Using the} two \(\What\)-actions on \(H^*(\Fl\times \ZZ)\) we can construct the operator
$\Ad_\rho$ of conjugation by $\rho\in \What$ (see Definition \ref{def:xyaction})
on $H^*(\Fl\times 0)=H^*(\Fl)$,
where $\rho_i=i+1$ is an ``extended'' affine permutation.
%, meaning it does not satisfy
%the second condition in \eqref{eq:introaffweyl}.
%\st{which are collectively denoted by $\What$}
%{\color{blue} The group  generated by \(\rho\) and \(W\) is denoted by
%$\What$, and is described
%in Section \ref{sec:affine-permutations}}.
The right \(W\)-action and \(\Ad_\rho\) have versions in 
Borel-Moore homology 
$\hbm_*(\Fl)$, as well as their
equivariant versions.

The first main result of this paper defines an action of $DR_n$ on $\hbm_*(\Fl)$
in type A. In this construction, 
the $x_i$ variables act by multiplication by Chern classes of the natural line bundles. 
The $y_i$ variables are defined in terms of 
the operators from the previous paragraph by
\begin{equation}
\label{eq:introyi}    
y_i=z_i-1,\ \ 
z_i(f)= \rho f \psi_i^{-1}= \Ad_\rho(f) (\rho \psi_i^{-1}),\quad f\in\hbm_*(\Fl)
\end{equation}
where $\rho$ is as above and $\psi_i:\mathbb{Z}\rightarrow \mathbb{Z}$ is the extended affine permutation
\begin{equation*}
{\psi_i(j)}=\begin{cases} j+n & j \cong i\ \ (\mmod\ n) \\ j & \mbox{otherwise} \end{cases},\end{equation*}
such that \(\rho\psi^{-1}_i\in W\).
We also let $\Delta_n \in \hbm_*(\Fl)$ 
be the Schubert class associated to the permutation $w_0=(n,...,1)\in S_n \subset W$.

Our first theorem determines an action of
$DR_n$ on
$\hbm_*(\Sp_{n,n+1})$, which can be realized as a subspace $\hbm_*(\Fl)$
by Proposition~\ref{injsurjprop} and       Theorem~\ref{thma} below.
\begin{thm1}
  \label{thm1}
  The operators $x_i,y_i$ define an action of $DR_n$ on 
  the { Borel-Moore} homology of the affine
  flag variety that preserves 
  $\hbm_*(\Sp_{n,n+1}) \subset \hbm_*(\Fl)$ for coprime $(n,m)$.
  %In the case $m=n+1$,
  This action induces an isomorphism
  $\hbm_*(\Sp_{n,n+1}) \cong DR_n$
  by applying $f\in DR_n$ to the generator
  $\Delta_n$.
\end{thm1}

To prove Theorem A, we explicitly construct the
action of  \(\C[\mathbf{x},\mathbf{y},\epsilon]\) on the equivariant Borel-Moore homology
\(\hbm_*^{\CC^*}(\Fl)\) of the affine flag variety \(\Fl\), inducing an action of
$\C[\mathbf{x},\mathbf{y}]$ on nonequivariant homology.
Here $\epsilon$ is the additional parameter coming from the action of the torus.
We then show that this action agrees up to first order in $\epsilon$
with the noncommutative 
action { of} Oblomkov and Yun \cite{OblomkovYun16} that preserves
the subspace \(\hbm_*^{\mathbb{C}^*}(\Sp_{n,n+1})\hookrightarrow H^{\mathbb{C}^*}_*(\Fl)\), which is an injection by  Proposition~\ref{injsurjprop} and       Theorem~\ref{thma}.
Then the compatibility of our action
with the surjection
$\hbm_*^{\C^*}(\Sp_{n,n+1})\rightarrow
\hbm_*(\Sp_{n,n+1})$
shows that
$\C[\mathbf{x},\mathbf{y}]$ preserves the subspace
$\hbm_*(\Sp_{n,n+1})$ in the nonequivariant setting as well.

The rest of the proof is based on 
affine Schubert calculus, 
as discussed in \cite{lam2014book}, 
in which one identifies 
$\hbm_*(\Fl)\cong \Lambda$,
where $\Lambda=R_n(\bx) \otimes \C[h_1,...,h_{n-1}]$ is the one-variable coinvariant algebra \(R_n(\bx)\) tensored with a polynomial ring generated by the complete symmetric functions $h_1,...,h_{n-1}$.
We give explicit formulas
for the action of $x_i,y_i$ on $\Lambda$, 
which are used to show that $x_i,y_i$ commute, and that
$\coinvid(\bx,\by)$ acts by zero.
In order to show that $DR_n$ injects into
$\Lambda$, we use a result of Haiman \cite{Haiman02} regarding the restriction map of sections of the Procesi bundle on the punctual Hilbert scheme to a certain affine open subspace. This suggests an interesting interpretation of $\Lambda$ as the space of sections of a vector bundle on a certain dense open subset of the Hilbert scheme 
of points on $\C^2$, see Remark~\ref{rem:open-Haiman}.

In our second result, we consider a 
family of closed topological subspaces \[\hessleq(\ba)\subset \Sp_{n,n+1}\]
for $\ba=(a_1,...,a_n)\in 
\mathbb{Z}_{\geq 0}$, defined as unions of 
cells $Y_w^\circ \subset \Sp_{n,n+1}$,
which are the
Schubert cells 
{ in the affine flag variety} 
intersected with $\Sp_{n,n+1}$. These subspaces have the property that
$\hessleq(\ba)\subset \hessleq(\bb)$
whenever $\ba \leq_{des} \bb$,
where $\leq_{des}$ is the Garsia-Stanton descent order
\cite{garsia1984group,allen1994descent}
defined by
$\ba\leq_{des} \bb$ if
\begin{enumerate}
    \item $\sort_>(\ba)<_{lex} \sort_>(\bb)$, or
    \item $\sort_>(\ba)= \sort_>(\bb)$
and $\ba \leq_{lex} \bb$.    
\end{enumerate}
Here $\sort_>(\ba)$ sorts $\ba$ in reverse order to obtain a partition, and $\leq_{lex}$ is the usual lexicographic order. 
The standard monomials of $R_n(\by)$ with respect to this order are known as the Garsia-Stanton descent basis, given by
\begin{equation}
\label{eq:introdescmon}    
g_\tau(\by)=\prod_{i\in \mathrm{Des}(\tau)}(y_{\tau_1}\cdots y_{\tau_i}).
\end{equation}
where $\tau\in S_n$ ranges over the usual permutations,
and $\mathrm{Des}(\tau)$ is the set of indices $1\leq i\leq n-1$ for which $\tau_{i}>\tau_{i+1}$. The exponent
vector in $\by^{\ba}=g_{\tau}(\by)$ is denoted
$\ba=\majt(\tau)$, so that 
the degree of $g_\tau(\by)$ is the major index $\maj(\tau)=|\majt(\tau)|$.

We then consider the filtration on the Borel-Moore homology defined by taking $F_{\ba}\hbm_*(\Sp_{n,n+1})$ to be the image of
$\hbm_*(\hessleq(\ba))$ under the pushforward of the inclusion map, which is injective by Lemma \ref{desbrulebm}, which relates the descent order to the Bruhat order.
Under the isomorphism
$\hbm_*(\Sp_{n,n+1}) \cong DR_n$ from Theorem A,
we obtain a filtration $F_{\ba} DR_n$ by vector subspaces, which are in fact $\C[\bx]$-submodules,
due to the fact that the $x_i$ act by Chern classes, and so are compatible with inclusion maps by the projection formula. 
Our second \st{main} theorem interprets $F_\ba DR_n$ in terms of the descent order on the monomials $\by^{\ba}$, and uses it to produce a monomial basis $DR_n$:

\begin{thm2}
Let $F_\ba DR_n$ as above, and let 
$G_\ba DR_n=F_\ba DR_n/F_{<_{des} \ba} DR_n$ be the associated graded components. Then
  \begin{enumerate}[a)]
  \item We have that $F_\ba DR_n$ is the $\C[\bx]$-submodule of $DR_n$ spanned by the monomials $\by^{\ba'}$
  for $\ba'\leq_{des} \ba$.
  \item \label{item:introbasis} We have a vector space basis of $DR_n$ given by 
  \begin{equation}
  \mathcal{B}_n=
      \{\by^{\majt(\tau)} x_{\tau_1}^{k_1}\cdots x_{\tau_n}^{k_n}
:\tau \in S_n,\  0\leq k_i \leq \sched_i(\tau)-1     
      \}
  \end{equation}
\item As a $\C[\bx]$-module, $G_\ba DR_n$ is zero unless $\ba=\majt(\tau)$ for some $\tau$, in which case it is isomorphic to the principal ideal
$(f_\tau(\bx))$ in $R_n(\bx)$ where
 \[f_{\tau}(\bx)=
 x_{\tau_1}\cdots x_{\tau_{n-l}}
 \prod_{i=1}^n \prod_{j=i+\sched_i(\tau)+1}^{n} (x_{\tau_i}-x_{\tau_j}),\]
and $l$ is the multiplicity of zero in $\ba$.
  \end{enumerate}  
\end{thm2}

Noting that $\sched_i((1,...,n))=n-i+1$,
we see that the monomial basis in item
\ref{item:introbasis}) interpolates between the Garsia-Stanton basis of $R_n(\by)$
and the standard Artin (or sub-staircase) basis of $R_n(\bx)$, which are subspaces of $DR_n$.
As an immediate consequence, we recover the schedules version of \eqref{eq:introhagloehr}.

Using elementary arguments combined with the
\emph{ungraded} version of the 
Shuffle Theorem, which follows from 
Proposition 3.13 of \cite{Haiman02}, 
we are then able to recover the full bigraded Frobenius character of $DR_n$. For any
Young subgroup
$S_\mu=S_{\mu_1}\times \cdots \times S_{\mu_l}\subset S_n$, we can identify the corresponding anti-invariants 
using the anti-symmetrization operator
\begin{equation}
\label{eq:introcoranti}    
(DR_n \otimes \signrep)^{S_\mu}=
\antisym_\mu DR_n,\ \ 
\antisym_\mu f(\bx,\by)=\sum_{\sigma \in S_{\mu}} \signrep(\sigma) f(\bx_{\sigma},\by_\sigma).
\end{equation}
%in Corollary \ref{cor:shf} below:
Then the following corollary implies an independent proof of the (non-compositional) restated Shuffle Theorem when combined with \eqref{eq:drntonabla} (
see Theorem~\ref{thm:shuffpf} and the discussion after for more details):
%as predicted by the 
%
%Shuffle Theorem together with Haiman's formula 
%as follows\
\begin{cor1}
For any composition $\mu$ we have a basis of
$(DR_n \otimes \signrep)^{S_\mu}$ given by
\begin{equation}
    \label{eq:youngbasis}
    \mathcal{B}_\mu=\left\{
    \antisym_{\mu} \big(
    \mathbf{y}^{\mathbf{area}(P)}_{\sigma}
    \mathbf{x}^{\mathbf{dinv}(P)}\big):P=(\pi,\sigma)\in\mathrm{PF}^{>}_\mu(n)\right\},
\end{equation}
where $\mathrm{PF}^{>}_\mu(n)$ is a collection
of $\mu$-sorted parking functions as formulated in 
Theorem 1 below.
\end{cor1}

Though it is not needed to prove
Corollary A, 
a result of
 Borho-MacPherson \cite{borho1983partial}
 shows that the $S_\mu$- anti-invariants of the Weyl group action on the usual nilpotent Springer fibers recover the cohomology of the Springer fibers in partial flag varieties.
This suggests that using a version of
 Borho-MacPherson result for the affine Springer fiber 
 and studying Schubert cells in the parabolic case may produce a proof
 of Corollary A that does not rely on the ungraded Shuffle Theorem either.

 We are optimistic that the methods of the 
 paper, when combined with the more general results of \cite{OblomkovYun16,OblomkovYun17}, could be extended to the rational Shuffle Theorem, corresponding to more general affine Springer fiber \(\Sp_{n,m}\), as well as more general root systems.  In particular, the recent construction of the monomial basis for the diagonal super-coinvariants, due to  Haglund and Sergel
 \cite{HaglundSergel21}, begs for a geometric interpretation, possibly similar to the setting of the current paper.

The  proof of Theorem B is more involved than that of Theorem A. It involves studying the  action of $x_i,y_i$ on
equivariant Borel-Moore homology
in the fixed point basis, 
and identifying, as explained in Lemma~\ref{zinfalem}, the nonzero matrix elements in this torus fixed-point basis.
{ This uses} a combinatorial bijection between
$\Sp_{n,n+1}^{\C^*}$ and parking functions, as described in Section \ref{sec:hesscomb}.
In this way, 
certain subsets of parking functions
called $\cars(\tau)$ correspond to the torus fixed points called $\Res(\tau)\subset \Sp_{n,n+1}$,
which appear in $\hessleq(\ba)$, but not $\hessleq(\ba')$ for any $\ba'<_{des} \ba$, see Lemma \ref{zinfalem}.
The statistics such as $\dinv$ are used in dimension arguments. 
%While we give an algebraic-combinatorial proof we just described, 
%this is related to the Hessenberg paving
%\cite{goresky2003purity}, which we explain in Section \ref{sec:geometry}.

The critical step is to show that the monomials in item \ref{item:introbasis}) are linearly independent in
$G_\ba DR_n$, which is done in Corollary \ref{cor:gabasis} below. To do this,
we translate this into a statement that their equivariant versions
are $\C[\epsilon]$-independent in
$\hbm_*^{\C^*}(\Sp_{n,n+1})$,
and consider their expansions in the fixed point basis, 
enumerated by $\Res(\tau)$. We then produce another bijection which
identifies $\Res(\tau)$ with a subset 
$\mathrm{Hess}(\tau)\subset \hesscomb(h_\tau)$ of the torus fixed points of the regular nilpotent Hessenberg variety denoted
$\hessgeo(N,h_\tau)$, where $h_\tau$ is a certain combinatorial Hessenberg function associated to $\tau$. We then translate this back into geometry, making use of a certain monomial basis of $H^*(\hessgeo(N,h))$ presented
in \cite{harada2021filtration}.

The argument we just described is clearly ultimately a geometric one, which we make precise in our third result. Specifically,
 given a permutation $\tau\in S_n$, let $\hesspav(\tau)$ denote the complementary subspace given by
\begin{equation}
\label{eq:defsptau}
    \hesspav(\tau)=\hessleq(\ba)-\bigcup_{\ba'<_{des} \ba} \hessleq(\ba')\subset \Sp_{n,n+1},\ \ 
    \ba=\majt(\tau).
  \end{equation}
These spaces yield a paving of the affine Springer fiber:  \(\Sp_{n,n+1}=\bigsqcup_{\tau\in S_n} \hesspav(\tau).\)
We prove the following, listed as Theorem \ref{thm:geo}:
%While we retain the explicit argument given above, we prove the following theorem:
\begin{thm3} We have that
$\hesspav(\tau)$,\(\tau\in S_m\) is isomorphic to a local trivial
affine bundle 
%where
%$\mathcal{E}_\tau$ is a vector bundle 
over a smooth base, which is the intersection of a certain Schubert variety \(\fltaugeo(\tau)\)
inside the usual flag variety \(\flag_n\) with
the regular nilpotent Hessenberg variety 
$\hessgeo(N,h)\subset \flag_n$ for a particular 
Hessenberg function $h=h_\tau$.
\end{thm3}
Like the affine Springer fiber, these 
Hessenberg varieties have an affine paving by intersected Schubert cells by a result of Tymoczko
\cite{tymoczko2007paving}, 
as do their intersection with $\fltaugeo(\tau)$ (see Lemma \ref{lem:pfhcells}). 
Since $\hesspav(\tau)$ is paved by affines, its Borel-Moore homology is nonzero only in even degree, and we have a short exact sequence in Borel-Moore homology
\begin{equation}
0\rightarrow \sum_{\ba'<_{des} \ba}\hbm_*(\hessleq(\ba'))\rightarrow
\hbm_*(\hessleq(\ba))\rightarrow
\hbm_*(\hesspav(\tau)).
\rightarrow 0
\end{equation}
Combining this with Theorem C establishes that $G_\ba(DR_n) 
\cong \hbm_*(\hessgeo(\tau))$ up to a grading shift, which is stated
in Corollary \ref{cor:ftau}.
This explains 
the appearance of Hessenberg varieties
{ in the aforementioned} bijection between 
$\Res(\tau)$ with $\hesscomb(\tau)$,
which are precisely the fixed point 
sets of $\hesspav(\tau)$ and $\fltaugeo(\tau)\cap \hessgeo(N,h_\tau)$ respectively.

The proof of Theorem C relies on a slight generalization of the duality between the
Hilbert schemes and the stable pairs from \cite{PandharipandeThomas10}.
We generalize the duality to the setting of the flags of stable pairs and show that the
flag stable pairs are exactly affine Springer fibers studied in \cite{OblomkovYun16}. To prove the bundle statement 
and connect this with Hessenberg varieties, 
we use results of Iarrobino and G\"{o}ttsche \cite{iarrobino1977punctual,gottsche1990punctual} on the Hilbert scheme of homogeneous ideals $I\subset \C[x,y]$ with respect to total
degree, which turn out to be isomorphic to the affine Grassmannian 
(or Catalan) version of $\hesspav(\tau)\rightarrow \hessgeo(\tau)$.
To generalize this to the full flagged case, we extend these
arguments to a space called the \emph{parabolic flag Hlibert scheme},
which was introduced and shown to be smooth in \cite{carlsson2017parabolic}.

The paper is divided into six sections. In Section~\ref{sec:geometry} we discuss the geometric results and definitions that we will need for the main construction, 
including the results of \cite{OblomkovYun16}.
In the interest of making
our paper readable to combinatorialists, we have compartmentalized
the necessary algebraic
facts from this section into Proposition \ref{prop:spbasis} of Section \ref{sec:affine-schub-calc},
so that it may be safely skipped.
In Section~\ref{sec:affine-perm-park} we recall combinatorial facts about affine permutations and parking functions, and we give a new description of parking functions
in terms of a bijection of Haglund \cite{haglund2008catalan}, which
turns out to be similar
to the description of the fixed points of regular nilpotent Hessenberg varieties
\cite{insko2013affine,precup2017palindromic}.
Section~\ref{sec:affine-schub-calc} recalls the
algebraic constructions of the affine Schubert polynomials and
nil Hecke algebras \cite{lam2014book}.
In Section~\ref{doubsec}, we state and prove { Theorems A and B}.
Finally, in Section~\ref{sec:geohess} we prove Theorem C and develop necessary geometric tools.

{\bf Acknowledgments} 
The authors would like to thank Thomas Lam, Mark Shimozono,
and S. J. Lee for interesting discussions about affine Schubert calculus.
{ The first author was partially supported
  by NSF DMS-180237.} The authors are grateful to anonymous referees who provided us
many suggestions and corrections. In particular, the Hessenberg paving of the affine Springer fibers (see Theorem C)
and the simplified combinatorial statement (see Corollary A) 
are results of the  multiple exchanges with our brilliant referees.
The second author was partially supported
by NSF CAREER grant DMS-1352398.

\section{Combinatorial notation and preliminaries}

\label{sec:affine-perm-park}

We recall certain combinatorial notations and preliminary statements which will be used in the proof of Theorems A and B.
This includes several different versions of the 
Shuffle Theorem
\cite{haglund2005combinatorial,CarlssonMellit12},
in which the combinatorial objects are described by three different versions of parking functions, 
namely labeled Dyck paths, 
restricted affine permutations, and a third one known as ``schedules'' \cite{haglund2003conjectured,hicks2005thesis,haglund2008catalan,gorsky2014rational}. We explain several known bijections between all three of these objects, in a way that is compatible with certain statistics, such as $\area$, $\dinv$, and reading word order.
We then describe a partition of each set
into groups {called \emph{cars} in parking function language}, 
labeled by usual (non-affine) 
permutations satisfying
a condition that is similar to  { one describing} the fixed points of
Hessenberg varieties 
\cite{insko2013affine,precup2017palindromic,abe2014equivariant}.
We sum up the relations in the diagram below:

% https://q.uiver.app/#q=WzAsMTAsWzAsMiwiXFxidWxsZXQiXSxbMiwyLCJQRihuKSJdLFs0LDIsIlxcYnVsbGV0Il0sWzAsNCwiXFxidWxsZXQiXSxbMiw0LCJcXGJ1bGxldCJdLFs0LDQsIlxcYnVsbGV0Il0sWzAsMCwiXFxidWxsZXQiXSxbMiwwLCJcXGJ1bGxldCJdLFs0LDAsIlxcYnVsbGV0Il0sWzYsMiwiXFxidWxsZXQiXSxbMywwLCIiLDAseyJzdHlsZSI6eyJ0YWlsIjp7Im5hbWUiOiJob29rIiwic2lkZSI6ImJvdHRvbSJ9fX1dLFszLDQsIiIsMix7InN0eWxlIjp7InRhaWwiOnsibmFtZSI6Imhvb2siLCJzaWRlIjoidG9wIn19fV0sWzQsMSwiIiwyLHsic3R5bGUiOnsidGFpbCI6eyJuYW1lIjoiaG9vayIsInNpZGUiOiJib3R0b20ifX19XSxbNSwyLCIiLDIseyJzdHlsZSI6eyJ0YWlsIjp7Im5hbWUiOiJob29rIiwic2lkZSI6ImJvdHRvbSJ9fX1dLFs2LDBdLFs3LDFdLFs4LDJdLFs4LDcsIiIsMix7InN0eWxlIjp7InRhaWwiOnsibmFtZSI6Imhvb2siLCJzaWRlIjoiYm90dG9tIn19fV0sWzIsMSwiIiwwLHsic3R5bGUiOnsidGFpbCI6eyJuYW1lIjoiaG9vayIsInNpZGUiOiJib3R0b20ifX19XSxbNSw0LCIiLDAseyJzdHlsZSI6eyJ0YWlsIjp7Im5hbWUiOiJob29rIiwic2lkZSI6ImJvdHRvbSJ9fX1dLFs4LDldLFs1LDldLFsyLDldLFs2LDcsIiIsMix7InN0eWxlIjp7InRhaWwiOnsibmFtZSI6Imhvb2siLCJzaWRlIjoidG9wIn19fV0sWzAsMSwiIiwwLHsic3R5bGUiOnsidGFpbCI6eyJuYW1lIjoiaG9vayIsInNpZGUiOiJ0b3AifX19XSxbNyw5XSxbNCw5XV0=
\[\begin{tikzcd}[cramped]
	\lRes{\mu}         && \Res(n,n+1)        && \Res(\tau) \\
	\\
	 \lpfl{\mu}(n)              && {\pfl(n)}       && \cars(\tau) && \mathbb{Z}^{2} \\
	\\
	\rschedpfl{\mu}(n)  && { \schedpfl(n)} && \scheds(\tau)
	\arrow[hook, from=1-1, to=1-3]
	\arrow[from=1-1, to=3-1]
	\arrow[from=1-3,"{\resizebox{1.3cm}{0.9ex}{$\sim$}}"',sloped,"\rotatebox{90}{\(\cA_{n+1}\circ \mathrm{inv}\)}",to=3-3]
	%\arrow[from=1-3, to=3-7]
	\arrow[hook', from=1-5, to=1-3]
	\arrow[from=1-5,  to=3-5, "{\resizebox{1.3cm}{0.9ex}{$\sim$}}"',sloped,"\rotatebox{90}{\(\alpha\)}" ]
	\arrow[from=1-5, to=3-7, "{\ind(
          \tau)}\times
       \dinv"]
	\arrow[hook, from=3-1, to=3-3]
	\arrow[hook', from=3-5, to=3-3]
	\arrow[from=3-5, to=3-7,"{\area} \times \dinv"]
	\arrow[from=5-1, to=3-1]
	\arrow[hook, from=5-1, to=5-3]
	\arrow[ from=5-3, to=3-3, "{\resizebox{1.3cm}{0.9ex}{$\sim$}}",sloped,"\rotatebox{270}{\(\varphi\)}"']
	%\arrow[from=5-3, to=3-7]
	\arrow[ from=5-5, to=3-5, "{\resizebox{1.3cm}{0.9ex}{$\sim$}}",sloped,"\rotatebox{270}{\(\beta\)}"']
	\arrow[from=5-5, to=3-7,"{\maj(\tau)\times |\bullet|}"']
	\arrow[hook', from=5-5, to=5-3]
\end{tikzcd}\]

The sets in the diagram and statistics are defined below. The bottom arrows of the diagram were studied by Haglund  \cite{haglund2008catalan} { who} defined \(\sched(\tau)\subset \mathbb{Z}^n\). In particular, as explained in this section, 
{this}
provides a \((q,t)\)-version of the combinatorial formula
\begin{equation}\label{eq:simple-f}
  (n+1)^{n-1}=\sum_{\tau \in S_n}\prod_{i=1}^n\sched_i(\tau).
\end{equation}
The geometric part of this paper yields a geometric interpretation of the last formula as
partition of the affine Springer fiber into affine bundles over Hessenberg type varieties.

The last column of the diagram connects the shuffle version of the sets attached to a composition \(\mu\). In Theorem $1''$ from this
section we show how one recovers the \(q,t\)-characters of \(S_n\)-isotypical pieces of { \(DR_n\)} from the shuffle version of \eqref{eq:simple-f}.

Most of combinatorial construction from this section have an extension to a rational case: \((n,n+1)\) is replaced by
\((n,m)\) where \(m\) is coprime with \(n\). Whenever it is possible we provide the constructions in the more general context.
In last subsection we provide another realization of parking functions in terms of flags of lattices, this description
used in the later sections.

\subsection{Combinatorial notations}

\label{sec:combnoations}

The set of \emph{partitions} $\Par(n)$
is the collection of sequences 
$\lambda=(\lambda_1,...,\lambda_l)$ 
of positive integers 
{ summing to $n$}
which are 
non-strictly sorted in decreasing order. 
To each partition we have its \emph{Ferrers diagram}
or \emph{Young diagram}
\begin{equation}
\label{eq:defferrers}    
\ferrers(\lambda)=\left\{(r,s)\in \mathbb{Z}_{\geq 0}^2:
r<\lambda_{s+1}\right\},
\end{equation}
and we will sometimes simply write $(r,s)\in \lambda$.
This follows the conventions of \cite{Haiman02},
except that we will draw the Young diagrams in English notation,
which is reflected over the $x$-axis so that they are aligned with the coarea region of parking functions.
%We have the arm and leg length given by
%%\begin{equation}
%\label{eq:armleg}    
%arm_{\lambda}(s)=
%\end{equation}
The \emph{transposed partition} 
$\lambda'$ is the one for which 
$(r,s)\in \lambda' \Leftrightarrow (s,r)\in \lambda$.

By a \emph{composition} of $n$, we will mean 
a finite list 
of positive integers 
$\mu=(\mu_1,...,\mu_l)$
such that
$|\mu|=\mu_1+\cdots+\mu_l=n$, but which is not necessarily sorted.
%in other words a possibly unsorted partition.
An \emph{ordered set partition} will
mean an ordered list of nonempty subsets 
$(B_1|\cdots |B_l)$, such that
\[B_1 \sqcup \cdots \sqcup B_l=[n],\quad [n]=\{1,...,n\}.\] 
Given a composition $\mu$,
we will denote the set of all ordered set partitions with $|B_i|=\mu_i$ by $\osp(\mu)$.

Given a composition $\mu=(\mu_1,...,\mu_l)$
let
\[S_\mu=S_{\mu_1}\times \cdots \times S_{\mu_l} \subset S_n\]
denote the Young subgroup of the symmetric group
$S_n$.
The minimal and maximal length elements of the left coset space $S_\mu\backslash S_n$ are known as $\mu$-shuffles and reverse
$\mu$-shuffles respectively. The sets of shuffles and reverse shuffles, denoted 
$\lshuffs{\mu}$ and $\rshuffs{\mu}$ respectively, 
consist of permutations
$\sigma\in S_n$ whose entries are sorted in the blocks of
$\mu$ in increasing (resp. decreasing) order.
For instance, for $\mu=(2,3)$ we would have
\[\rshuffs{\mu}=\left\{(2, 1, 5, 4, 3), 
(2, 5, 1, 4, 3), (2, 5, 4, 1, 3), (2, 5, 4, 3, 1),
    (5, 2, 1, 4, 3),\right.\]
    \[\left.(5, 2, 4, 1, 3), (5, 2, 4, 3, 1), (5, 4, 2, 1, 3),
(5, 4, 2, 3, 1), (5, 4, 3, 2, 1)\right\}\]
which consists of all permutations for which both
$\{1,2\}$ and $\{3,4,5\}$ appear in reverse order, i.e.
$\sigma^{-1}(1)>\sigma^{-1}(2)$,
and $\sigma^{-1}(3)>
\sigma^{-1}(4)>
\sigma^{-1}(5)$.  The elements of $
\lshuffs{\mu},\rshuffs{\mu}$ are each in bijection with $\osp(\mu)$. The element of
$\osp(\mu)$ consisting of the blocks 
$\{1,...,\mu_1\},\{\mu_1+1,...,\mu_1+\mu_2\},...$ will be denoted $\minosp(\mu)$.
We will some times denote the non-reversed shuffles by $\shuffs{\mu}=\lshuffs{\mu}$.

Given a permutation $\sigma \in S_n$, we define the \emph{inversion table} by
$\invt(\sigma)=(a_1,...,a_n)$
where
\begin{equation}
\label{eq:invtdef}    
a_{\sigma_j}=\#\left\{1\leq i\leq j-1:\sigma_i>\sigma_j\right\}.
\end{equation}
The
\emph{major index table} is given by
$\majt(\sigma)=(a_1,...,a_n)$,
where
\begin{equation}
\label{eq:majtdef}
 a_{\sigma_i}=\#\left\{i \leq j \leq n-1: \sigma_{j}>\sigma_{j+1}\right\}
\end{equation}
For instance, for
$\sigma=(2, 1, 3, 6, 5, 4)$
we would have
$\invt(\sigma)=(1, 0, 0, 2, 1, 0)$
and 
$\majt(\sigma)=(2, 3, 2, 0, 1, 2)$.
We define $\maj(\tau)=|\majt(\tau)|$, and also 
define the \emph{comajor index}
\[\comaj(\tau)=n(n-1)/2-\maj(\tau).\]

Compositions of the form $\ba=\invt(\sigma)$ and 
$\ba=\majt(\sigma)$ for some $\sigma$ are called 
\emph{staircase} and 
\emph{descent} compositions respectively. 
 { We can recover \(\sigma\) from \(\majt(\tau)=\ba\) as follows. We set \(\mu_{m-i+1}\) to be the number of occurrences
  of \(i\) in \(\ba\), here \(m=\mathrm{max}(\ba)\). Then \(\sigma\in \lshuffs{\mu} \) and the blocks corresponding to the part \(\mu_{m-i-1}\)
consist of indices \(j\) such that \(a_j=i\).}  
The staircase and descent compositions are the exponents in the Artin and Garsia-Stanton descent monomials defined below.

\subsection{Coinvariant algebras}

Given a number $n$, we will use bold letters to denote $n$-tuples. For instance, a set of variables will be denote by
$\bx=(x_1,...,x_n)$, while the exponents 
may be written $\bk=(k_1,...,k_n)\in \mathbb{Z}_{\geq 0}^n$
so that $\bx^\bk=x_1^{k_1}\cdots x_n^{k_n}$.
If $\sigma\in S_n$ is a permutation, we will denote the result of permuting the indices by $\bx_{\sigma}=(x_{\sigma_1},...,x_{\sigma_n})$.

\begin{defn}
The \emph{coinvariant} algebra in $n$ variables is defined by
\begin{equation}
\label{eq:coinvdef}    
R_n=R_n(\bx)=\C[x_1,...,x_n]/\coinvid(\bx)
\end{equation}
where $\coinvid(\bx)=(e_1(\bx),...,e_n(\bx))$ is the ideal generated by the 
(elementary) symmetric polynomials with vanishing constant term.
\end{defn}

There are two well-known monomial bases of $R_n$ indexed by permutations, 
called the \emph{Artin basis} $\{f_\sigma(\bx)\}$, and \emph{Garsia-Stanton descent basis} $\{g_\sigma(\bx)\}$, where 
\begin{equation}
    \label{eq:artinbasis}
    f_\sigma(\bx)=\bx^{\invt(\sigma)},\quad g_\sigma(\bx)=\bx^{\majt(\sigma)}.
\end{equation}
Written another way, we have
\begin{equation}
\label{eq:descmon}    
g_\sigma(\bx)=
\prod_{i:\sigma_{i}>\sigma_{i+1}}
x_{\sigma_1}\cdots x_{\sigma_i},
\end{equation}
whereas the Artin basis can be described as the sub-staircase monomials
\begin{equation}
\label{eq:substaircase}    
\{f_\sigma(\bx)\}=
\left\{\bx^{\ba}: a_i\leq n-i\right\}.
\end{equation}
Note that this different from other notations, which often use
$a_i\leq i-1$.

The Artin monomials are in fact the \emph{standard monomials} of $\coinvid(\bx)$ with respect to the lexicographic order on $\ba$. The descent monomials are also standard monomials, but for a different order, called the \emph{descent order}:
 \begin{defn}
    \label{desdef}
  The descent order on compositions is defined by $\ba\leq_{des} \bb$ if
    \begin{enumerate}
    \item $\sort_{>}(\ba) <_{lex} \sort_{>}(\bb)$ or
    \item $\sort_{>}(\ba)=\sort_{>}(\bb)$ and $\ba \leq_{lex} \bb$ 
    \end{enumerate}
      where $\sort(\ba)$ sorts a composition in decreasing order to produce
      a partition.
    \end{defn}
    For instance, for $n=2$, we would have
\[(0,0)<(0,1)<(1,0)<(1,1)<(0,2)<(2,0)<\cdots\]
noting that it is possible to have $\ba <_{des} \bb$, but $|\ba|>\bb|$.
The descent order does not satisfy the multiplicativity property required of monomial orders 
in the sense of Gr\"{o}bner bases 
\cite{cox2008ideals}. However,
the following proposition shows that the descent monomials are still standard monomials, and in fact gives
an algorithm for their reduction:
\begin{prop} (Allen \cite{allen1994descent})
  \label{prop:allen}
  For any composition $\ba$, that is not a descent partition (i.e.
  there is no \(\sigma\in S_n\) such that
  \(\ba=\majt(\sigma)\)) there exists a non-empty partition $\mu$ and a descent composition
  $\bc=\majt(\tau)$, \(\tau\in S_n\) such that
  \[\by^{\bc} m_\mu(\by)=\by^\ba+\sum_{\bb<_{des} \ba} a_{\bb} \by^{\bb},\]
  where $m_\mu(\by)$ is the monomial symmetric function.
  %Furthermore, $\mu$ is the empty partition
  %if and only if $\ba$ is a descent composition, that is \(\by^\ba=\by^{\majt(\sigma)}\) for some \(\sigma\in S_n\).
\end{prop}

We now define the two-variable version of coinvariant algebras:
\begin{defn}
The double (or diagonal) 
coinvariant algebra is defined as
\begin{equation}
    \label{eq:diagcoinv}
    DR_n=\C[\bx,\by]/
    \coinvid(\bx,\by),\,\,\, 
    \coinvid(\bx,\by)=
\lang \sum_{k} x_k^i y_k^j:
  (i,j)\neq (0,0)\rang.
\end{equation}    
\end{defn}
Since $\coinvid(\bx,\by)$ is homogeneous with respect to the bigrading
$\mdeg(x_i)=(1,0)$ and $\mdeg(y_i)=(0,1)$, we may take its graded dimension which is a polynomial in two variables
\[\grdim_{q,t} DR_n =\sum_{i,j} 
%a_{i,j} 
\dim(DR_n^{(i,j)})
q^it^j.\]
Here $DR_n^{(i,j)}$ is the homogeneous component of $DR_n$ with bigrading $(i,j)$,
so that $q,t$ correspond to the gradings in the $x$ and $y$-variables, respectively.

Since $\coinvid(\bx,\by)$ is preserved by the diagonal action of the symmetric group
\[(\sigma\cdot f)(x_1,...,x_n,y_1,...,y_n)=
f(x_{\sigma_1},...,x_{\sigma_n},y_{\sigma_1},...,y_{\sigma_n}), \]
we have an action of $S_n$ on $DR_n$, and in fact on each homogeneous component $DR_n^{(i,j)}$. 
In \cite{Haiman02}, Haiman proved that the character is given explicitly by the formula
\begin{equation}
\label{eq:drntonabla}
\grdim_{q,t} DR_n^{S_\mu}=
\nabla e_n\big|_{m_\mu}
%   \sum_{i,j} q^it^j \dim (DR_n^{(i,j)})^{S_\mu}.
\end{equation}
where $DR_n^{S_\mu}$ is the space of invariants under the Young subgroup, the vertical bar means we are extracting the coefficient of $m_\mu$,
and $\nabla$ is the famous nabla operator of \cite{garsia1996remarkable}, which acts diagonally in the modified Macdonald basis.
%and $\omega$ is the Weyl involution.

On the other hand, the  \emph{Shuffle Theorem}
\cite{haglund2005combinatorial,haglund2012compositional,CarlssonMellit12} provides a combinatorial formula for the right side of 
\eqref{eq:drntonabla} in terms of  parking function statistics,
\(\area\) and \(\dinv\),
these statistics are defined in Definition~\ref{def:dinv} in Section~\ref{sec:ratpark}.
 There is also an
equivalent \emph{super} version, which is an expression for
$\omega \nabla e_n$, where $\omega$ is the Weyl involution.
Combining the Shuffle Theorem with \cite{Haiman02}
implies a combinatorial formula for the Frobenius character of $DR_n$ expanded in the monomial basis,
whereas the super version determines the 
corresponding expansion of 
the sign twist $DR_n \otimes \signrep$.
%
%
%In this paper, we will be more interested in the 
%which encodes the similar invariants but with the twist of $DR_n$ by the sign representation $DR_n\otimes \signrep$.
%Both versions are equivalent and encode the multiplicities of the irreducible representations $\chi_\mu$ of $S_n$.
%Though we will not use this fact explicitly, the reason the sign-twisted version is more useful has to do with the fact that the $S_n$ action is a version of the Springer action, whose anti-invariants under $S_\mu$ encode the homologies of the corresponding parabolic subgroups \cite{borho1983partial}.

\subsection{Rational parking functions}

\label{sec:ratpark}

We recall the combinatorial objects that appear in the Shuffle
Theorem and some of its variants
%versions of the Shuffle Theorem
\cite{haglund2005combinatorial,haglund2012compositional,CarlssonMellit12}. 
For a reference, see \cite{haglund2008catalan}.

An $(n,m)$-Dyck path $\dyckpath$
is a path in $\mathbb{Z}^2$ consisting of North and East steps from $(0,0)$ to $(m,n)$, which stays entirely above the line $y=(n/m)x$. The area sequence $(a_1,...,a_n)=\areat(\pi)$ is the integer vector with the property that $a_i$ is the length of the $i$-th row between the path and the diagonal, starting from the bottom. The 
co-area sequence $\coareat(\pi)$ 
is determined by
$\areat(\pi)_i+\coareat(\pi)_i=\lfloor (i-1)m/n \rfloor$, the sum being equal to the area sequence of a maximal $(n,m)$-path.

\begin{figure}
\begin{tikzpicture}[scale=0.750000]
\draw[help lines] (0,0) grid (7,4);
\draw[dotted] (0,0)--(7,4);
\draw[black,semithick] (0,0)--(0,4)--(7,4)--(7,0)--(0,0);
\draw[black,semithick] (0,0)--(0,1)--(0,2)--(1,2)--(2,2)--(2,3)--(3,3)--(4,3)--(4,4)--(5,4)--(6,4)--(7,4);
\node[scale=1.500000] (2) at (0.500000,0.500000) {2};
\node[scale=1.500000] (4) at (0.500000,1.500000) {4};
\node[scale=1.500000] (1) at (2.500000,2.500000) {1};
\node[scale=1.500000] (3) at (4.500000,3.500000) {3};
\end{tikzpicture}
\caption{A rational parking function 
$P=(\pi,\sigma)\in \pfl(4,7)$.
Then we have $\areat(\pi)=(0,1,1,1)$,
$\coareat(\pi)=(0,0,2,4)$.}
%\label{dyckpic}
    \label{fig:pf47}
\end{figure}

An $(n,m)$-parking function $P=(\dyckpath,\sig)$ 
consists of the pair of
an $(n,m)$-Dyck path $\pi$ together with the labeling
of the rows by a permutation $\sigma \in S_n$,
that are decreasing along each vertical wall.
The set of $(n,m)$-parking functions is denoted by
$\pfl(n,m)$.
An example is shown in Figure \ref{fig:pf47}.

We will write $\pfl(n)=\pfl(n,n)$ in the special case of $m=n$.
In this case, the usual $\dinv$ statistic is given by
\begin{defn}\label{def:dinv}
    Let $P\in \pfl(n)$, and let $\ba=\areat(P)$. Then $\dinv_i(P)$ is equal to the number of pairs $(i,j)$ with $1\leq i<j\leq n$, which satisfy
    \begin{equation}
\label{eq:dinvdef}
a_i=a_j \mbox{ and } \sigma_i<\sigma_j,\mbox{ or }
a_i=a_j+1 \mbox{ and } \sigma_i>\sigma_j.
\end{equation}
We define \(\dinv(P)\) as a sum of the entries of the vector
\[\mathbf{dinv}(P)=(\dinv_1(P),\dots,\dinv_n(P)).\]
\end{defn}
%Both the area and coarea sequences agree as well.
See Figure \ref{fig:bigpf} for an example.
In the square case, the integer vector 
$\areat(P)_{\sigma^{-1}}$
is always a descent composition, 
which will be denoted $\descpf(P)$.

\begin{figure}
    \centering
\begin{tikzpicture}[scale=0.750000]
\draw[help lines] (0,0) grid (7,7);
\draw[dotted] (0,0)--(7,7);
\draw[black,semithick] (0,0)--(0,7)--(7,7)--(7,0)--(0,0);
\draw[black,semithick] (0,0)--(0,1)--(0,2)--(1,2)--(1,3)--(1,4)--(2,4)--(2,5)--(3,5)--(4,5)--(4,6)--(5,6)--(6,6)--(6,7)--(7,7);
\node[scale=1.500000] (6) at (0.500000,0.500000) {6};
\node[scale=1.500000] (7) at (0.500000,1.500000) {7};
\node[scale=1.500000] (2) at (1.500000,2.500000) {2};
\node[scale=1.500000] (3) at (1.500000,3.500000) {3};
\node[scale=1.500000] (5) at (2.500000,4.500000) {5};
\node[scale=1.500000] (1) at (4.500000,5.500000) {1};
\node[scale=1.500000] (4) at (6.500000,6.500000) {4};
\end{tikzpicture}
    \caption{A parking function $P=(\pi,\sigma)\in \pfl(7)$ with $\sigma=(6,7,2,3,5,1,4)$, 
    $\dinv(P)=4$, $\areat(P)=(0, 1, 1, 2, 2, 1, 0)$,
$\coareat(P)=(0, 0, 1, 1, 2, 4, 6)$,
$\readword(P)=(5, 3, 1, 2, 7, 4, 6)$, and 
$\descpf(P)=(1, 1, 2, 0, 2, 0, 1)$, the descent composition $\ba$ whose $i$-th entry $a_i$ is the area in the row containing $\sigma_i$.}
    \label{fig:bigpf}
\end{figure}

If $P\in \pfl(n)$, its \emph{reading word}
$\readword(P)\in S_n$ is the result of reading off the entries in $P$ from upper-right to lower-left, in decreasing order of area.
If $\mu$ is a composition of $n$, we will denote
\begin{equation}
    \label{eq:pflmudef}
    \lpfl{\mu}(n)=\left\{P\in \pfl(n): \readword(P) \in \lshuffs{\mu}\right\},
\end{equation} 
and similarly for $\rpfl{\mu}(n)$ and $\rshuffs{\mu}$.
Then combined with the (highly nontrivial) formula
\eqref{eq:drntonabla},
the signed version of
the (non-compositional) Shuffle Theorem is equivalent to the following.
%\cite{haglund2005combinatorial,haglund2012compositional} states that
%\begin{equation}
%\label{eq:shuffthmsigned}    
%\omega \nabla(e_n)\big|_{m_\mu}=
%\sum_{(\pi,\sigma)\in \pfl^{>}_\mu(n)}
%t^{\area(\pi)} q^{\dinv(\pi,\sigma)}
%\end{equation}
%, stated in terms of coinvariants is
\begin{shuffthm0}[Restated Shuffle Theorem \cite{haglund2005combinatorial,CarlssonMellit12}+\cite{Haiman02}]
\label{thm:shuffpf}
We have
\begin{equation}
       \label{eq:shuffthmpf}
\grdim_{q,t} ( DR_n  \otimes \signrep)^{S_\mu} =
\sum_{(\pi,\sigma)\in \pfl^{>}_\mu(n)}
t^{\area(\pi)} q^{\dinv(\pi,\sigma)}
\end{equation}
where the left hand side is the bigraded dimension of the $S_\mu$-invariants of the twist of $DR_n$ by the sign representation.
\end{shuffthm0}

Let us explain why we called statement above Restated Shuffle Theorem.
There are two separate statements in the literature.
  Statement 1 claims that the  bigraded Frobenius character of \(DR_n\) is
 equal to $\nabla e_n$, where $\nabla$ is a certain operator
  \cite{garsia1996remarkable} on symmetric functions. This was famously proved by Haiman
in \cite{Haiman02}, see formula~\eqref{eq:drntonabla} from above.

  Statement 2 is a combinatorial identity expressing $\nabla e_n$ as a sum over parking functions. Traditionally,
  Statement 2  is called Shuffle Conjecture in celebrated paper of
 Haglund, Haiman, Loehr, Remmel
and Ulyanov
  \cite{HaglundHaimanLoehrRemmelUlyanov05}, and it was first proved in \cite{CarlssonMellit12}.
Our Theorem~\ref{thm:shuffpf}, which we called restated Shuffle conjecture,
implies  the combinatorial formula for the bigraded Frobenius character of \(DR_n\). Thus, combined with Statement 1, this theorem implies
  Statement 2 and vice versa, but strictly speaking it is logically independent from them.

The version of  Theorem~\ref{thm:shuffpf} without the sign twist is given by replacing 
$\rpfl{\mu}(n)$ with $\lpfl{\mu}(n)$.
In terms of symmetric functions,
the polynomial on either side of \eqref{eq:shuffthmpf} is the coefficient of the monomial symmetric function $m_\mu$ in
$\omega \nabla(e_n)$, where $\nabla$ is the nabla operator \cite{garsia1996remarkable}, 
and $\omega$ is the Weyl involution.

Let us also point out that in the case \(\mu=(n)\) the combinatorial
side of the formula \eqref{eq:shuffthmpf} is the sum over the
Dyck paths and this sum is known as \(q,t\)-Catalan number.
We refer to the beautiful book by Haglund \cite{haglund2008catalan}
where many combinatorial aspects of the theory of Catalan numbers are
discussed.

We also explain some special properties of the Catalan case in Example~\ref{ex:schedpfl1} of this section. We also explain how
our monomial basis yields a set on minimal generators
of the ideal generated by the alternating polynomials in
corollary~\ref{cor:min-gen}.

\subsection{Schedules}

\label{sec:sched}

We now describe the ``schedules'' version of Theorem \ref{thm:shuffpf} \cite{hicks2005thesis,haglund2003conjectured}.
For any $\tau\in S_n$, we define the \emph{runs}, 
denoted $\runs(\tau)=(\run_1(\tau),...,\run_k(\tau))$
as the maximal consecutive increasing subsequences of
$\tau$. By convention, if there are $k$ runs, we define $\run_{k+1}(\tau)$ to consist of a 
single run containing only the number zero, 
thinking of $\tau_{n+1}=0$.
For instance, for
$\tau=(3,5,1,2,7,4,6)$ we would have
$k=3$ and
\[\run_1(\tau)=(3,5),\ \ 
\run_2(\tau)=(1,2,7), \ \ 
\run_3(\tau)=(4,6),\ \ 
\run_4(\tau)=(0).\]

If $\tau_i$ is in the $j$-th run of $\tau$, then we define $\sched_i(\tau)$ to be the number of elements of the $\run_j(\tau)$ that are greater than $\tau_i$, together with the number of elements of $\run_{j+1}(\tau)$ which are less than $\tau_i$.
Then \emph{schedule} of $\tau$ is the sequence 
$\sched(\tau)=(\sched_1(\tau),...,\sched_n(\tau))$.
For instance, for the above choice of $\tau$ we would have
$\sched(\tau)=(3,2,2,1,2,2,1)$. See the discussion preceding Theorem 5.3 of \cite{haglund2008catalan}.
%is the number of 

\begin{defn}
\label{def:schedules}    
For any $\tau$, we define
\begin{equation}
    \label{eq:scheds}
    \scheds(\tau)=\left\{(k_1,...,k_n): 0\leq k_i \leq \sched_i(\tau)-1\right\}
\end{equation}
We then define the schedules version of parking functions
\begin{equation}
\label{eq:schedpfl}
    \schedpfl(n)=
\left\{(\majt(\tau),\bk): \bk_{\tau} \in \scheds(\tau)\right\}.
\end{equation}
noting that the indices of $\bk$ are permuted by $\tau^{-1}$.
\end{defn}
Permuting the indices $\bk$ 
allows us to correctly encode the shuffles.
If $\mu$ is a composition of $n$, 
we will also let
$\lschedpfl{\mu}(n)$ denote the set of those pairs
$(\ba,\bk) \in \schedpfl(n)$ with the property that
whenever $i$ and $i+1$ are in the same block of $\minosp(\mu)$, we have
\begin{equation}
    a_i\geq a_{i+1}, \ \ 
    a_i=a_{i+1}\Rightarrow k_i \leq k_{i+1}.
\end{equation}
We define $\rschedpfl{\mu}(n)$ similarly, but with the conditions
\begin{equation}
\label{eq:scheddecreasing}
    a_i\leq a_{i+1}, \ \ 
    a_i=a_{i+1}\Rightarrow k_i > k_{i+1}.
\end{equation}
noticing the strict inequality in the second.

Then we have
\begin{prop}
\label{prop:schedpf}    
There is a combinatorial bijection 
$\schedpfl(n)\rightarrow \pfl(n)$
whose restriction identifies $\lschedpfl{\mu}(n)$ with $\lpfl{\mu}(n)$ for any $\mu$, and similarly in the reverse direction.
Moreover, if $(\ba,\bk)$ corresponds to
$P=(\pi,\sigma)$, then we have
\begin{equation}
\label{eq:schedstats}    
\ba=\areat(P)_{\sigma^{-1}},\ \ 
\mathbf{dinv}(P)=\bk.
\end{equation}
\end{prop}
\begin{proof}
The bijection is 
given by sending $(\majt(\tau),\bk_\tau)$ to
$\varphi(\bk)$, where $\varphi$ is
the bijection described in the proof of Theorem 5.3 of 
\cite{haglund2008catalan}. We will not define 
this map in detail since we give an equivalent version in Section \ref{sec:hesscomb}, but see Example \ref{ex:cars} below.
\end{proof}

The set $\schedpfl(n)$ is partitioned into bins $\schedpfl(n)=\bigsqcup_\tau \schedpfl(\tau)$
according to the permutation $\tau$
whose major index is $\ba$. By the leftmost equality in \eqref{eq:schedstats}, the parking functions associated to $\schedpfl_\tau$
under the bijection of Proposition \ref{prop:schedpf} are the ones for which
the labels in rows of area $l$ are the runs of $\tau$ in some order, which is denoted $\cars(\tau)$. An example is shown in Figure \ref{fig:cars}.
\begin{figure}
    \centering
    \begin{tikzpicture}[scale=0.5500000]
\draw[help lines] (0,0) grid (5,5);
\draw[dotted] (0,0)--(5,5);
\draw[black,semithick] (0,0)--(0,5)--(5,5)--(5,0)--(0,0);
\draw[black,semithick] (0,0)--(0,1)--(0,2)--(1,2)--(1,3)--(2,3)--(2,4)--(2,5)--(3,5)--(4,5)--(5,5);
\node[scale=1.000000] (4) at (0.500000,0.500000) {4};
\node[scale=1.000000] (5) at (0.500000,1.500000) {5};
\node[scale=1.000000] (2) at (1.500000,2.500000) {2};
\node[scale=1.000000] (1) at (2.500000,3.500000) {1};
\node[scale=1.000000] (3) at (2.500000,4.500000) {3};
\end{tikzpicture}
\begin{tikzpicture}[scale=0.5500000]
\draw[help lines] (0,0) grid (5,5);
\draw[dotted] (0,0)--(5,5);
\draw[black,semithick] (0,0)--(0,5)--(5,5)--(5,0)--(0,0);
\draw[black,semithick] (0,0)--(0,1)--(0,2)--(1,2)--(1,3)--(1,4)--(2,4)--(3,4)--(3,5)--(4,5)--(5,5);
\node[scale=1.000000] (4) at (0.500000,0.500000) {4};
\node[scale=1.000000] (5) at (0.500000,1.500000) {5};
\node[scale=1.000000] (2) at (1.500000,2.500000) {2};
\node[scale=1.000000] (3) at (1.500000,3.500000) {3};
\node[scale=1.000000] (1) at (3.500000,4.500000) {1};
\end{tikzpicture}
\begin{tikzpicture}[scale=0.5500000]
\draw[help lines] (0,0) grid (5,5);
\draw[dotted] (0,0)--(5,5);
\draw[black,semithick] (0,0)--(0,5)--(5,5)--(5,0)--(0,0);
\draw[black,semithick] (0,0)--(0,1)--(0,2)--(1,2)--(1,3)--(2,3)--(2,4)--(2,5)--(3,5)--(4,5)--(5,5);
\node[scale=1.000000] (4) at (0.500000,0.500000) {4};
\node[scale=1.000000] (5) at (0.500000,1.500000) {5};
\node[scale=1.000000] (1) at (1.500000,2.500000) {1};
\node[scale=1.000000] (2) at (2.500000,3.500000) {2};
\node[scale=1.000000] (3) at (2.500000,4.500000) {3};
\end{tikzpicture}
\begin{tikzpicture}[scale=0.550000]
\draw[help lines] (0,0) grid (5,5);
\draw[dotted] (0,0)--(5,5);
\draw[black,semithick] (0,0)--(0,5)--(5,5)--(5,0)--(0,0);
\draw[black,semithick] (0,0)--(0,1)--(0,2)--(1,2)--(1,3)--(1,4)--(2,4)--(3,4)--(3,5)--(4,5)--(5,5);
\node[scale=1.000000] (4) at (0.500000,0.500000) {4};
\node[scale=1.000000] (5) at (0.500000,1.500000) {5};
\node[scale=1.000000] (1) at (1.500000,2.500000) {1};
\node[scale=1.000000] (3) at (1.500000,3.500000) {3};
\node[scale=1.000000] (2) at (3.500000,4.500000) {2};
\end{tikzpicture}
    \caption{The elements of $\cars(\tau)$ for $\tau=(3,1,2,5,4)$, as in Figure 4 of \cite{haglund2008catalan}.}
    \label{fig:cars}
\end{figure}

We now have the schedules version of Theorem \ref{thm:shuffpf}:
\begin{shuffthm1}
\label{thm:schedshuff}
Let $\mu$ be a composition of $n$. Then we
have
\begin{equation}
\label{eq:schedshuff}    
\grdim_{q,t} ( DR_n  \otimes \signrep)^{S_\mu}=
%\sum_{\mu \in \Par(n)} 
\sum_{(\ba,\bk)\in
\rschedpfl{\mu}(n)} t^{|\ba|}q^{|\bk|}
\end{equation}
\end{shuffthm1}

When $\mu=(1^n)$, the right hand side is equal to
\begin{equation}
  \label{eq:hagloehr}
\sum_{\tau}
  t^{\maj(\tau)}\prod_{i=1}^n [\sched_i(\tau)]_q,
\end{equation}
where $[k]_q$ is the $q$-number.

\begin{ex}
\label{ex:schedpfl1}
For $n=3$, the elements of $\schedpfl(n)$ are given by 
\[(000,000),(000,010),(000,100),(000,110),\]
\[(000,200),(000,210),(101,000),(010,000),\]
\[(010,100),(011,000),(011,010),(001,000),\]
\[(001,100),(001,001),(001,101),(012,000)\]
Taking the sum as in the right side of
\eqref{eq:schedshuff} gives
\[{q}^{3}+{q}^{2}t+q{t}^{2}+{t}^{3}+2\,{q}^{2}+3\,qt+2\,{t}^{2}+2\,q+2\,t
+1,\]
which is the Hilbert series of $DR_3$. 
\end{ex}
\begin{ex}
For the partitions
$\mu=\{(1,1,1),(2,1),(3)\}$,
the sizes of the sorted schedules
$|\rschedpfl{\mu}(3)|$ 
are given by
$16,10,5$ respectively, where taking $\mu=(1,1,1)$ results in no extra condition. Generally, 
$\schedpfl((n))$ is the number of Dyck paths of size $n$ and the graded sum is the $q,t$-Catalan number.
\end{ex}

The following proposition will be used to deduce
Theorem 1$'$ from our monomial basis.
\begin{prop}
 \label{prop:schedswitch}   
 Suppose that $(\ba,\bk)\in \schedpfl(n)$, and either $a_i>a_{i+1}$ or $a_i=a_{i+1}$ and $k_{i}<k_{i+1}$ for $1\leq i \leq n-1$. Then we have that $(\ba_{s_{i}},\bk_{s_i})\in \schedpfl(n)$, where $s_i=t_{i,i+1}$ is the simple transposition.
\end{prop}
\begin{proof}
  We have \(\sigma\in S_n\) such that $\ba=\majt(\sigma)$.
In the first case   $a_i>a_{i+1}$, hence let $\ba',\bk'$ be the result of switching the labels in positions $i,i+1$
in $\ba,\bk$ respectively. Then it is not hard to see that
$\ba'=\majt(s_i \sigma)$, where $s_{i}$ is the simple transposition, so that $\ba'$ is a descend vector.
It can then be checked that
${\sched_{\sigma_j}(s_i{\sigma})\geq
\sched_{\sigma_j}(\sigma)}$ for all $j$, so that
$(\ba',\bk')\in \schedpfl(n)$.

The second case follows from the statement that if $a_i=a_{i+1}$, then we have
${\sched_i(\sigma)=\sched_{i+1}(\sigma)+1}$, so that
$(\ba,\bk')\in \schedpfl(n)$, where $\bk'$ is as above.
\end{proof}

\subsection{Restricted permutations}\label{sec:affine-permutations}
Let $\affperms$ denote the affine permutations, i.e. those
bijections $w:\mathbb{Z}\rightarrow \mathbb{Z}$
satisfying
\[w_i=w_{i-n}+n,\quad w_1+\cdots +w_n= n(n+1)/2.\]
If the second condition is dropped, then $w$ is called
an extended affine permutation, the set of which will be denoted $\extperms$.
Any (extended) affine permutation is determined by its window notation
$w=(w_1,...,w_n)$, since $w$ is determined by its values
on $\{1,...,n\}$.

for any relatively prime $i,j$ which are not congruent modulo $n$, we have a transposition 
$t_{i,j} \in W$ which interchanges the two, noting that $i,j$ are not uniquely determined by $t_{i,j}$. 
For instance, for $n=6$ we would have that
\[t_{3,-7}=(1,2,-7,4,15,6).\]
We have the simple transpositions $s_i=t_{i,i+1}$ for any integer $i$,
which is used to define the affine Bruhat order denoted $\leq_{bru}$ on both 
$\affperms$ and $\extperms$ \cite{Bjrner1996AffinePO}.

 We will make use 
of two extended permutations,
the rotation and translation elements,
given by
\begin{equation}
\label{eq:rotshift}
\rotelt_i=i+1,\quad (\shiftelt_j)_i=i+n\delta_{\langle i\rangle_n,\langle j\rangle_n}
\end{equation}
where $\langle i\rangle_n =\overline{(i-1)}+1$ is the unique element of $\{1,...,n\}$ which is congruent to $i$ modulo $n$.
We define $\min(w)=\min(w_1,...,w_n)$, which is the same as the minimum value of $w$ over all positive numbers $i>0$.
We define the \emph{index} $\indres(w)=1-\min(w)$ to be the smallest number with the property that $w^+=\rho^{\indres(w)} w$ satisfies $w_i^+>0$ whenever $i>0$.

Let $(n,m)$ be relatively prime.
The set of $m$-stable permutations is the subset
\begin{equation}
\label{eq:stable}
\Stab(n,m) = \left\{w \in \tSn : w_{i+m}> w_i \mbox{ for all $i$}\right\}
\end{equation}
The set of $m$-restricted permutations $\Res(n,m)$ is
the subset of affine permutations
whose inverse is $m$-stable. This set is finite and was shown
to have size $m^{n-1}$, and to parametrize the torus fixed points
of the $(n,m)$-affine Springer fiber
\cite{gorsky2014rational,hikita2014affine}.
Intersecting the Schubert varieties with the Springer fiber determines an affine paving
\cite{LusztigSmelt91}, and the dimension of the cell centered at $w\in \Res(n,m)$ is
\begin{equation}
\label{eq:dimcell}    
\dim_m(w)=
\#\left\{(i,j):1\leq i<j \leq n,\ 0 < w^{-1}_i- w^{-1}_j < m\right\}
\end{equation}
We also have the codimension
$\codim_m(w)=
(n-1)(m-1)/2-\dim_m(w)$.

Similarly to \(\lshuffs{\mu},\rshuffs{\mu}\) we
define $\lRes{\mu}(n,m)$ and $\rRes{\mu}(n,m)$ to be 
the set of those restricted permutations 
$w\in \Res(n,m)$ with the property that the elements of $(w_1,...,w_n)$
are in increasing or respectively decreasing 
order along the components of $\minosp(\mu)$.
In other words, they are representatives of 
the right coset $w S_{\mu}$
which are minimal resp. maximal in the Bruhat order.

Following \cite{gorsky2014rational}, 
we have a bijection $\cA_{m}:\Stab(n,m)\rightarrow \pfl(n,m)$,
defined as follows:
for each $j$, there are unique integers $r_j,k_j$ which satisfy
\[w^{-1}_j-\min(w)= r_jm-k_jn,\ \ 
r_j\in \{0,...,n-1\},\] 
%for $$,
which necessarily implies $k_j \geq 0$. Then
$\cA_m(w)$ is defined as the unique parking function
$P=(\pi,\sigma)$ for which $\coareat(P)=\ba_\sigma$, where $\ba$ is
defined by $a_{j}=k_j$.
For instance, the restricted permutation $w=(4,-2,3,5)\in \Res(4,7)$ has the property that $\cA_7(w)$ is the parking function in Figure \ref{fig:pf47}.

The following map connects these objects to the Shuffle Theorem:
\begin{defn}
\label{def:extcomb}
Define $\respf:\Res(n,n+1)\rightarrow \pfl(n)$ by setting $\respf(w)$ to be the image of $\cA_{n+1}(w^{-1})$ under the bijection $\pfl(n,n+1)\rightarrow \pfl(n)$ which removes the final East step.
\end{defn}

\begin{defn}
\label{def:ind}    
Given an affine permutation $w$, we define its \emph{index} by
\begin{equation}
  \label{aa1def}
  \indt(w)= \ba,\ \ 
  a_i= (w^+_i-\langle w^+_i\rangle_n)/n,\ \ 
  w^+=\rotelt^{\indres(w)}w.
  \end{equation}
  %where $\indres(w)=\indt(w)$.
\end{defn}

Then we have the following proposition, which is straightforward to check.
\begin{prop}
\label{prop:extstats}    
We have that $\majt(\respf(w))=\indt(w)$, and
also $\dinv(\respf(w))=\codim_{n+1}(w)$.
Additionally,  \(\alpha\)
carries $\lRes{\mu}(n,n+1)$
into parking functions whose reading word is a reverse
$\mu$-shuffle, and similarly for the opposite orders. 
\end{prop}
\begin{proof}
{
The statement about the $\dinv$ statistic can be found in \cite{gorsky2014rational}. For the others, we apply the following characterization of the inverse bijection
$\alpha^{-1}$:
Suppose that $w\in \Res(n,n+1)$
and that $P=\alpha(w)$ has coarea
sequence $\ba$ and row labels 
$\sigma$. Then 
$w=\rho^{k} w'$ where
$w'_{\sig_i}=(i-1)(n+1)-n a_i$
and $k$ is the unique shift which forces
$w$ to be an affine permutation. 
The other statements follow easily from this description.}
\end{proof}

We then obtain a third 
version of Theorem \ref{thm:shuffpf}:
\begin{shuffthm2}
We have
\begin{equation}
    \label{eq:shuffres}
\grdim_{q,t} ( DR_n  \otimes \signrep)^{S_\mu}=
\sum_{w\in \lRes{\mu}(n,n+1)}
t^{\indres(w)}
q^{\codim_{n+1}(w)}
\end{equation}
\end{shuffthm2}

Now assume that $m=n+1$ and let us describe the partitioning of 
$\pfl(n)$ into $\cars(\tau)$ in terms of 
restricted permutations. Let
\begin{equation}
    \label{eq:restau}
    \Res(\tau)=\left\{w\in \Res(n,n+1):
    \indt(w)=\majt(\tau)\right\}
\end{equation}
Then the bijections from the previous sections restrict to give
\begin{equation}
\label{eq:bijrescarssched}
\Res(\tau)\xrightarrow{\respf}
\cars(\tau)
\xleftarrow{\ressched} \scheds(\tau)
\end{equation}

In the case of the corresponding
Grassmannian restricted affine permutations, 
we find that $\Res_{(n)}^{<}(\tau)=\Res^<_{(n)}(n,n+1) \cap \Res(\tau)$ 
is nonempty if and only if $\majt(\tau)$ is sorted in increasing
order, which happens when $\tau$ is the 
unique permutation satisfying
%%\begin{equation}
%\label{eq:taufromnu}    
$\tau=(B_k,...,B_1)$ where $(B_1,...,B_k)=\Pi(\mu)$ is the minimal ordered set composition associated to some composition $\mu$
as in Section \ref{sec:combnoations}.
In this case, $\majt(\tau)$ is the level set composition of $\mu$
written in reverse order, which consists of $\mu_k$ zeroes, followed by
$\mu_{k-1}$ ones, ending with $\mu_1$ entries equal to $(k-1)$.
We will denote this set by $\Res^<(\mu)$.
% \end{equation}

\subsection{Lattice description of parking functions}\label{sec:latts}

We have another map from $\Res(n,n+1)$ to parking functions,
which will be used in the geometric discussion
in Section \ref{sec:geohess}.

Let $\Gamma=\Gamma_{n,m}\subset \mathbb{Z}_{\geq 0}$ be the semigroup generated by relatively prime numbers \(n\),\(m\). Let $\latl(n,m)$ denote the set of all ideals 
\begin{equation}
\label{eq:latdp}
\latl^d(n,m)=\left\{E\subset \Gamma_{n,m}:
\ZZ_{\ge  \conduc} \subset E,\ 
|\Gamma_{n,m}-E|=d\right\}
\end{equation}
where %$E_{n,m}=\{k:k\geq \conduc\}$, and
$\conduc=(n-1)(m-1)$ is the \emph{conductor}.
Then $\latl(n,m)$ is in bijection with the 
set of $(n,m)$-rational Dyck paths
\cite{gorsky2011compactified}. 
%the lattice
We have a family of maps 
$\latcomb_i:\Res(n,m)\rightarrow 
\latl(n,m) =\bigcup_{d} \latl^d(n,m)$ given by
\begin{equation}
\label{eq:latmap}
\latcomb_i(w)=\mathbb{Z}-\left\{\conduc-w^+_j:j>i\right\}
\end{equation}
The initial map $\latcomb_0$ determines
a bijection $\Res^<_{(n)}(n,m)\leftrightarrow \latl(n,m)$.

For the flag version, we define
$\flatl(n,m)$ to be the collection of flags
$\latflag=(E_0 \supset E_{1} \supset \cdots )$
of $\Gamma_{n,m}$-submodules of $\mathbb{Z}_{\geq 0}$
satisfying
\begin{equation}
\label{eq:flatl}
E^0 \in \latl^d(n,m),\ \  
|E_{i}-E_{i+1}|=1,\ \ 
E_{i+n}=\eps^n E_i
\end{equation}
where $\eps$ is the translation operator
$k\mapsto k+1$. We have a bijection
\[\latcomb :\Res(n,m)\rightarrow \flatl(n,m)\]
given by $\latcomb(w)=\latflag$ where 
$E_i=\latcomb_{-i}(w)$.
%\begin{equation}
%    \label{eq:latmap}
%L_i=\mathbb{Z}-\left\{\conduc-w^+_j:j> i\right\}
%\end{equation}
\begin{ex}
\label{eq:latpf}    
If $w=(4,-2,3,5) \in \Res(4,7)$ is the restricted permutation
corresponding to the parking function in Figure \ref{fig:pf47}, then $\latcomb(w)$ is given by
%\[(\langle 14,15,16\rangle \supset
%\langle 15,16,18,21\rangle,\supset 
%\langle 15,18,20,21\rangle \supset
%\langle 15,18,20\rangle),\]
\[(\eps^{14},\eps^{15},\eps^{16},\eps^{21}) \supset
(\eps^{15},\eps^{16},\eps^{18},\eps^{21})\supset 
(\eps^{15},\eps^{18},\eps^{20},\eps^{21}) \]
\[\supset (\eps^{15},\eps^{18},\eps^{20},\eps^{25}) \supset
(\eps^{18},\eps^{19},\eps^{20},\eps^{25})\supset \cdots\]
where $\eps^{k}$ is the generator resulting from applying $\eps^k$ to $0$, and the above lattices are written in terms of their generators over $\eps^n$. Then $E_0$ determines a Dyck path, whose inner corners are the generators, as shown in the following picture:

\begin{center}
\begin{tikzpicture}[scale=0.7500000]
\draw[dotted] (0,0)--(7,4);
\draw[gray,semithick] (0,0)--(0,4)--(7,4)--(7,0)--(0,0);
\draw[black,semithick] (0,0)--(0,1)--(0,2)--(1,2)--(2,2)--(2,3)--(3,3)--(4,3)--(4,4)--(5,4)--(6,4)--(7,4);
\node[scale=1.000000] (21) at (0.500000,0.500000) {21};
\node[scale=1.000000] (25) at (1.500000,0.500000) {25};
\node[scale=1.000000] (29) at (2.500000,0.500000) {29};
\node[scale=1.000000] (33) at (3.500000,0.500000) {33};
\node[scale=1.000000] (37) at (4.500000,0.500000) {37};
\node[scale=1.000000] (41) at (5.500000,0.500000) {41};
\node[scale=1.000000] (45) at (6.500000,0.500000) {45};
\node[scale=1.000000] (14) at (0.500000,1.500000) {14};
\node[scale=1.000000] (18) at (1.500000,1.500000) {18};
\node[scale=1.000000] (22) at (2.500000,1.500000) {22};
\node[scale=1.000000] (26) at (3.500000,1.500000) {26};
\node[scale=1.000000] (30) at (4.500000,1.500000) {30};
\node[scale=1.000000] (34) at (5.500000,1.500000) {34};
\node[scale=1.000000] (38) at (6.500000,1.500000) {38};
\node[scale=1.000000] (7) at (0.500000,2.500000) {7};
\node[scale=1.000000] (11) at (1.500000,2.500000) {11};
\node[scale=1.000000] (15) at (2.500000,2.500000) {15};
\node[scale=1.000000] (19) at (3.500000,2.500000) {19};
\node[scale=1.000000] (23) at (4.500000,2.500000) {23};
\node[scale=1.000000] (27) at (5.500000,2.500000) {27};
\node[scale=1.000000] (31) at (6.500000,2.500000) {31};
\node[scale=1.000000] (0) at (0.500000,3.500000) {0};
\node[scale=1.000000] (4) at (1.500000,3.500000) {4};
\node[scale=1.000000] (8) at (2.500000,3.500000) {8};
\node[scale=1.000000] (12) at (3.500000,3.500000) {12};
\node[scale=1.000000] (16) at (4.500000,3.500000) {16};
\node[scale=1.000000] (20) at (5.500000,3.500000) {20};
\node[scale=1.000000] (24) at (6.500000,3.500000) {24};
\end{tikzpicture}
\end{center}
The number in the box which is $i$ steps down and $j$ steps to the right from the upper left is given by $7i+4j$.
%A box whose label is in $E_{4,7}$
%is one that satisfies $7i+4j\geq f=18$, which can be seen to be all but the ones entirely above the main diagonal.
The rest of the parking function can be determined by labeling the generator that is removed with the numbers $\{1,2,3,4\}$ in decreasing order. 
\end{ex}

\section{Hessenberg  varieties and related combinatorics}
In this section we recall the basic definitions from the theory of the   Hessenberg varieties
\( \hessgeo({  S},h)\) in Section~\ref{sec:hesscohomology}. Most relevant case for us is when
\(S\) is equal to a Jordan block of maximal size. In this case we recall an affine paving for this variety
as well as some properties of the cohomology ring.

In Section~\ref{desdef} we establish an isomorphism between the set \(\Res(\tau)\) and some particular
subset \(\hesscomb(\tau)\) of the torus fixed points in \(\hessgeo(N,h_\tau)\), \(\tau\in S_n \). In the
subsequent section \ref{sec:geohess} we reveal that \(\hesscomb(\tau)\) is the torus fixed
locus of the some natural subvariety \(\hessgeo(\tau)\subset \hessgeo(N,h_\tau)\). The results of
Section~\ref{desdef} provide a natural basis for \(H^*(\hessgeo(\tau))\). The combinatorics of
this section is related to the sets in the bottom of the diagram
at the start of Section \ref{sec:affine-perm-park}.

\subsection{Cohomology of the regular nilpotent Hessenberg variety}

\label{sec:hesscohomology}

We recall some facts about the regular nilpotent Hessenberg variety 
\cite{demari1988eulerian,demari1992hessenberg,abe2014equivariant,tymoczko2007paving,abe2020survey}.

Let $B\subset GL_n(\C)$ be the Borel subgroup of 
upper triangular matrices, and let
\begin{equation}
\label{eq:defflag}    
\gFl_n=G/B=\left\{(V_1\subset \cdots \subset V_n) : V_i\subset\C^n,
\dim(V_i)=i\right\}
\end{equation}
be the usual complex flag variety. 
Let $\mathcal{L}_i=V_{i}/V_{i-1}$ denote the tautological line bundles for $1\leq i \leq n$.
Let $X^{\circ}_\sigma=B \sigma B$ be the Schubert cell, where we are identifying $\sigma$ with its permutation matrix
$(P_\sigma)_{i,j}=1$ if $i=\sigma_j$ or zero otherwise.

For instance, the Schubert cell corresponding to 
$\sigma=(5,4,1,3,2)$ is given by
 \begin{equation} 
 \label{eq:schubex}
 A_\sigma=
 \left( \begin {array}{ccccc} a_{{1,5}}&a_{{1,4}}&1&0&0
\\ a_{{2,5}}&a_{{2,4}}&0&a_{{2,3}}&1
\\ a_{{3,5}}&a_{{3,4}}&0&1&0\\a_{{
4,5}}&1&0&0&0\\ 1&0&0&0&0\end {array} \right) 
\end{equation}
The corresponding flag $V_i$ is the span of columns
$1$ through $i$.
The coordinates are labeled so that $A_\sigma$ is the result of eliminating a subset of the entries of $P_\sigma U$, where $U\in B$ satisfies $U_{i,i}=1$ and $U_{i,j}=a_{i,j}$ for $i<j$. The indeterminants 
consist of variables $a_{\sigma_j,\sigma_i}$ for 
which $i<j$ and $\sigma_i>\sigma_j$.

If $\mu$ is a composition of $n$, we let $P_\mu\subset G$ 
denote the corresponding parabolic subgroup, whose elements
satisfy $A_{i,j}=0$ when $i$ is in a block of $\mu$
which is after the block containing $j$.
The corresponding Schubert cells $A_{\sigma,\mu}$ 
are labeled by permutations $\sigma$ which are in
reverse order along each block of $\mu$. We have that
$A_{\sigma,\mu}$ is the result of eliminating entries within
$A_\sigma$ so that the columns corresponding to each block
of $\mu$ are in reduced row echelon form.

The \emph{Hessenberg varieties} are certain generalizations of the Springer fiber, introduced in 
\cite{demari1988eulerian,demari1992hessenberg}:
\begin{defn} 
\label{def:hessvariety}
Let \({ {S}}\in \gl(n)\), and let $h:[n]\rightarrow [n]$ be a  \emph{Hessenberg function}, meaning 
that it is weakly increasing and satisfies $h(i)\geq i$ for all $i$.
Then the \emph{Hessenberg variety} \(\hessgeo(S,h)\) is defined by the conditions
\begin{equation}\label{eq:HesDef}
  \hessgeo({  S},h)=
\{(V_1\subset \dots\subset V_n)\in \gFl_n:SV_j\subset V_{h(j)}\}.
\end{equation}
Alternatively, we may write
\begin{equation}  
\label{def:hessvarietya}
\hessgeo({  S},h)=\left\{gB \in G/B: g^{-1}Sg \in H_h\right\},
\end{equation}
where $H_h$
is the collection of all matrices $A$ such that
$A_{i,j}=0$ whenever $i\geq h(j)+1$, 
which satisfies $\mathfrak{b}\subset H_h$ and
$[\mathfrak{b},H_h]\subset H_h$, where
$\mathfrak{b}=\Lie(B)$.
\end{defn}

We also have the parabolic (or partial)
Hessenberg varieties
$\hessgeo_\mu(S,h)\subset G/P_\mu$, in which case $B$ is replaced by
$P_\mu$, which will appear in Section 
\ref{sec:cell-decomp-comp}.
The corresponding Hessenberg functions 
come from Hessenberg functions $\{1,...,k\}\rightarrow \{1,...,k\}$ where $k$ is the length of $\mu$,
which must be constant along the blocks of $\mu$.
One example is the partial version of the \emph{Peterson Hessenberg function} denoted $h_\mu$, which is given by setting
$h_\mu(j)$ to be the maximum value of $i$ which is
in the next block of $\mu$ after the one containing $j$,
or $h_\mu(j)=n$ if $j$ is in the final block.
See \cite{horiguchi2025partial,Precup2017HessenbergVO} for more details.

We are interested in the case of the \emph{regular 
nilpotent Hessenberg varieties}, 
that is \(S=N\) where \(N\) is the size \(n\) 
upper triangular Jordan block matrix
  \begin{equation}\label{eq:Ndef}
    N(e_i)=e_{i-1}
    \mbox{ for $2\leq i \leq n$},
    \ \ 
    N(e_1)=0.
  \end{equation} 
whose dimension is equal to
$\dim(\hf)=\sum_i (h(i)-i)$.
We denote the set of permutations corresponding to points in
$\hessgeo(N,h)$ by
\begin{equation}
    \label{eq:defhesscomb}
    \hesscomb(h)=\left\{\sigma\in S_n: P_\sigma \in \hessgeo(N,h)\right\}.
\end{equation}
We have the following descriptions of the collection of permutation matrices which are elements of 
$\hessgeo(N,h)$:
\begin{prop}
\label{prop:hessperms}    
The set $\hesscomb(h)\subset S_n$ consists of permutations satisfying any of the following equivalent properties:
\begin{enumerate}[a)]
\item \label{item:hesscomba} We have that $(P_\sigma^{-1} N P_{\sigma})_{i,j}=0$ for $1\leq j \leq n$, and 
$h(j)+1 \leq i \leq n$.
%, where $P_\sigma$ is the permutation matrix $(P_\sigma)_{i,j}=1$ if $i=\sigma_j$, zero otherwise.
\item \label{item:hesscombb} For all $1\leq j \leq n$ and $h(j)+1\leq i \leq n$,
we have $\sigma_j\neq \sigma_i+1$.
    \item \label{item:hesscombc} For all $j$ we have $\sigma^{-1}(\sigma_j-1)\leq h(j)$, where by convention $\sigma^{-1}(0)=0$.
\end{enumerate}
\end{prop}
See equation (2.4) of the survey paper \cite{abe2020survey} for another reference on
these descriptions.
It was shown in \cite{tymoczko2007paving} that the intersections 
$X^{\circ}_\sigma \cap \hessgeo(N,h)$ 
as $\sigma$ ranges over $\hesscomb(h)$
determine a paving of $\hessgeo(N,h)$
by affines. This is summarized in our setting
in the following proposition:
\begin{prop}
\label{prop:hesscells}    
There is an affine paving 
\[\hessgeo(N,h)=\bigsqcup X^{\circ}_\sigma \cap \hessgeo(N,h)\] 
as $\sigma$ ranges over $\hesscomb(h)$.
The coordinates are given by the set
\begin{equation}
\label{eq:hesscoords}
\mathcal{A}^h_{\sigma}=\left\{a_{\sigma_i,\sigma_j-1}:i\geq h(j)+1,\ \sigma_i<\sigma_j\right\}
\end{equation}
%where $(i,j)$ ranges over pairs for which 
%$i\geq h(j)+1$ and
%$\sigma_i<\sigma_j$.
%and 
%for $\sigma_i>\sigma_j$ with no constant term.
\end{prop}

\begin{proof}
{A key step is that} the Hessenberg condition leads to equations in terms of the Schubert coordinates as described in
\eqref{eq:schubex} of the form
\begin{equation}
\label{eq:hesscellform}    
a_{\sigma_{i}+1,\sigma_j}=
a_{\sigma_i,\sigma_j-1}+
\sum_{k=1}^{n} a_{\sigma_i,\sigma_k} 
g^{k}_{i,j}, 
\end{equation}
where $g^k_{i,j}$ is a polynomial
with no constant term.
We then solve for 
the remaining variables 
$a_{\sigma_i,\sigma_j}$ appearing in $A_\sigma$ 
in terms of the ones in 
$\mathcal{A}^h_\sigma$.
\end{proof}
Notice that by item \ref{item:hesscombb})
of Proposition \ref{prop:hessperms} and the fact that
$\sigma_i>\sigma_j$, the variables in \eqref{eq:hesscellform}
are among the indeterminants appearing in $A_\sigma$.
Then the dimension of $X^{\circ}_\sigma \cap \hessgeo(N,h)$ is given by $\dim_h(\sigma)$, where
\begin{equation}
\label{eq:dimhess}    
\dim_h(\sigma)=\#\left\{ (j,i): 1\leq j<i\leq n,\ \sigma_j>\sigma_i,\ i\leq h(j)\right\},
\end{equation}
and also write $\codim_h(\sigma)=\dim(h)-\dim_h(\sigma)$.
We therefore have
\begin{cor}
We have that
$H^i(\hessgeo(N,h))=0$ for $i$ odd, and that the dimension
of $H^{2i}(\hessgeo(N,h))$ is the number of cells of dimension $i$, which determine a basis. 
\end{cor}

In \cite{harada2021filtration}, the authors produced another basis in terms monomials in the Chern classes: let $\C[\bx]\rightarrow H^*(\hessgeo(N,h))$ be the 
map which sends $x_i$ to the Chern class of the tautological bundle $c_1(\cL_i)$, and which factors through the map $H^*(\gFl_n)\cong R_n(\bx)$. The authors showed that 
the images of the monomials
\begin{equation}
\label{eq:hessbasis}    
\{x_1^{k_1}\cdots x_n^{k_n}:
k_i\leq h(i)-i\}
\end{equation}
determine a monomial basis of $H^*(\hessgeo(N,h))$.
In particular, the restriction map from the cohomology of the flag variety is surjective.

We now recall some results from the equivariant setting.
Consider the left action of 
\begin{equation}
\label{eq:hesstorus}
    \C^*=\{(z,...,z^n)\} \subset T=(\C^*)^n
\end{equation} 
on $\gFl_n$,
which acts with isolated fixed points, and preserves 
$\hessgeo(N,h)$, as well as the paving.    
%Since the paving is preserved as well,
Then the action is equivariantly formal,
and we have that $H^*_{\C^*}(\hessgeo(N,h))$
is free over $\C[\epsilon]$ and injects 
into the fixed point basis.
It was shown in \cite{abe2014equivariant} 
that the map 
$\C[\bx,\epsilon]\rightarrow H^*_{\C^*}(\hessgeo_h(N))$,
which evaluates a polynomial on the (equivariant) Chern classes as above, is surjective.

These statements have some purely algebraic consequences.
\begin{defn}
\label{def:chisets}
Given a subset $A\subset S_n$, we define a map
$\chi_A: \C[\bx,\epsilon]\rightarrow M_{S_n}=\bigoplus_{\sigma \in S_n} \C[\epsilon] \sigma$ of $\C[\bx,\epsilon]$-modules by
\begin{equation}
\label{eq:chiset}    
\chi_A(f(\bx,\epsilon))=
\sum_{\sigma \in A} f(\sigma_1\epsilon,...,\sigma_n\epsilon,\epsilon) \sigma.
\end{equation}
We denote the kernel of $\chi_A$ by
\begin{equation}
I_A=\left\{f\in \C[\bx,\epsilon]:
g(\sigma_1\epsilon,...,\sigma_n\epsilon,\epsilon)=0:\sigma \in A\right\}
\end{equation}
%be the kernel,
and the image by $M_A\subset M_{S_n}$, which is isomorphic to 
$\C[\bx,\epsilon]/I_A$.
\end{defn}

We can now translate the above facts 
into the following purely algebraic proposition.
\begin{prop}
    \label{prop:alghess}
Let $h$ be a Hessenberg function and let $\chi_h=\chi_{\hesscomb(h)}$ be as in Definition \ref{def:chisets}. Then the set of monomial images
\[\{\chi_h(\bx^\bk) : 0\leq k_i \leq h(i)-i\} \subset \bigoplus_{\sigma \in \hesscomb(h)} \C[\epsilon] \sigma\]
define a $\C[\epsilon]$-basis of $M_h$.
\end{prop}

\begin{proof}

We have that
$\chi_h$ factors through the composition
%of the Chern class map with the
\[\C[\bx,\epsilon]\rightarrow H^*_{\C^*}(\hessgeo(N,h))
\rightarrow \bigoplus_{\sigma \in \hesscomb(h)} \C[\epsilon],\]
the first being the Chern class map, and the second being the localization. Since the first is surjective and the second is injective, we have that
$H^*_{\C^*}(\hessgeo(N,h))\cong M_h$.
By equivariant formality we have that
$M_h$ is free, and that
We have that $H^*(\hessgeo(N,h))\cong M_h/(\epsilon)M_h$.
Being a $\C$-basis of $M_h/(\epsilon) M_h$ is equivalent to 
being a $\C[\epsilon]$-basis of $M_h$.
%, so that the monomials in \eqref{eq:hessbasis} are a 
%$\C[\epsilon]$-basis of $M_h=\C[\bx,\epsilon]/I_h$.
    
\end{proof}

\subsection{Hessenberg paving combinatorics}

\label{sec:hesscomb}

We give another description of these three sets in \eqref{eq:bijrescarssched}, determined as a certain subset of the torus fixed points in a regular nilpotent Hessenberg variety.
The underlying geometry is closely related to the ``Hessenberg paving'' of affine Springer fibers 
\cite{goresky2003purity}.
%, discussed in Section \ref{sec:geometry}.

\begin{defn}
\label{def:fltaucomb} 
Given $\tau\in S_n$, let $\mu=(|\run_1(\tau)|,...,|\run_k(\tau)|)$ be the \emph{run composition} whose elements are the sizes of the runs of $\tau$, and
let $(A_1,...,A_k)=\minosp(\mu)$ be the corresponding ordered set partition.
Let $\fltaucomb(\tau)$ be the set of permutations
$\sigma \in S_n$ such that $\sigma^{-1}_i\in A_j$ for some $j\geq k-i+1$, for all $1\leq i \leq k$.

\end{defn}

In other words, the number 1 in $\sigma$ appears to the right of the final descent in $\tau$, the number 2 appears to the right of the second to last one, etc.
For instance, the first condition says that
$\sigma^{-1}_1\geq n-l+1$
where $l=|\run_k(\tau)|$ is the number of elements in the final run, which is the same as
the multiplicity of zero in 
$\majt(\tau)$. The space whose fixed points are described by $\fltaucomb(\tau)$
will be discussed in Section \ref{sec:geohess}.

\begin{defn}
\label{def:hesstau}    
We define $\hesscomb(\tau) \subset S_n$
by
%to be the permutations $\sigma$ which satisfy
\begin{equation}
\label{eq:hesstau}    
\hesscomb(\tau)=
\hesscomb(\hf_\tau) \cap \fltaucomb(\tau)
\end{equation}
where $\hf_\tau$ is the Hessenberg function
defined by
\begin{equation}
\label{eq:hftau}
\hf_\tau(i)
=\min(i+\sched_i(\tau),n)
=i+\sched_i(\tau)-
\begin{cases}
    1 & \tau_i\in \run_k(\tau)\\
0 & \mbox{otherwise}
\end{cases}
\end{equation}
where as above $\run_k(\tau)$ is the final run of $\tau$.
\end{defn}

\begin{ex}

For $n=3$ we would have
\[h_{(1,2,3)}=(3,3,3),\ \ 
h_{(1,3,2)}=(2,3,3),\ \ 
h_{(2,1,3)}=(2,3,3),\]
\[h_{(2,3,1)}=(3,3,3),\ \ 
h_{(3,1,2)}=(3,3,3),\ \ 
h_{(3,2,1)}=(2,3,3).\]
Generally, the Hessenberg functions we obtain in this way are the ones bounded below by 
the Hessenberg function describing the Peterson variety,
$h(i)=\min(i+1,n)$.    
\end{ex}

We have a second description of $\hesscomb(\tau)$:
\begin{lemma}
\label{lem:eq:hessleq}    
For any $\tau\in S_n$, 
the elements 
of $\hesscomb(\tau)$ are the elements
of $\hesscomb(h_\tau)$ 
which additionally satisfy 
\begin{equation}
\label{eq:lemmahesstau}    
\sigma_i\neq 1 \mbox{ for $1\leq i \leq n-l$},
\end{equation}
where $l$ is the length of the final run of $\tau$.
\end{lemma}

\begin{proof}
Substituting $k=\sigma_i$
and setting $h=h_\tau$, we
can rewrite \eqref{eq:hesstau} as
    \begin{equation}
\label{eq:hessleq}        
\sigma^{-1}_{k-1} \leq 
\sigma^{-1}_k+\sched_{\sigma^{-1}_k}(\tau)
    \end{equation}
    Relabeling the indices again so that $k=\sigma_{i}$ and 
$k-1=\sigma_j$, we obtain
part 
%\ref{item:hesstaua}) 
for $1\leq j \leq n$. 
Now range $n-l+1 \leq i \leq n$ 
are the values at which 
$\sched_i(\tau)+i=n+1$, establishing the case of $i=1$ in Definition \ref{def:fltaucomb}.
Equation \eqref{eq:hessleq} shows that
the conditions for $i\geq 2$ follow from the condition for $i=1$.
%, establishing %part \ref{item:hesstaub}.

\end{proof}

Now consider the map 
$\chi_{\hesscomb(\tau)}$ as well the ideal $I_{\hesscomb(\tau)}$
and module $M_{\hesscomb(\tau)}$
from Definition \ref{def:chisets}.
We have the following, which will be used in our proof of Theorem B.

\begin{prop}
\label{prop:schedindep}
The monomials $\{\bx^\bk:\bk\in \scheds(\tau)\}$ are linearly independent over $\C[\epsilon]$ in
$\C[\bx,\epsilon]/I_{\hesscomb(\tau)}$.
\end{prop}

We will see in the proof of Lemma \ref{lem:equivindep} that they are in fact 
a basis.

\begin{proof}
Let $h=h_\tau$ and $I_h$ be as above.
By Lemma \ref{lem:eq:hessleq}, we have that $\hesscomb(\tau)\subset \hesscomb(h_\tau)$ so that
$I_{h} \subset I_{\hesscomb(\tau)}$.
Now let
\begin{equation}
\label{eq:qdef}
    q(\bx,\epsilon)=(x_1-\epsilon)\cdots (x_{n-l}-\epsilon),
\end{equation}
where $l$ is the length of the final run of $\tau$, so that $q(\sigma_1 \epsilon,...,\sigma_n\epsilon,\epsilon)=0$ whenever $\sigma$
satisfies \eqref{eq:lemmahesstau}.
Then we have that
$q(\bx,\epsilon) 
I_{\hesscomb(\tau)}\subset I_h$,
so that there is a well defined map
$\C[\bx,\epsilon]/I_{\hesscomb(\tau)}
\rightarrow \C[\bx,\epsilon]/I_h$
induced by multiplication by
$q(\bx,\epsilon)$. 
It suffices to show that the images of the desired monomials are linearly independent in $\C[\bx,\epsilon]/I_h$.

Applying the ring isomorphism $x_i\mapsto x_i-\epsilon$, Proposition \ref{prop:alghess} implies that the monomials
\begin{equation}
\label{eq:abetranformed}    
\left\{(x_1-\epsilon)^{k_1}\cdots (x_n-\epsilon)^{k_n}: 0\leq k_i \leq h(i)-i\right\}
\end{equation}
are linearly independent over $\C[\epsilon]$
in $\C[\bx,\epsilon]/I_h$.
Now by \eqref{eq:hftau} from the definition of $h_\tau$, the monomials
\begin{equation}
\label{eq:hesstaubasis}
\left\{q(\bx,\epsilon)(x_1-\epsilon)^{k_1}\cdots (x_n-\epsilon)^{k_n}: \bk\in \scheds(\tau)\right\}
\end{equation}
are actually subset of \eqref{eq:abetranformed}, and so are also independent.
\end{proof}

We exhibit bijections of $\hesscomb(\tau)$ with the three sets in \eqref{eq:bijrescarssched}. 
We define an element $v_\tau\in W^+$ by
\begin{equation}
\label{eq:defvtau}    
v_\tau= w\tau,\ \ w_i=i-a_in,\ \ \ba=\majt(\tau).
\end{equation}
It is straightforward to show that $v=v_\tau$ is a Grassmannian affine permutation, 
meaning that $v_1<\cdots<v_n$ and that
$h_\tau$ is the Hessenberg function 
determined by the rule that
\begin{equation}
    \label{eq:vtauhf}
    j>h_\tau(i) \Leftrightarrow v_j>v_i+n.
\end{equation}
Define $\reshess_\tau: \Res(\tau)\rightarrow S_n$ by
\begin{equation}
\label{eq:reshess}    
\reshess_\tau(w)=\rho^{\maj(\tau)}wv_\tau
\end{equation}

A second map is given by
$\schedhess_\tau:\scheds(\tau)\rightarrow S_n$ 
as follows: first start by setting $\sig$ to be an
arrangement starting with the number $n+1$, which
we will think of as $\sig_0$.
Then for $i$ from $n$ to $1$, insert the number $i$
to the right of the $k_i$-th element of $\run_i(\sig)$,
where the order is the opposite of the order in which they
appear in $\sigma$, i.e. right to left. Finally, remove the leading $n+1$
and let $\sktof_\tau(\bk)=\sigma^{-1}$.

\begin{ex}
Let $\tau=(3,5,1,2,7,4,6)$, and
$\bk=(2,1,0,0,1,0,0)\in 
%\sigtokll(\tau)$, 
\Res(\tau)$,
which corresponds to the parking function 
in Figure \ref{fig:bigpf} under $\varphi$.
Then the sequence would be
\[8,\ 87,\ 876,\ 8756,\ 87546,\ 875436,\ 8754236,\ 87541236,\]
so $\sktof_\tau(\bk)$ would be 
$(7,5,4,1,2,3,6)^{-1}=(4,5,6,3,2,7,1)$.
\end{ex}

We now prove a third description of this set:
\begin{prop}
  \label{prop:hessbij}
  We have that $\hesscomb(\tau)$ is equal to the images of both $\reshess_\tau$
  and $\schedhess_\tau$, and each map is  a bijection onto its image.
  They are compatible with the bijections in \eqref{eq:bijrescarssched}, meaning that
  $\schedhess^{-1}_\tau \reshess_\tau=\ressched^{-1} \respf$.
Moreover, if $\sigma=\reshess_\tau(w)=\schedhess_\tau(\bk)$, then
$\codim_h(\sigma)=\codim_{n+1}(w)=|\bk|$.
\end{prop}
\begin{proof}
 
% We check that the image of $\schedhess_\tau$ is $\hesscomb(\tau)$, leaving the rest.

For the first map $p_\tau$, 
it is straightforward to translate the 
criteria determining that $\sigma^{-1} 
\in \hesscomb(\tau)$
to the criteria that $v_\tau \sigma$ is $(n+1)$-stable
using \eqref{eq:vtauhf}.
The criteria for being in $\hesscomb(h_\tau)$
correspond to the constraints in
\eqref{eq:stable}
for $i\in \{1,...,n-1\}$, whereas the additional criterion
in \eqref{eq:lemmahesstau} determining
$\hesscomb(\tau)$ corresponds to $i=n$.

For the second map,
 it is clear that \eqref{eq:hessleq} is satisfied at
  every step in the construction of 
  $\schedhess_\tau$,
  because each number is added to the right of a number in
  $r_i(\tau)$, and adding a smaller number to the left
  of any digit preserves the condition.
  This shows that $\im(\schedhess_\tau)\subset \hesscomb(\tau)$.

  To see the reverse, suppose that 
  $\sig^{-1}$ satisfies the desired condition,
  and let $\sig'$ denote the result of adding $\sig_i$ immediately to the right of
  $\sig_j$ at every step in Definition of 
  $\schedhess_\tau$, where $j$ is the
  largest index satisfying $j<i$, and $\sig_j>\sig_i$, or $j=0$ if none
  exists. It is clear that $\sig'=\sig$, and it remains to show
  that we necessarily have $\sig_j \in r_{\sig_i}(\tau)$, so that
  $\sig' \in \im(\schedhess_\tau)$.
  To see this, we simply confirm the equation
  \[\sig_{j}\leq \sig_{j+1}+ \sched_{\sig_{j+1}}(\tau) \leq \sig_{i}+\sched_{\sig_i}(\tau),\]
establishing that 
$\hesscomb(\tau)\subset \im(\schedhess_\tau)$.

We leave the other statements.

\end{proof}

\begin{ex} 
    \label{ex:cars}
We list the four sets for 
$\tau=(3,1,2,5,4)$. This example also discussed on the page 80  of the  book \cite{haglund2008catalan}, with a slightly different notations.
First, the schedules are given by
\[\sched(\tau)=(2,2,1,1,1),\ \ \scheds(\tau)=\{00000,01000,10000,11000\},\]
so that $\schedpfl(\tau)$ is
\[\{(11201,00000),(11201,00100),(11201,01000),(11201,01100)\}.\]

We also have $h_\tau=(3,4,4,5,5)$,
and $l=|r_3(\tau)|=1$, so that we have $\sigma_5=1$ for all $\sigma \in \fltaucomb(\tau)$. 
We find that $\hesscomb(h_\tau)$ has 36 elements, and that
\[\hesscomb(\tau)=\{(4,3,5,2,1),(4,5,3,2,1),(5,3,4,2,1),(5,4,3,2,1)\}.\]

Next, applying $\reshess_\tau^{-1}$, we obtain
%\[\Res(\tau)=
%\{(4,3,10,-4,2),(3,5,9,-4,2),(5,3,9,-4,2),(3,4,10,-4,2)\}\]
\[\Res(\tau)=
\{(4, 3, 10, -4, 2), (5, 3, 9, -4, 2), (3, 4, 10, -4, 2), (3, 5, 9, -4, 2)\}\]
are the restricted permutations.

Finally, these sets correspond under the above bijections to the elements of $\cars(\tau)$ shown in Figure \ref{fig:cars}.
\end{ex}

\section{Hilbert schemes and related combinatorics}
In this section we study combinatorics of the Hilbert schemes of points of points on the plane. In Section~\ref{sec:hilbcomb} we study combinatorics of torus fixed locus \(\pfhilbcomb^{d}(n)\) of the parabolic { flag} Hilbert scheme that  { appears} in Section~\ref{sec:geohess}.
We relate this set to \(\PF(n)\). In Section~\ref{sec:hilbcells} we study the cell decomposition of the usual Hilbert scheme
\(\hilbgeo^d(\CC^2)\) that is equivariant with respect to the one-dimensional torus \(\CC^*_{a,b}\).

In Section \ref{prop:iarrobino} 
%\st{we strengthen the results of the previous section in the case \(a=b=-1\)}
{ We show how the combinatorial quantities from the previous section appear as}
%and we show that {\color{blue} they} form
a fixed component of the fixed locus
the attracting set 
{ within} the punctual Hilbert scheme \(a=b=-1\), { which was 
shown by Iarrobino to be
a smooth affine fibration with a smooth base.} 
This result is used in
Section~\ref{sec:geohess}.

\subsection{Hilbert scheme combinatorics}

\label{sec:hilbcomb}

Call a subset $S\subset \zptwo$ an \emph{ideal} of 
colength $d$ 
if it is closed under addition
by the elementary unit vectors $e_1,e_2=(1,0),(0,1)$, and
$|\mathbb{Z}^2_{\geq 0}-S|=d$.  
The complement $\zptwo-S$ of an ideal
is a finite set which is
%set called the \emph{pattern} of $T$, using 
%the terminology from \cite{iarrobino1977punctual}.
%The patterns are precisely 
the Ferrers diagrams $\ferrers(\lambda)$ 
of some Young diagram $\lambda$. Using the terminology
of \cite{iarrobino1977punctual}, a partition is
said to have \emph{normal pattern}
if the diagonal of points 
$(a,b) \in \lambda$ for which $a+b=k$ consists
of the values $(0,k),...,(j,k-j)$ for some $j$.

For any $S \subset \zptwo$, we define 
the Hilbert series
\begin{equation}
\label{eq:defhscomb}    
\hs_{q,t}(S)=\sum_{(r,s) \in S} q^rt^s \in \C[[q,t]].%\ \ 
\end{equation}
If the set is a pattern $P$ associated
%\in \hilbcomb^d$, then its complement is the Ferrers diagram
to a Young diagram $\lambda$, then we have
$\hs_{q,t}(P)=B_{\lambda}(q,t)$ where
%$\hs_{q,t}(T)=M^{-1}-B_\lambda(q,t)$ where
%\[M=(1-q)(1-t),\ \ 
%\[\hs_{q,t}(P)=
\[B_\lambda(q,t)=\sum_{(r,s)\in \ferrers(\lambda)} q^r t^s.\]
If $S$ is the complement of $\lambda$, 
then $\hs_{q,t}(S)=M^{-1}-B_\lambda(q,t)$ where $M=(1-q)(1-t)$.
Now let $\hs_{z}(S)=\hs_{q,t}(S)|_{q=z,t=z}$ denote the Hilbert function corresponding to the total sum. There is at
most one partition $\lambda$ with normal pattern for each 
possible Hilbert series $\hs_z(\lambda)$.

We now define
$\pfhilbcomb^{d}(n)$ to 
be the collection of flags of partitions
$\parflag=(\lambda^0 \subset \cdots \subset \lambda^{n})$
%$\tilde{T}=(T_0 \supset \cdots \supset T_{-n})$ 
such that $|\lambda^i|=d+i$, and
%$T_{-i} \in \hilbcomb^{d+i}(\zptwo)$ and 
$S_{n} \subset e_1+S_{0}$
where \(S_i=\ZZ_{\ge 0}^2-\lambda^i\), $e_1+S$ is the set of vectors 
$(a+1,b)$ for $(a,b)\in S$, and $S_i$ is the complement
of $\lambda^i$.
This set corresponds to certain
torus fixed points of the
\emph{parabolic flag Hilbert scheme} defined in
\cite{carlsson2017parabolic}, which will appear in
Section \ref{sec:geohess}.
We let $\pfhilbcomb(n)$ denote the union of all the
$\pfhilbcomb^{d}(n)$ for all $d$.

% We have an injective map $\pfl(n) \rightarrow \pfhilbcomb(n)$ as follows: given a parking function in 
% $\pfl(n)$, let the diagram of $\lambda^0$ consist of all pairs $(a,b)\in \zptwo$ which are above the 
% path, oriented so that $x$ corresponding to East steps and $y$ corresponding to South steps starting from the upper left. 
% In other words, $\lambda^0$ is the partition associated to the underlying Dyck path, with all conventions about
% alignments in agreement.
% We then define each subsequent partition $\lambda^{i}$ 
% by adding the squares containing the label $j$ for which
% $j \leq i$.

% The following describes the images in terms of the above bijections between restricted affine permutations and combinatorial lattices.
\begin{prop}
  \label{prop:extmapaa} There is an injective map:
  \[\extcomb:\\\Res(n,n+1)\to \pfhilbcomb(n).\]
For any $w\in \Res(n,n+1)$,
the we define  $\extcomb(w)=(\lambda^0\subset \dots\subset \lambda^n)$ in 
$\pfhilbcomb(n)$  by
\begin{equation}
\label{eq:extmap}
     \zptwo-\lambda^i=\left\{(a,b)\in \zptwo :
    na+mb\in E_i\right\}   
\end{equation}
where $\latflag=\latcomb(w)$.
\end{prop}

The composition $\extcomb\circ \respf^{-1}$ where $\alpha$ is the bijection of restricted permutations to 
parking functions is particularly easy to describe:
the region above the graph of the underlying Dyck path 
corresponds to the Ferrer's diagram of $\lambda^d$,
while the subsequent partitions are the result of adding the labeled boxes in decreasing order.
For instance,
the first parking function in Figure \ref{fig:cars},
which corresponds
to $w=(4,3,10,-4,2)$ %in Example \ref{ex:cars},
from Example \ref{ex:cars},
would map to the sequence
\begin{align}
\label{eq:xypflex}
\notag
(2,2,1)\subset (2,2,1,1) \subset (2,2,1,1,1)
\subset \\
 (3,2,1,1,1) \subset (3,2,2,1,1) \subset (3,3,2,1,1).
\end{align}
%where we are identifying the flags $T$ with their monomial ideals.
Notice that not all elements of $\pfhilbcomb(n)$ come from parking functions. The elements of the latter require that 
$\lambda^{n}-\lambda^0$ is a vertical strip, whereas 
for parking functions the strips will 
always have the additional boxes appear in exactly the top $n$ rows.

The image of $\cars(\tau)$ consists of flags
$\parflag$ which have the same value of
$\hs_z(\lambda^j)$ for each $j$. These polynomials will be denoted by $F_\tau^\bullet$ for $F^j_\tau(z)$.
The first one can be described explicitly in terms of 
$\tau$ by $F^0_\tau(z)=F_\mu(z)$ where
\begin{equation}
\label{eq:Fmu}    
   F_\mu(z)=\sum_{i=1}^{ n} iz^{i-1}-
    \sum_{i={ 0}}^{k-1} 
    (\mu_1+\cdots+\mu_{k-i}) z^{n-1-i}
\end{equation} 
and $\mu=(\mu_1,...,\mu_k)=\mu(\tau)$ is the run composition 
of $\tau$. The subsequent ones satisfy
$F^{i+1}_\tau(z)-F^{i}_\tau(z)=
z^{{ n-1-a_{n-i+1}}}$
where $(a_1,...,a_n)=\majt(\tau)$.
We record this description in the following proposition.
%of the 
%image of $\cars(\tau)$.
\begin{prop}
\label{prop:idealcars}
 The image of $\cars(\tau)$
in $\pfhilbcomb(n)$ under $\extcomb\circ \respf^{-1}$
is the collection of flags 
$\parflag=(\lambda^{d}\subset \cdots \subset \lambda^{d+n})$ satisfying $F_\tau^j(z)=\hs_z(\lambda^j)$.
\end{prop}
For instance, for the example in \eqref{eq:xypflex} we would have %$d=5$, and
\[F^{0}_\tau(z)=1+2z+2z^2,\ \ 
(F_\tau^{i}(z)-F^{i-1}_\tau(z))_{i=1}^5={ (z^3,z^4,z^2,z^3,z^3)}.\]
This comes from an element of $\cars(\tau)$ for $\tau=(3,1,2,5,4)$,
for which $\majt(\tau)=(1,1,2,0,1)$. Said another way, we can recover $\tau$ by reading the exponents in decreasing order from left to right to break ties, in the way that recovers $\tau$ from its descent composition, { described in Section \ref{sec:combnoations}}.

\subsection{Cell decompositions of the punctual Hilbert scheme}

\label{sec:hilbcells}

Recall the Hilbert scheme of points in the plane
$\hilbgeo^d(\C^2)=\hilbgeo^d(\C[x,y])$, 
described as a set by
\begin{equation}
\label{eq:hilbring}    
\hilbgeo^d(R)=\left\{I\subset R
: \dim_{\C} R/I=d\right\}.
\end{equation}
We also have the \emph{punctual Hilbert scheme}
$\hilbgeo^d_0(\C^2)=\hilbgeo^d(\C[[x,y]])$ in which the polynomial
ring is replaced by the power series ring.
We have an inclusion $\hilbgeo^d_0(\mathbb{C}^2) \subset \hilbgeo^d(\mathbb{C}^2)$, 
whose image is identified with the collection of ideals
$I \in \hilbgeo^d(\mathbb{C}^2)$ whose vanishing set
is supported at the origin, meaning it is 
contained in some power of the maximal ideal 
$\mathfrak{m}=(x,y)$.
The Hilbert scheme is a smooth quasi projective variety
\cite{fogarty1968algebraic}, whereas the punctual Hilbert scheme is { just} irreducible \cite{briancon1977description}.
%\cite{Nakajima94}.

We have the usual well-studied action of the torus 
$\C^*\times \C^*$ on both $\hilbgeo^d(\C^2)$ and
$\hilbgeo^d_0(\mathbb{C}^2)$ by
\begin{equation}
\label{eq:fulltorushilb}
    (q,t)\cdot 
f(x,y)= f(qx,ty).
\end{equation}
The fixed point set of both spaces
consists of the discrete set of monomial ideals,
which is in bijection with the set of 
combinatorial ideals $\hilbcomb^d(\zptwo)$ from Section \ref{sec:hilbcomb}. 
These are in turn
in bijection with partitions $\lambda$ of size $d$,
and the corresponding ideals are labeled by $I_\lambda$.
%The torus character of the cotangent space at 
%such a point is given by
%\begin{equation}
%\label{eq:tanhilb}    
%\hs_{q,t} (T^*_\lambda)=
%%qtB_\lambda(q,t)+B_\lambda(q^{-1},t^{-1})-M %%B_{\lambda}(q,t)B_\lambda(q^{-1},t^{-1});
%\end{equation}
For integers $a,b$, we let
$\C^*_{a,b}=\{(z^a,z^b)\}\subset \C^*\times \C^*$
be the one-dimensional subtorus with those weights 
whose fixed point loci may not be discrete 
for some particular $d$. 
%Positive values of both $(a,b)$
%leads to attracting

Following \cite{ellingsrud1988cell,haiman1998catalan},
We describe a local system of coordinates about each
$I_\lambda$ as follows.
First, we have the standard monomials of $I_\lambda$,
given by $\mathcal{B}_\lambda=\left\{x^ry^s:(r,s)\in \lambda\right\}$, which determine a vector space basis of the quotient space $R/I_\lambda$.
Then we have an open cell, defined by
\begin{equation}
U_\lambda=\left\{I\in \hilbgeo^d(\C^2):
\mbox{$\mathcal{B}_\lambda$ spans
$R/I$}\right\}.
\end{equation}
For any $I\in U_\lambda$, there is a unique 
expansion
\begin{equation}
\label{eq:ulaexpansion}    
x^r y^s =\sum_{(h,k)\in \lambda} c^{r,s}_{h,k}(I) x^hy^k.
\end{equation}
The coordinates $(h,k)$ range over the boxes of $\lambda$,
whereas the $(r,s)$ range over all of $\zptwo$. For $(r,s)\in \lambda$,
we have $c^{r,s}_{h,k}$ is one if $(r,s)=(h,k)$ or zero otherwise.

We have a local coordinate system
$c^{r,s}_{h,k}(I)$
of $U_\lambda$ parametrized by a collection
$\mathcal{C}_\lambda=\{c^{r,s}_{h,k}\}$
of $2n$ particular coordinates,
%$\mathcal{C}(\lambda)=\{c^{r,s}_{h,k}:(h,k)\in \lambda,(r,s)\in \mathcal{C}_{h,k}(\lambda)\}$, where the inner set is 
defined as follows. Given $(h,k)\in \lambda$,
let $(f,k)$ denote the bottom square in column $k$ of
$\lambda$, and let $(h,g)$ denote the rightmost square
in row $h$. Then we have a local system of parameters 
$\mathcal{C}_\lambda$
consisting of the coordinates of the form $c^{h,g+1}_{f,k}$, or
$c^{f+1,k}_{h,g}$ (see \cite{haiman1998catalan} Corollary 2.5). 
These are in bijection with the weights of the tangent space to the Hilbert scheme at $\lambda$, so that we have an identity
\begin{equation}
\label{eq:hilbchartqt}    
\sum_{c^{r,s}_{i,j}}
q^{r-i}t^{s-j}=\hs_{q,t}(T^*_\lambda).
\end{equation}
The remaining coordinates may be solved for as a power series in terms of the others. It is shown that the $U_\lambda$ are 
open subvarieties, though in general are not affine cells.

\begin{ex}
\label{ex:esstrip}
If $\lambda=(2)$, the family of ideals $U_\lambda$ may be described by
\begin{equation}
(x^2-c^{20}_{10}x-c^{20}_{00},
y-c^{01}_{10}x+c^{01}_{10} c^{20}_{10}-c^{11}_{10}).
\end{equation}
While the coefficients in this example are polynomials in $c^{rs}_{hk}$, they are not in general.
\end{ex}

We describe the Bialynicki-Berula type
cell decomposition of 
$\hilbgeo^d(\C^2)$ and $\hilbgeo^d_0(\C^2)$
from \cite{ellingsrud1988cell} in these 
coordinates. Given $(a,b) \in \zptwo$,
%be such that the fixed points of the one-dimensional torus
%let $U=\mathbb{C}^*_{a,b}$ acting on $\hilbgeo^{d}(\C^2)$,
%and 
%$\hilbgeo^d_0(\C^2)$
%are discrete, and so given by the partition ideals.
%We 
let us denote the attracting sets of $U=\C^*_{a,b}$ by
\begin{equation}
\label{eq:escells}
\hilbgeo^{d}(\C^2)^{U,+}_\lambda=
\left\{I\in \hilbgeo^d(\C^2):
\lim_{z\rightarrow 0} z\cdot I=I_\lambda\right\}.
\end{equation}
The intersections with $U_\lambda$ can be described as follows.
Let 
\begin{equation}
    \mathcal{C}_\lambda^{a,b}=
    \left\{ c^{r,s}_{h,k} \in \mathcal{C}_\lambda:
    ar+bs>0\right\}.
\end{equation}
Let $S_{\lambda}^{a,b}=\C[\mathcal{C}^{a,b}_\lambda]$ denote the polynomial algebra generated by these coordinates. We have a compatible action of the full two-dimensional torus with 
scales each $c^{r,s}_{i,j}$ by $q^{r-i}t^{s-j}$.

Notice that if $a,b>0$, then not every point in the Hilbert scheme is in some attracting set, as points can be sent to infinity.
However, it was shown that intersecting the attracting set with $U_\lambda$ gives a cell decomposition of both the punctual and non-punctual Hilbert scheme:
\begin{prop}[\cite{ellingsrud1988cell}, Theorem 2+corollaries]
\label{prop:escells}
Suppose that $a,b<0$ and that $U=\C^*_{a,b}$ acts on $\hilbgeo^d(\C^2)$ with isolated fixed points. Then we have a cell decomposition of the full Hilbert scheme by the cells
$C^{a,b}_\lambda \cong \Spec(S^{a,b}_\lambda)$. If we instead have $a,b>0$, then we obtain a cell decomposition
of the \emph{punctual} Hilbert scheme by
$C^{a,b}_{\lambda}\cong \Spec(S^{a,b}_\lambda)$. The isomorphisms
are equivariant with respect to the action of the full torus
$T=\C^*\times \C^*$.
%$a\gg b>0$. Then the cells
%in \eqref{eq:escells}
%The intersections $C_{\lambda,0}=C_\lambda \cap \hilbgeo^d_0(\C^2)$ determine a cell decomposition of the punctual Hilbert scheme
%determine a cell decomposition of
%$\hilbgeo^d_0(\C^2)$. The additional relations are given by $c^{r,s}_{h,k}=0$
%unless $h+k\geq r+s$.
\end{prop}

\subsection{Decomposition of homogeneous ideals}
%which is discussed .

We now describe a homogeneous version of the Hilbert scheme studied by Iarrobino and later by G\"{o}ttsche \cite{iarrobino1977punctual,
gottsche1990punctual}. 
In Section \ref{sec:geohess},
we show that these spaces turn out to be
precisely the Grassmannian version 
of the space whose torus fixed points are
described by $\hesscomb(\tau)$ from 
Section \ref{sec:hesscomb}.

For any $I\subset \C[[x,y]]$,
let $\inideal(I)$ denote the ideal of 
initial forms of $I$ with respect to 
$\mathfrak{m}=(x,y)$. In other words, it is 
given by
%the kernel of the map $\C[[x,y]]\rightarrow
%\gr(\C[[x,y]]/I)$ where 
\begin{equation}
\label{eq:initideal}    
\inideal(I)=\bigoplus \inideal_j(I) \subset 
\C[[x,y]],\ \ 
\inideal_j(I)=%\frac{
\frac{I \cap \mathfrak{m}^j}{I\cap \mathfrak{m}^{j+1}}
\end{equation}
which is the kernel of the map to the associated graded ring of $\C[[x,y]]/I$.
We say that an ideal is homogeneous if it is generated by homogeneous elements with respect
to total degree. If $I$ is homogeneous, we have its Hilbert series $\hs_z(I)$ with 
respect to total degree, which is consistent with the notation of the previous section.

Now consider the collection
$\ghilbgeo(\C^2)\subset\hilbgeo_0(\C^2)$ 
of ideals which are homogeneous with respect to total
degree.
%The Hilbert scheme of 
%homogeneous ideals $I$ a 
This is a disconnected set, 
whose components are determined by 
\begin{equation}
    \ghilbgeo(\C^2)_{F}=
    \left\{\mbox{homogeneous $I\subset \C[x,y]$}:
\hs_{z}(\C[x,y]/I)=F(z)\right\},
\end{equation}
%where the superscript indicates that $x,y$ both have degree one,
which agrees with the fixed point set 
$\hilbgeo_0(\C^2)_F^U=\hilbgeo(\C^2)^{U}_F$
for $U=\C^*_{1,1}$. 
Then taking the initial ideal defines a function
which can be interpreted as taking the attracting point for
the this action
\begin{equation}
%\inideal(I)=
\inideal :\hilbgeo_0(\C^2)\rightarrow \ghilbgeo(\C^2),\ \ 
    \inideal(I)=\lim_{z\rightarrow 0} z\cdot I
\end{equation}
The fibers are therefore the attracting cells, giving a partition
\begin{equation}
\label{eq:grhilbpartition}
 \hilbgeo(\C^2)=\bigsqcup \ghilbgeo(\C^2)^+_F
,\ \ 
\ghilbgeo(\C^2)^{+}_{F}=
\left\{I: 
\inideal(I)\in\ghilbgeo(\C^2)_F\right\}
%\hs(\in(I))=\right\}
\end{equation}
and similarly for the punctual Hilbert scheme, whose components
are denoted $\ghilbgeo_0(\C^2)^+_F$.

\begin{prop}[\cite{iarrobino1977punctual} Theorem 2.12]
\label{prop:iarrobino}
For each polynomial with nonnegative integer coefficients $F$, 
let $a_i$ be the coefficient of $z^i$ 
in the power series
$1/(1-z)^2-F(z) \in \C[[z]]$, which is
the Hilbert series of 
any homogeneous ideal $I$ with 
$\hs_z(\C[[x,y]]/I)=F(z)$. Then 
\begin{enumerate}
    \item 
Either
$\ghilbgeo(\C^2)_{F}$ is 
empty, or it is a smooth variety of dimension 
$\sum_{i\geq 1} (a_i-a_{i-1})(a_{i+1}-a_{i}-1)$.
\item \label{item:iarrorank}
If it is nonempty, then 
$\ghilbgeo_0(\C^2)^{+}_{F}$
is an affine bundle over it of rank 
$\sum_{i\geq 1} a_{i-1}(a_{i+1}-a_i-1).$
The projection is the map which sends
$I$ to $\inideal(I)$.
\end{enumerate}

\end{prop}

We recall some of the supporting lemmas in the proof for future use, written in the context of 
Proposition \ref{prop:escells}, which appeared later. 
%{\color{red} 
Let us denote by \(G^{a,b}_\lambda\) the intersection
\(C^{a,b}_\lambda\cap  \ghilbgeo(\C^2)\), which is the 
attracting set in \(\ghilbgeo(\C^2)\).
%Notice first that the cells
%$C_{\lambda,0}$ from Proposition \ref{prop:escells}
%are entirely contained in 
%$\ghilbgeo_0(\C^2)^{+}_{F}$
%where $F(z)=B_{\lambda}(z,z)$.
%Let $G_\lambda= C_{\lambda} \cap \ghilbgeo^d(\C^2)$
%o be the corresponding cell in the Hilbert scheme of homogeneous ideals.
%\end{equation}
\begin{lemma}
\label{lem:iarrobinoproduct}
Let $\lambda$ be a partition whose Young diagram has a normal pattern, and suppose that $\deg(\hs_z(\lambda))<a<b$.
Then we have an isomorphism
\begin{equation}
\label{eq:normalprod}    
C^{a,b}_{\lambda}\cong G_{\lambda}^{a,b} \times \mathbb{A}^k
\end{equation}
where $\mathbb{A}^k$ is affine space of dimension $k$,
and $k$ is equal to the desired rank of the bundle
in Proposition \ref{prop:iarrobino}.
\end{lemma}
\begin{proof}
Using the equivariance with respect to the full torus 
from Proposition
\ref{prop:escells}, we see that $G_\lambda^{a,b}$ is described 
as a subspace
of $C^{a,b}_{\lambda}$ by the additional conditions that 
$c^{r,s}_{i,j}=0$ unless $r+s=i+j$. The desired expression for the rank can been seen to agree with the number of elements of $\mathcal{C}^{a,b}_\lambda$ which do not satisfy this property.
%with the number of remaining parameters from Proposition 
%\ref{prop:escells}.
\end{proof}
The cells with normal pattern are the top-dimensional 
ones in $\ghilbgeo(\C^2)_F^{+}$ for 
$F(z)=B_\lambda(z,z)$, 
and are open. In fact, any ideal can be moved into this 
set under a change of coordinates:
\begin{lemma}
\label{lem:iarrobinounion}
For any $I\in \ghilbgeo_0(\C^2)_{F}^{+}$, there is a number 
$c\in \C$ such that the image of $I$ under the substitution 
$y=y+cx$ is in $C^{a,b}_{\lambda}$, where $\lambda$ is the pattern from the previous lemma. 
It is possible to select finitely
many numbers $c_i$ such that every $I$ is carried into 
$C^{a,b}_{\lambda}$ for some such transformation with $c=c_i$.
\end{lemma}

\begin{proof}
See \cite{iarrobino1977punctual} Proposition 3.2.
%Notice that $C_\lambda$    
\end{proof}
Lastly, the full, non-punctual Hilbert scheme is smooth, 
{so the fixed locus for the action of $\C^*_{1,1}$ is also smooth. We have}
\begin{lemma}
\label{lem:iarroninosmooth}
The space of homogeneous ideals 
$\ghilbgeo(\C^2)_{F}$ is smooth.    
\end{lemma}

We can now prove Proposition \ref{prop:iarrobino}.
\begin{proof}
    By Lemma \ref{lem:iarrobinounion} we can
    describe $\ghilbgeo(\C^2)_F^{+}$ as a finite union of 
    isomorphic copies of $C^{a,b}_{\lambda}$ under the substitution $y=y+c_i x$. By Lemma
    \ref{lem:iarrobinoproduct}, we have that each one is locally trivial. The smoothness is Lemma \ref{lem:iarroninosmooth}.
\end{proof}

We next recall a result of G\"{o}ttsche describing a cell decomposition of this space, which was used to compute the Betti numbers of the punctual Hilbert scheme. 
\begin{prop}[\cite{gottsche1990punctual}, Theorem 2.1]
\label{prop:gottsche} 
%Let $U=\C^*_{n,n+1}$ for some $n\geq d$. Then 
For any $n>\deg(\hs_z(\lambda))$, we have cell decompositions
\begin{equation}
\label{eq:gottscheattracting}    
\ghilbgeo^d_0(\C^2)^{+}_{F}=\bigsqcup_{\lambda} C^{n,n+1}_{\lambda},\ \ 
\ghilbgeo^d(\C^2)_F=\bigsqcup_{\lambda} G^{n,n+1}_\lambda
\end{equation}
%\hilbgeo^d_0(\C^2)^{U,+}_\lambda
%is the union of the attracting cells $_\lambda$ as 
where $\lambda$ ranges over partitions for which
$F(z)=\hs_z(\lambda)$. The dimension of $G_\lambda$ is the 
number of coefficients $c^{r,s}_{i,j}$ for which 
$r+s=i+j$ and $i<r$.
%number of boxes $(i,j)\in \lambda$ for which 
%$i-j+\lambda'_{i+1}-\lambda_{j+1}\notin \{0,1\}$.
\end{prop}

\begin{proof}
The left side of \eqref{eq:gottscheattracting}
amounts to checking that $C^{n,n+1}_{\lambda}$ lies entirely
in $\ghilbgeo^d_0(\C^2)^+_F$, which follows since 
the initial ideal
$\inideal(I)$ is in $C^{n,n+1}_\lambda$ whenever $I$ is
for this action.
The remaining statements follow from the fact that 
$G^{n,n+1}_\lambda$ can be described by setting the
remaining coordinates $c^{r,s}_{i,j}$ to zero,
making use of the equivariance statement in Proposition
\ref{prop:escells} again.
\end{proof}

\section{Cherednik algebras and the affine Springer fiber}
\label{sec:geometry}
We now recall some results about the affine Springer fibers and
affine flag varieties that we will need for our main results
in Chapter \ref{doubsec}.
The reader interested mainly in algebra can skip
everything in this section, except for possibly the
conventions for the root system in type $A$,
provided they are willing to take Proposition \ref{prop:spbasis} of Section \ref{sec:affine-schub-calc} on faith. In this paper, $(n,m)$ will always be coprime.

\subsection{Root systems}
In this section we fix our conventions on the root system for type \(A\).
Let $\fg=\fsl_n$, let
$\fghat=\fslhat_n$ be the corresponding affine Lie algebra,
and let $\fthat$ denote the Lie algebra 
of the maximal torus $\That \subset \widehat{SL}_n$.
%The dual $\fthat^*$ of the maximal
%torus is spanned by the fundamental weights \(\lambda_i\): %together withthe imaginary root
Respectively, \(\underline{\fthat}\)  is Lie algebra of the
corresponding maximal torus of \(\widehat{GL}_n\):
\begin{equation} \label{eq:defweights}
\underline{\fthat}^*=  \langle \lambda_1,...,\lambda_{n},\delta\rangle.
\end{equation}
The ambient space \(\underline{\fthat}^*\) is equipped with the bilinear form
\[\langle \lambda_i,\delta\rangle=\langle \delta,\delta\rangle=0,\quad
  \langle\lambda_i,\lambda_j\rangle=-1/n,\quad
  \langle\lambda_i,\lambda_i \rangle=(n-1)/n,\]
where \(1\le i,j\le n\) and \(i\ne j\).
It is convenient to  define \(\lambda_i\), \(i\in \ZZ\) by
setting $\lambda_{i+n}=\lambda_{i}-\delta$ for all $i$.
% The roots
% \(\lambda_i\), \(i\in \{1,\dots,n\}\) form a basis of
% a subspace $\mathfrak{t}^*\oplus\langle\delta\rangle$. In particular,
% the projection: \(\fthat^*\to \mathfrak{t}^*\) acts by
% \begin{equation}\label{eq:proj-roots}
%   \delta\mapsto 0,\quad \lambda_i\mapsto \eta_i,\quad i=1,\dots,n.
%\end{equation}
%Thus we fix notation \(\eta_i\) for the spanning set of \(\mathfrak{t}^*\) that satisfies the relation \(\eta_1+\dots+\eta_n=0\).
%\todo[inline]{oh man, can we find another symbol? :)}
%\[\underline{\fthat^*}=\langle \lambda_1,\dots,\lambda_n\rangle.\]
Let us now define { the} subspace \(\fthat^*\).
The simple roots in \(\fthat^*\) are given by
\[\alpha_{i}=\lambda_{i}-\lambda_{i+1},\quad 0\leq i \leq n-1,\]
and these for a basis of \(\fthat^*\).
In particular,
the projection: \(\fthat^*\to \mathfrak{t}^*\) acts by
\begin{equation}\label{eq:proj-roots}
  \delta\mapsto 0,\quad \lambda_i\mapsto \eta_i,\quad i=1,\dots,n.
\end{equation}
Thus we fix notation \(\eta_i\) for the spanning set of \(\mathfrak{t}^*\) that satisfies the relation \(\eta_1+\dots+\eta_n=0\).
%\todo[inline]{oh man, can we find another symbol? :)}

The action of the affine Weyl group
(see Section \ref{sec:affine-permutations}) 
is given by
\begin{gather}
  \label{wact}
 { s_i(\lambda_j)=\lambda_j-\langle\alpha_i,\lambda_j\rangle\alpha_i},\quad
  s_i(\delta)=\delta,\quad 
  w(\lambda_j)=\lambda_{w^{-1}_j},
\end{gather}
for $i,j\in\{0\dots,n-1\}$ and \(w\in W\).
The first equation defines the action of \(W\) on the ambient space \(\underline{\fthat}^*\) and this action preserves subspace \(\fthat^*\).

The third equation follows from the first two,
and in fact holds for any integer $j$, and is defined below for
extended affine permutations $w\in \extperms$ as well. 
The elements of \(\extperms \setminus W\) are not compositions of the reflections \(s_i\). Later we use the elements \(\psi_i, \rho\in \extperms\setminus W\), \(i=1,\dots,n\) and
the third formula in \eqref{wact} implies the action on 
on the ambient space:
\[{ \psi_i(\lambda_j)=\begin{cases}\lambda_j, &  j\ne i \\
  \lambda_i+\delta, & j=i
\end{cases},
\quad \rho(\lambda_j)=\begin{cases}\lambda_{j+1},&  j<n\\ \lambda_1+\delta,& j=n.
\end{cases},}
\]
where \(j=1,\dots,n\).

\subsection{The affine flag variety}
\label{sec:geometryFl}

Let \(G\) be a complex algebraic group such that its Lie algebra
\(\frg\) is simple. We define \(\calO=\CC[[t]]\) to be the ring of formal power series of \(t\), and its quotient field is \(\calK\).
Respectively, \(G(\calK)\) is the group of formal loops and \(G(\calO)\) is the subgroup of holomorphic loops.
The quotient \(\agrass=G(\calK)/G(\calO)\) has the structure of the ind-scheme,
as an inductive limit by smooth subschemes.
These subschemes can be taken as
the union of Schubert cells,
described below,
with an upper bound on the length.
For more details, see the survey \cite{Zhu16}.

The affine flag variety is the ind scheme \(\Fl=G(\calK)/I\) where \(I\subset G(\calO)\) is the subgroup of elements \(g(t)\in G(\calO)\)
such that \(g(0)\in B\). 
In this paper we assume that \(G=SL(n)\) and \(T\subset B\subset G\) are the maximal torus and the Borel subgroup.
In this paper we will generally use a superscript of $\grsup$ to distinguish subspaces of the affine Grassmannian, 
and either $\flsup$ or no subscript 
for the affine flag case.

A lattice in \(\CC^n\otimes \calK\) is a subspace \(\flaglat\) that is preserved by \(\calO\) and the intersection
\(\flaglat\cap \calO^n\) is of finite codimension inside \(\flaglat\) and \(\calO^n\). The index \(\mathrm{ind}(\flaglat)=\mathrm{codim}_{\flaglat}\flaglat\cap\calO^n-
\mathrm{codim}_{\calO^n}\flaglat\cap\calO^n\) is well-defined for a lattice.
Geometrically, the affine flag variety $\Fl=\Fl_n$
can be described compactly as the collection of flags of lattices 
$(L_0\supset \cdots  \supset L_n)$ for which
%\[\Fl=\{\CC^n \otimes \calK \supset \cdots \supset \flaglat_{i} \supset \flaglat_{i+1}\supset \cdots \supset 0\ :\ 
%i\in \ZZ,\]
\begin{equation}
\label{eq:defaffflag}
\mathrm{ind}(\flaglat_0)=0,\quad
\flaglat_{i+1}\subset \flaglat_i,\ 
  \flaglat_i/\flaglat_{i+1}\cong \CC,\ 
  \ \flaglat_{n}=t\flaglat_0.
  \end{equation}
In this description we have the tautological line bundle \(\cL_i\) over \(\Fl\) has fiber \(\flaglat_i/\flaglat_{i+1}\) at the point \(\flaglat_\bullet\in\Fl\).
We may extend this cyclically using the final condition, and refer
to the flag $L_i$ for any $i\in \mathbb{Z}$, as well as the corresponding line bundle.

The torus \(\hat{T}=T\times\CC^*\) acts on \(G(\calK)\) by
allowing \(T\) to act by
left multiplication, and \(\CC^*\) 
to act by loop rotation
%\todo[inline]{should it be $\lambda^{-1}$ here? Or does this
%give the wrong sign?}
\(\mu\cdot g(t)=g(\mu^{-1} t)\) for $\mu \in \CC^*$.
This action has isolated fixed points which are enumerated by the affine permutations $W$. 
Indeed, if \(e_0,\dots,e_{n-1}\) is a basis of \(\CC^n\)
that is fixed by \(T\), then there is a unique flag of
torus-invariant lattices $\flaglat^w_\bullet\in \Fl$ satisfying
\begin{equation}\label{eq:Lambda-w}
\flaglat^w_i/\flaglat^w_{i+1}=\langle e_kt^{-m}\rangle,\quad
  w_i=mn+k,\quad  0\le k<n,\end{equation}
provided that $w$ satisfies:
\[w_i=w_{i-n}+n,\quad w_1+\dots+w_n= n(n+1)/2.\]
Thus, there is a natural identification between \(\tilde{\mathcal{F}}l^{\hat{T}}\) and \(\affperms\), as defined in Section~\ref{sec:affine-permutations}.

Recall that the group \(\affperms\) is generated by the reflections \(s_i\), \(i=0,\dots,n-1\) where \(s_i\) are defined in
Section~\ref{sec:affine-permutations}.
  There is also a natural homomorphism 
map \(W\to S_n\) that is defined as \(w\mapsto \sigma\), \(\sigma_i=
\langle w_i\rangle_n\) as defined after the formula~\ref{eq:rotshift}.
The kernel of the projection homomorphism the  free abelian group
with the generators \(\psi_{i}\psi_{i+1}^{-1}\), \(i=1,\dots,n-1\).
Here \(\psi_i\) are elements of the { extended} group \(\What\) and
they are defined by the formula~\eqref{eq:rotshift}.

There is a natural embedding \(\imath:\affperms\rightarrow G(\calK)\) such that \(\imath(w)=\flaglat_\bullet^w\).
The Bruhat decomposition  \(G(\calK)=\bigcup_{w\in\affperms}IwI\) induces the decomposition of \(\Fl\) into affine cells \(\Fl=\bigsqcup_{w\in\affperms}X^{\circ}_w\)
where \(X^{\circ}_w=IwI\) is the cell of dimension \(\ell(w)\). The affine Schubert variety \(X_w\) is the Zariski closure of
\(X_w^\circ\), which is the is the union of cells
\(X_w=\bigsqcup_{v\le_{bru}w} X^\circ_v\), where $\leq_{bru}$ is the Bruhat order. 
The Schubert cells in the case of the affine Grassmannian are given by
$X^{\grsup,\circ}_{w}=I w G(\mathcal{O})$ as $w$
ranges over the Grassmannian affine permutations.

We recall the construction 
of the equivariant
Borel-Moore homology from \cite{Graham01}.
In this paper, all (equivariant) homology and cohomology
groups will have coefficients in $\gf$.
Let \(Z\) be a scheme with an action of a linear algebraic group \(G\). Let \(V\)
be a representation of \(G\) and let \(U\subset V\) be an open subset where
\(G\) acts freely. Then the equivariant cohomology and Borel-Moore homology
are defined by:
\[\hbm_{G}^i(Z)=H^i(U\times^G Z),\quad \hbm_j^G(Z)=\hbm_{j+2(\dim V-\dim G)}(U\times^G Z),\]
where $U\times^G Z=(U\times Z)/G$,
provided the complex codimension of \(V-U\) in \(V\) is greater than \(i/2\)
and $\dim X-j/2$.

Notice that in our definition the homological degree is bounded from above by
\(2\dim Z\) and is not bounded from below. The main advantage
of using equivariant Borel-Moore homology 
is that we have a fundamental class \([Z]\in \hbm_{2d}^{G}(Z)\),
\(d=\dim Z\). In particular,  fundamental class \([pt]\in \hbm_{0}^{G}(pt)\) and
cap product provide an identification 
between
\(\hbm_*^{G}(pt)\) and \(H^*_G(pt)\).
%\todo[inline]{moved this up from the next paragraph}
Let us also notice that
\(\hbm_*^{G}(pt)\) and \(H^*_{G}(pt)\) both
have a ring structure and the above mentioned identification of both
spaces respects the ring structure. In particular, we fix notation for the ring:
\[S=\hbm_*^{\That}(pt)=H^*_{\That}(pt)=\mathrm{Sym}(\fthat^*).\]

Thus for any \(X\) with a
\(\That\)-action, the spaces \(\hbm_*^{\That}(X)\) and \(H^*_{\That}(X)\) are
naturally \(S\)-modules and the natural
pairing between these two spaces is 
\(S\)-linear.

The equivariant Borel-Moore
homology of the affine flag variety is defined as the
direct limit
\[\hbm_*^{\That}(\Fl)=\lim_{\rightarrow} \hbm_*^{\That}(X_\bullet).\]
It has the structure of noncommutative ring with an explicit algebraic
presentation, called the nil Hecke algebra, $\Aaf$
\cite{kostant1986nil,lam2014book}.
The Schubert classes, $A_w\in \Aaf$ for $w\in\awg$ are defined
as the fundamental classes $[X_w]$ of the closures of the
Schubert cells $\Omega_w$ again using
Borel-Moore homology \cite{lam2008schubert}.

Since we define \(\Fl\) as inductive limit of finite-dimensional schemes
\(X_\bullet\),
it is natural to define the cohomology as inverse limit with respect to the pullback maps:

\[H^*_{\That}(\Fl)=\lim_{\leftarrow}H^*_{\That}(X_\bullet),\]
as graded modules, as described in the last paragraph of \cite{Graham01}.
%\todo[inline]{don't we want this subspace $\Lambda$ like in Kumar?}
Then \(H^*_{\That}(\Fl)\) is a module over the equivariant cohomology of the point
\({  S=\mathrm{Sym}(\hat{\mathfrak{t}}^*)}\),
which may be identified as a submodule
\begin{equation}
  \label{cohlambdaeq}
\Lambda\cong H^*_{\That}(\Fl) \subset \Hom_S\left(\hbm_*^{\That}(\Fl),S\right).
  \end{equation}
Then the affine Schubert polynomials may be defined as a dual
basis to $A_w$, see \cite{kumar2002kac,lam2008schubert,lam2014book}.
%\todo[inline]{was there some reference we should for kashiwara? wasn't sure}
We will denote by \(x_i\) the first Chern class \(c_1(\cL_i)\in H^*_{\hat{T}}(\Fl)\).
These classes, together with the pullback of
the equivariant cohomology of the affine Grassmannian,
generate the equivariant cohomology as an $S$-module,
with relations described in Section \ref{sec:nil-hecke-gkm}.

\subsection{The affine Springer fiber}
\label{sec:geometryASF}

Given an element \(\gamma\in\frg[t]\) the authors of \cite{KazhdanLusztig88} attach a subspace
\begin{equation}
\label{eq:defaffsp}
    \Sp_\gamma=\{gI : g^{-1}\gamma g 
    \in \Lie(I)\}\subset \Fl
\end{equation}
%The lattice \(L\subset T(\mathcal{F})\) consisting of elements commuting with \(\gamma\) naturally acts on \(\Sp_\gamma\). 
In the Grassmannian case we have
\begin{equation}
\label{eq:defaffspgr}
    \Sp^{\grsup}_\gamma=\{gG(\mathcal{O}) : g^{-1}\gamma g\in
    \mathfrak{g}(\mathcal{O})
    \}\subset \Gr,
\end{equation}
which is the image of \(\Sp_{\gamma}\) under the projection \(\Fl\to \agrass\).
A subgroup of \(T(\mathcal{F})\) 
consisting of elements commuting with \(\gamma\) naturally acts on \(\Sp_\gamma\) and $\Sp^{\grsup}_\gamma$.

The element \(\gamma\in\frg[t]\) is called homogeneous if \(\gamma(\torusrotelt^{-1}t)\) is conjugate to a scalar multiple of \(\gamma(t)\) for all \(\torusrotelt\in\CC^*\), see formula~\eqref{eq:conjD} for more details.
The homogeneous topologically nilpotent regular semi-simple elements were classified in \cite{OblomkovYun16} and the corresponding affine Springer fibers
have a natural \(\CC^*\)-action. Their homologies provide a geometric model for the representations
of the graded and rational Cherednik algebra of the corresponding type \cite{OblomkovYun16,VaragnoloVasserot10,Vasserot05}.
This paper deals only with the Springer theory in type \(A\), and we now recall the relevant results.

Let us denote by \(\gamma_{n,1}\in \frg[t]\) an element such that
\[\gamma_{n,1}(e_i)=e_{i+1},\quad i=0,\dots,n-2,\quad \gamma_{n,1}(e_{n-1})=te_0.\]
This element is homogeneous and regular semi-simple,
as is the element \(\gamma_{n,m}=\gamma_{n,1}^m\) for \(m>0\). If $(n,m)$ are coprime,
then the affine Springer fiber
\(\Sp_{n,m}=\Sp_{\gamma_{n,m}}\) is a projective variety,
which was first studied in \cite{LusztigSmelt91}.
Let \(j:\Sp_{n,m}\rightarrow \Fl\) be the inclusion map.

The full torus \(\hat{T}\) does not preserve the Springer fiber. %the variety \(\Sp_{n,m}\)
%but the one-dimensional subtorus \(U=\CC^*\), \(\phi: U\rightarrow \hat{T}\)  preserves
%it. 
%Indeed, let us fix notation for a diagonal matrix
However, given $s\in \C^*$, let
\(D(s)=\mathrm{diag}(s,s^2,\dots,s^n)/s^{(n+1)/2}.\) Then one can check that
\begin{equation}\label{eq:conjD}
  \torusrotelt^{1/n}D(\torusrotelt^{-1/n})\gamma_{n,1}(t) D(\torusrotelt^{1/n})=\gamma_{n,1}(\torusrotelt t).
\end{equation}
Thus the torus $U\subset \That$ preserves 
\(\gamma_{m,n}=(\gamma_{1,n})^m\)
up to a scalar multiple, where $U$ is the image of the map
\begin{equation}
  \label{utoruseq}
  \phi: \C^*\to T\times \CC^*=\That,\quad
  \phi(\torusrotelt)=(D(\torusrotelt^{-1/n}),\torusrotelt).
  \end{equation}

  As in \cite{OblomkovYun16} one needs to pass to the \(n\)-fold unramified cover \(U^{[n]}\) of \(U\) to work with the fractional powers in the last formula. The multiplication by
  \(n\) yields an isomorphism between \(H^*_U(pt)\) and \(H^*_{U^{[n]}}(pt)\) and we
  assume this isomorphism for the rest of the paper.

  In the paper \cite{OblomkovYun16} the Springer fiber \(\tilde{\mathcal{S}}_{n,m}\)
  is defined as \(\mathcal{S}_{\tilde{\gamma}_{n,m}}\) where
 \[\tilde{\gamma}_{n,m}(e_i)=e_{i+1},\quad i=0,\dots, n-2,\quad \tilde{\gamma}_{n,m}(e_{n-1})=e_0t^m.\]

 The element \(\tilde{\gamma}_{n,m}\) is conjugate to \(\gamma_{n,m}\). In the case that is most important for our results \(m=n+1\) and 
 we have \(D(t)\tilde{\gamma}_{n,n+1}(t)D(t)^{-1}=\gamma_{n,n+1}(t)\). The last formula together \eqref{eq:conjD} implies
 \[\torusrotelt^{(n+1)/n}D(\torusrotelt^{-(n+1)/n})\tilde{\gamma}_{n,n+1}(t) D(\torusrotelt^{(n+1)/n})=\tilde{\gamma}_{n,n+1}(\torusrotelt t)\]
 and that is exactly the \(\CC^*\) used
 in \cite{OblomkovYun16}. A similar argument is available for any \(m\) and thus the results from \cite{OblomkovYun16} apply in the setting of the current paper.

We fix our conventions by setting \(H^*_U(pt)=\C[\epsilon]\).
  %Strictly speaking, the map $\phi$ is not quite defined because the exponents may be fractions, but this has no effect, and the normalization
 % $\delta=\epsilon$ is preferred for Cherednik
 % algebras.
  Since $\tilde{\mathcal{F}}l^U=\tilde{\mathcal{F}}l^{\That}$,
  the fixed point
  set $\Sp_{n,m}^U$ is naturally a subset of
  $\tS_n$.
  This set is denoted \(\Res(n,m)\), and has an
  explicit description
given in Section~\ref{sec:affine-permutations}.

It was shown in \cite{LusztigSmelt91} that \(\Sp_{n,m}\cap X^\circ_w\), \(w\in \Res(n,m)\) is an affine space of dimension \(\dim_m(w)\ge 0\), where
\(\dim_m\) is
the combinatorial function defined in \eqref{eq:dimcell}.
Respectively, we denote by \(Y_w\) the closure of the intersection $Y^\circ_w=\Sp_{n,m}\cap X_w^{\circ}$.
As in \cite{Graham01}, there is a well-defined fundamental class \([Y_w]\in H^{\That}_*(\Sp_{n,m})\).
Then we have the following proposition:
\begin{prop}
  \label{injsurjprop}
  For $\Sp=\Sp_{n,m}$ with $(n,m)$ coprime, we have
  \begin{enumerate}[a)]
  \item \label{injpart} The pushforward map $j_*: \hbm_*^U(\Sp)\otimes_{\CC[\epsilon]} \CC[\epsilon^{\pm 1}]\rightarrow \hbm_*^U(\Fl)\otimes_{\CC[\epsilon]} \CC[\epsilon^{\pm 1}]$ is
    injective.
  \item \label{surjpart} The restriction map $ j^*:H^*_{U}(\Fl)\otimes_{\CC[\epsilon]}\CC[\epsilon^{\pm 1}]\rightarrow H^*_U(\Sp)\otimes_{\CC[\epsilon]}\CC[\epsilon^{\pm 1}]$ is
    surjective.
  \item \label{injlocpart} The localization map 
    $i^*_{\Res(n,m)}:H^*_U(\Sp)\rightarrow 
    H^*_U(\Res(n,m))$
    to the fixed point set is injective.
  \item \label{freegenpart}
   The equivariant Borel-Moore homology is freely generated
    over $\gf[\epsilon]$ by the fundamental classes
    $[Y_w] \in \hbm_*^U(\Sp)$ for $w\in \Res(n,m)$.
  \item \label{dualpart}
%    There is a dual Schubert basis \([X^u]\)
    The equivariant cohomology is freely generated
    by dual elements $[Y^w]\in H_U^*(\Sp)$, %which are dual to the
    such that the pairing of $[Y_v]$ with $[Y^w]$ is the delta
    function $\delta_{v,w}$.
  \end{enumerate}
\end{prop}
\begin{proof}
  Part \ref{surjpart}) is proven in
  \cite{OblomkovYun16,OblomkovYun17}.
  Parts \ref{freegenpart})
  and \ref{dualpart})
  follow from the formality theorem for cohomology
  \cite{GoreskyKottwitzMacPherson97},
  and the formality of the homology \cite{Graham01}, Proposition 2.1.
  Part \ref{injpart}) follows from parts
  \ref{surjpart}) and \ref{freegenpart}),
  and part \ref{injlocpart}) follows from \cite{EdidinGraham98},
  Proposition 6.
\end{proof}

\subsection{Action of the Cherednik algebra}
\label{sec:cher-algebra-acti}

Let us recall the definition of the graded Cherednik algebra \(\grH\). As a \(\C\)-vector space,
\[\grH=\C[u,\delta]\otimes\Sym(\mathfrak{t}^*)\otimes\C[\affperms],\]
with grading given by
\[\deg \tilde{w}=0,\quad \tilde{w}\in \affperms, \]
\[\deg(u)=\deg(\delta)=\deg(\xi)=2,\quad \xi\in\mathfrak{t}^*.\]

Let us fix notation \(\fthat^*=\frt^*\oplus \langle \delta\rangle\) and a section of the projection \eqref{eq:proj-roots}:
\begin{equation}\label{eq:sec-roots}
{ \lambda_i=\eta_i-\delta/n,\quad i=1,\dots,n.}
\end{equation}
The algebra structure is defined by the \(W\)-action from \eqref{wact} and the relations:

\begin{enumerate}
\item \(u\) is central.
\item \(\C[\affperms]\) and \(\Sym(\fthat^*)\) are subalgebras
\item { \(s_i\xi-s_i(\xi) s_i=\langle\xi,\alpha_i\rangle u\), \(\xi\in \fthat^*\), \(i=0,\dots,n-1\)}.
\end{enumerate}

The element \(\delta\in \fthat^*\) is also central, and thus for \(\nu\in\C\)
we can define an algebra
\[\grH_\nu=\grH/(u+\nu\delta).\]
This is the {\it the graded Cherednik algebra with the central charge} \(\nu\). We set the image of \(\delta=-u/\nu\) to be \(\epsilon\).
If we specialize \(\epsilon\) to \(1\) we obtain the algebra \(\grH_{\nu,\epsilon=1}\) which is the trigonometric algebra in the literature.

The rational Cherednik algebra was introduced in \cite{EtingofGinzburg02}, the  rational Cherednik algebra is
a doubly graded degeneration DAHA introduced by Cherednik \cite{Cherednik95}. The trigonometric Cherednik algebra \(\grH_{\nu,\epsilon=1}\) plays intermediate role, it is a
single graded degeneration of DAHA and the trigonometric algebra
degenerate to the rational Cherednik algebra, for further discussion
of the degenerations one can consult for example Section 7 of \cite{Oblomkov04}.
In the current paper we use notations and conventions of the paper
\cite{OblomkovYun16}.

The subalgebra \(\C[\epsilon]\otimes\C[\affperms]\) has a trivial representation and the induced representation
\[\mathrm{Ind}_{\C[\epsilon]\otimes\C[\affperms]}^{\grH_\nu}(\C[\epsilon])=\C[\epsilon]\otimes\Sym(\mathfrak{t}^*),\]
is called {\it polynomial representation} of \(\grH_\nu\).
The subalgebra \(\Sym(\frt^*)\) acts  by  multiplication on this representation. On the other hand there is
a standard action of \(\affperms\) on \(\C[\epsilon]\otimes\Sym(\frt^*)=
\Sym(\hat{\frt}^*)\)
given by \eqref{wact}.
The action of \(\C[\affperms]\subset \grH_\nu\) is a deformation of the standard action, the generator \(s_i\), \(i\in\{0,\dots,{n-1}\}\) acts
by the (right) operator
\begin{equation}
  \label{geomodrighteq}
 {  s_i+\nu\epsilon\frac{1-s_i}{\lambda_i-\lambda_{i+1}}.}
  \end{equation}

  The polynomial representation and its irreducible quotient has a geometric realization. Indeed, the equivariant Chern classes \(c_1(\mathcal{L}_i)\), \(i=1,\dots,n-1\)
  generate localized equivariant cohomology
  \(H^*_U(\Fl)\otimes \C(\epsilon)\), see Section 2.3 in \cite{BezrukavnikovFinkelberg08}. 
  Hence there is a natural isomorphism
  \(H^*_{U,\epsilon=1}(\Fl)=\Sym(\frt^*)\).
  %\todo[inline]{Which discussion says this? Also, shouldn't it be $\frt^*$, without the hat,
 %   which is like the polynomials in the $x$-variables?}
  Under this identification
\(H^*_{U,\epsilon=1}(\Fl)\) acquires structure of \(\grH_{m/n,\epsilon=1}\)-module. Respectively, \(H^*_U(\Fl)\) becomes an \(\grH_{m/n}\)-module.
The embedding \(j:\Sp_{n,m}\rightarrow \Fl\) induces the pullback map between the cohomology groups. This map was studied in \cite{OblomkovYun16}:

\begin{thm}\cite{OblomkovYun16}
  \label{oythm}
  For any coprime $(n,m)$ we have
  \begin{enumerate}[a)]
 % \item The restriction map: \(j^*: H^*_U(\Fl)\otimes \CC(\epsilon)\rightarrow H^*_U(\Sp_{n,m})\otimes \CC(\epsilon)\) is surjective.
  \item The kernel of \(j^*\) is preserved by \(\grH_{m/n}\), i.e. \(j^*\) is a homomorphism of
    \(\grH_{m/n}\)-modules.
  \item The equivariant cohomology at \(H^*_{U,\epsilon=1}(\Sp_{n,m})\) is the unique irreducible finite dimensional
    $\grH_{m/n,\epsilon=1}$-module
%    \todo[inline]{missing ``-module'', right?}
    \(\cL_{m/n}(triv).\)%\footnote{It is well-known that \(\grH_{m/n,\epsilon=1}\) has a unique irreducible finite-dimensional representation.} 
  \end{enumerate}
  
\end{thm}

\section{Affine Schubert calculus}
\label{sec:affine-schub-calc}
We review some background on affine Schubert calculus, for which
we refer to Goresky, Kottwitz, and MacPherson \cite{GoreskyKottwitzMacPherson97},
as well as Lam \cite{lam2008schubert},
Kostant and Kumar \cite{kostant1986nil},
and the book of Lam, Lapointe, Morse,
Schilling, Shimozono, and Zacbrocki \cite{lam2014book}.
We follow the descriptions of the latter.

\subsection{The nil Hecke and GKM rings}\label{sec:nil-hecke-gkm}
Let 
\[{ S=\mathrm{Sym}(\fthat^*),\quad F=\mathrm{Frac}(S),}\]
and consider the noncommutative algebra
\[F_{\awg}=\bigoplus_{w \in \awg} F w,\]
with product given by
\[(fu)(gv)=f\cdot u(g) uv,\]
where $f,g \in F$, and the action of $\awg$ on $F$
is determined by equation \eqref{wact}.
The inclusion $W\hookrightarrow F\cdot W$
determines a left and right action of $W$
on $F_W$, in such a way that the left action acts internally on the ground ring. We similarly have a conjugation action by all extended permutations.

For any $i\in \{0,...,n-1\}$, let
\begin{equation}
  { \label{Adefeq}
  A_i=\frac{1}{\alpha_i}(1-s_i).}
  \end{equation}
These operators satisfy the braid relations in type $A$,
and so one may define
\[A_w=A_{i_1}\cdots A_{i_k}\]
whenever $w=s_{i_1}\cdots s_{i_k}$ is a reduced word.

\begin{defn}
  The subring generated by the $A_i$ and $S\subset F$ is
called the \emph{affine nil Hecke algebra}, 
denoted by $\Aaf$.
It is graded by assigning the elements of 
$\fthat^*$ degree 1, and letting the degree
of $w$ be zero, so that
$A_w$ has degree $-l(w)$.
Respectively, there are two variants of the object dual to \(\Aaf\):
\[\hat{\Lambda} =\left\{f \in \Hom_F(F_W,F):
    f({A_w})\in S\right\},\]
\[\Lambda=\left\{f\in \hat{\Lambda}:
    f(A_w)=0 \mbox{ for all but finitely many $w$}\right\}.\]
\end{defn}

The \(S\)-module \(\Lambda\) is actually
an $S$-algebra with respect to the (commutative) product of pointwise multiplication
\[(fg)(w):=f(w)g(w), \quad f,g\in \Lambda,\quad w\in W. \]
Respectively, \(\Aaf\) has a natural
\(\Lambda\)-action:
\[f\cdot \sum_w c_ww=\sum_wf(w)c_w w,\quad f\in \Lambda,\quad \sum_w c_ww\in \Aaf.\]
%The corresponding action on $F_W$ preserves $\Aaf$, making $\Aaf$ into a
%$\Lambda$-module.
Also, $\Lambda$ is a free $S$-module with basis
\[\xi^v(A_u)=\delta_{u,v}.\]

We also have particular elements
$\underline{x}_i\in \Lambda$ for  all  \(i\), such that, \(\underline{x}_{i+n}=\underline{x}_i-\delta\)  and these elements are given by
\[\underline{x}_i(w)=w(\lambda_i)=\lambda_{w_i}\in S.\]
where $\lambda_i$ are as in \eqref{eq:defweights},
whose action on $\aaf$ is given by diagonal multiplication by $w(\lambda_i)$ in the fixed point basis. The left and right \(W\)-actions  on \(\Lambda\) defined to satisfy
relations
\[(w\cdot f)(w\cdot a)=f(a)=(f\cdot w)(a\cdot w), \quad f\in \Lambda, a\in \Aaf, w\in W.\]
The left and right actions of $W$
preserve both $\aaf$ and $\Lambda$, and are related to $\underline{x}_i$ by $w \underline{x}_i w^{-1}=x_{w_i}$.  Let us also notice that
\(\underline{x}_1+\dots+\underline{x}_n=\delta\).

The classes $A_\sigma$ for $\sigma \in S_n \subset W$ span a subalgebra $\mathbb{A}\subset \aaf$ corresponding to the classical, non-affine algebra.
We have an element
\begin{equation}
\label{eq:deltahat}
\widehat{\Delta}_n=
{ \frac{1}{\prod_{i<j} (\lambda_i-\lambda_j)}
\sum_{\sigma \in S_n} \signrep(\sigma) \sigma}
\end{equation}
which agrees with $A_{w_0}$,
where $w_0=(n,...,1)\in S_n$ is the maximal length permutation.

Then the elements $\xi^v$ in $\Lambda$ satisfy
  \begin{equation*}
    \label{schubdef}
    \partial_i \xi^{v}=\begin{cases} \xi^{v s_i} & l(vs_i)<l(v)
      \\ 0 & \mbox{otherwise} \end{cases}
    ,\quad
\Ad_\rotelt(\xi^v )=\xi^{\rotelt w \rotelt^{-1}},
\end{equation*}
where $\partial_i:\Lambda\to \Lambda$ is the BGG  operator 
\begin{equation}\label{eq:DMop}
  \partial_i(f)={ \frac{f-f\cdot s_i}{\underline{x}_i-\underline{x}_{i+1}}.}
\end{equation}
In fact, they are determined uniquely by $\xi^{1}=1$, and
either the first relation, or the second equation combined with
the first for $i\neq 0$ (see \cite{Billey1997KostantPA}).
Let us also remark that 
\(\xi^v\) are polynomials of \(\underline{x}_i\).

%As explained in Section \ref{sec:geometry}, we have
We have the following presentation, due to Kostant and Kumar:
\begin{prop} (Kostant, Kumar \cite{kostant1986nil})
  \label{prop:kostantkumar}
We have isomorphisms of graded $S$-modules
\begin{equation}
  \label{lambdafleq}
  \hbm^{\That}_*(\Fl)\cong \Aaf,\quad H_{\That}^*(\Fl)\cong \Lambda,
  \end{equation}
in which the Schubert cycles $[X_w]$ map to $A_w$, the
dual classes $[X^w]$ in cohomology map to $\xi^{v}$. The pointwise multiplication on $\Lambda$ agrees with the ring structure in equivariant cohomology, and the pairing between
homology and cohomology agrees with the pairing between $\Aaf$
and $\Lambda$. The $\underline{x}_i$ correspond to 
the Chern classes of the tautological line bundles $\underline{x}_i=c_1(\mathcal{L}_i)$.
\end{prop}

\subsection{The nonequivariant limit}

The affine nil Coxeter algebra $\aafo$
is the subalgebra of $\aaf$
generated by $A_w$ over $\C$, but not the nonconstant elements of $S$. 
It is noncommutative, and the relations are given by
\begin{equation}
  \label{eq:nilcox}
  A_u A_v = \begin{cases}
    A_{uv} & l(uv)=l(u)+l(v), \\
    0 & \mbox{otherwise.} \end{cases}
\end{equation}
Then 
$\aafo\cong\aaf \otimes_S \C= \aaf$ where $\C$ is the $S$-module on which the maximal ideal acts by zero. By equivariant formality, we have
\[\hbm_*(\Fl)\cong \aafo,\ \ 
H^*(\Fl)\cong \Lambda_0.\]
where $\Lambda_0=\Lambda \otimes_S \C$.  We also use the notation \(x_i\) for the non-equivariant limit of \(\underline{x}_i\in \Lambda\). In particular, we have
\(x_1+\dots +x_n=0\)  and \(x_{i+n}=x_i\).

Let $\phi_0:S\rightarrow \C$ be the map which sends all $\lambda_i$ to zero, so that 
$\mathfrak{m}=\ker(\phi_0)$ 
is the maximal ideal of $S$.
Then the map which ``forgets'' equivariance is given by
$\phi_0:\aaf\rightarrow \aafo$ given by
\[\phi_0:\sum_{w} a_w A_w\mapsto \sum_{w} \phi_0(a_w) A_w,\]
and similarly for 
$\Lambda\rightarrow \Lambda_0\cong \Lambda/\mathfrak{m} \Lambda$, which is a ring homomorphism.

%More generally, we can define 

%\begin{defn}
%  If $H\subset \That$ is a subtorus and \(\phi_T:\Sym(\fthat^*)\to\Sym(\hat{\mathfrak{h}^*})\), 
%  we similarly denote the restrictions by $\aaf^T,\Lambda_T,\phi_T$.
%\end{defn}

Following Lam \cite{lam2008schubert}, 
call a word $i_1\cdots i_k$ in the symbols 
$i_j\in \Z/n\Z$ \emph{cyclically decreasing}
if each letter appears at most once,
and we have that $i+1$ always precedes $i$
whenever both letters appear.
We say that $w\in W$ is cyclically decreasing
if there is some reduced word $w=s_{i_1}\cdots s_{i_k}$ for which $i_1\cdots i_k$ is cyclically decreasing. We denote the set of cyclically decreasing affine permutations  by $\cycdec{n}$.

For $0\leq k \leq n-1$, define
\begin{equation}
\label{eq:lamelts}
h_k=\sum_{w \in \cycdec{n}:\inv(w)=k} A_w \in \aafo
\end{equation}
%where the sum is over cyclically decreasing
%affine permutations $w\in W$ with length $\inv(w)=k$. 
The $h_k$ generate a commutative subalgebra of $\aafo$
called the \emph{Stanley-Fomin subalgebra}.
The algebra $\Lambda_{(n-1)}=\C[h_1,...,h_n]$
is the ring which contains the $k$-Schur functions \cite{lam2014book}.
Notice that $h_0$ acts by the identity, and so is not included as a generator.

The algebra \(\Lambda_{(n-1)}\) is naturally isomorphic to
the homology  algebra \(\hbm_*(\Gr)\) of the affine Grassmannian, as defined by Bott \cite{Bott58}. The projection map \(\pi: \Fl\to \Gr\)  is
a smooth map with fibers \(\gFl\) and \(\hbm_*(\Fl)=\hbm_*(\Gr)\otimes \hbm_*(\gFl)\).
The cohomology classes \(x_i\in \Lambda_0\) become the
Chern classes of the tautological line bundles of \(\gFl\) in the above product.
The homology \(\hbm_*(\Fl)\) are generated from the fundamental
class by cap product operations with elements of \(H^*(\Fl)\).
The fundamental class of a fiber of \(\pi\) is equal to
$\Delta_n=A_{w_0}\in \aafo$, where
$w_0=(n,...,1)\in S_n$ is the maximal length element. Then we have
\begin{prop}
\label{prop:afsiso}    
The action of left multiplication by $h_k$
on $\aafo$ commutes with multiplication by 
Chern classes $x_i$.
We have an isomorphism
\begin{equation}
\label{eq:afsiso}
   R_n(\bx)\otimes \Lambda_{(n-1)}
     \cong \aafo
\end{equation}
of modules over $\C[\bx]\otimes \Lambda_{(n-1)}$,
in which $1\otimes 1$ is sent to $\Delta_n$.
\end{prop}

\subsection{The nil Hecke Algebra and the affine Springer fiber}  

\label{sec:springact}
We represent the (equivariant) Borel-Moore homology of 
the affine Springer fiber 
in terms of the nil Hecke algebra.

Fix coprime $(n,m)$, and consider the subtorus $\Tres\cong \C^*\subset \That$
from \eqref{utoruseq}, which preserves $\Sp_{n,m}$. The
corresponding evaluation map $\fthat^*\rightarrow \affsptorus^*$
is given by
\begin{equation}
\label{sptorus}
{ \lambda_i\mapsto\left(\frac{n-1-2i}{2n}\right)\epsilon,\quad \delta\mapsto\epsilon,}
\end{equation}
where $\epsilon \in \affsptorus^*$ is the equivariant parameter, which determines a ring homomorphism \(\Sym(\fthat^*)\to \CC[\epsilon]\).
As explained in Section \ref{sec:geometryASF},
we have the same fixed point sets
\(\Flatletter^{\hat{T}}=\Flatletter^U=W \), and
the evaluation map $\hbm_*^{\That}(\Fl)
\rightarrow \hbm^U_*(\Fl)$ can be realized by taking
\eqref{sptorus} in the fixed point basis, as the denominators will never vanish.

Thus let us introduce the related specialized 
\(\CC[\epsilon]\)-modules:
\[
  \Aaf^U=\Aaf\otimes_S \C[\epsilon],\quad \Lambda_U=\Lambda \otimes_S \C[\epsilon],\]
and observe that \(\Aaf^U\) is naturally a \(\CC[\epsilon]\) submodule of \(\bigoplus_{w\in W}\CC[\epsilon^{\pm 1}] w\),
which is isomorphic to $\hbm_*^U(\Fl)$.
As a \(\CC[\epsilon]\)-module, \(\Aaf^U\) is isomorphic to \(\Aaf^0\otimes \CC[\epsilon]\), 
and the algebra morphism
\(\Aaf^U\to \Aaf^0\) that evaluates the coefficients in the
$A_w$ basis at \(\epsilon=0\) is compatible with the
specialization $\hbm_*^U(\Fl)\rightarrow \hbm_*(\Fl)$.

Next recall that $\Res(n,m)\subset W$ correspond to the
fixed point set $\Sp_{n,m}^U$. we define an ideal $I_{n,m}\subset \Lambda_U$ as the kernel of 
the inclusion map
\[i_{\Res(n,m)}^* : \Lambda_U \rightarrow \bigoplus_{u \in \Res(n,m)} \gf[\epsilon] u,\]
where the coefficient of $f$ is the evaluation of $f(u)\in \gf[\epsilon]$.
Since \(\Lambda_U=H^*_U(\Fl)\), we have a geometric interpretation for the  quotient:
\[\Lambda_U/I_{n,m}= j^*(H^*_U(\Fl)).\]
Respectively, we have a dual object inside \(\Aaf^U=\hbm_*^U(\Fl)\) is defined by
\[\Saf=\left\{c \in \Aaf^U: f\in I_{n,m}\Rightarrow f(c)=0\right\}.\]
%the dual space to $I_{n,m}$.
%This subspace agrees with the equivariant
%Borel-Moore homology under \eqref{lambdafleq}. We have

\begin{prop}
  \label{prop:spbasis}
  We have that ${ \Saf \cong j_*(\hbm_*^U(\Sp_{n,m}))\subset \hbm_*^U(\Fl)}$
  is the image under the
  % localization
  inclusion map. 
The image of the classes $[\spclass_w]$ determine elements
\begin{equation}
\label{eq:defbclass}    
B_w =\epsilon^{-d_w}\sum_{v\in \Res(n,m)} c_{v,w} v \in \Saf
\end{equation}
  for each $w\in \Res(n,m)$ satisfying:
  \begin{enumerate}[a)]
  %\item \label{cvpropfree} They freely generate
   % $\Saf$ as a $\C[\epsilon]$-module.
  \item \label{cvpropnz} The coefficients are rational numbers satisfying
    \[b_{v,w}\neq 0 \Rightarrow v \leq_{bru} w,\ \ b_{w,w}\neq 0.\]
\item \label{cvpropdeg} The degrees are given by $d_w=\dim_m(w)$ as defined in \eqref{eq:dimcell}.
\item \label{item:classsp} For $w\in S_n$, we have that $B_{w}$ is the evaluation of $A_w$ under \eqref{sptorus}.
  \end{enumerate}  
\end{prop}

In particular, taking $w=w_0$,
we have an element
\begin{equation}
\label{eq:deltatilde}    
\tilde{\Delta}_n=\frac{c}{\epsilon^{n(n-1)/2}} 
\sum_{\sigma \in S_n} \signrep(\sigma)\sigma
\in \Saf
\end{equation}
for $c$ a constant, coming from the specialization of $\widehat{\Delta}_n$ from \eqref{eq:deltahat}.
More generally,
the coefficients of the elements $B_w$ can be calculated for $w\in S_n$
by Billey's formula \cite{Billey1997KostantPA}, 
but for other elements
$w\in \Res(n,m)$, it is not even clear which coefficients are nonzero.

The affine Weyl group action on the homology of affine Springer fibers fibers was introduced by Lusztig \cite{Lusztig96}. This action was studied further by many authors, for
a detailed treatment of the relevant of the Demazure-Lusztig operators for this action see \cite{ChrisGinzburg}. 
The relation between the action graded Cherednik algebra  on the homology of the affine flag variety  and on the homology of the homogeneous affine Springer fiber is discussed in 
\cite{OblomkovYun16}. Below we give a purely algebraic proof of a variant of the corresponding statement from \cite{OblomkovYun16}:
%and
%are implicit in \cite{Vasserot05}:
\begin{prop}
  \label{spactprop}
  The Demazure-Lusztig operators (see \eqref{geomodrighteq}, \eqref{eq:DMop}):
  \begin{equation}
    \label{spacteq}
    f*_m s_i= f\cdot s_i+\nu\epsilon \partial_if,\quad \nu=m/n,
    \end{equation}
for ${ 0} \leq i \leq n$ define a right action of $\awg$ on $\Lambda_U$.
These operators, as well as conjugation by $\rotelt$,
preserve $I_{n,m}$, and hence the dual actions preserve
$\Saf \subset \Aaf^U$. 
In particular, the non-equivariant right action of $\awg$ and
$\rotelt$ preserves the subspace $\Saf\otimes_S \C \cong j_*(\hbm_*(\Sp_{n,m}))$.
\end{prop}
%\todo[inline]{forgot, do we cite vasserot for these too?}
\begin{proof}
  First, note that the conjugation action of $\rotelt$
  preserves the kernel of the evaluation map given by equation
  \eqref{sptorus}, and so at least acts on $\Aaf^U$.
  It preserves the kernel simply because
  conjugation by $\rotelt$ preserves the subset $\Res(n,m) \subset \awg$.

  The statement about the modified operators are due to Oblomkov
  and Yun \cite{OblomkovYun16,OblomkovYun17}, but we give a simple
  algebraic proof in our case: as elements of 
$F_{\awg}$, we have  
\begin{equation}
  \label{spactfawg}
  w *_m s_i= \left(\frac{m}{w_{i+1}-w_i}\right)w+
  \left(1-\frac{m}{w_{i+1}-w_{i}}\right)ws_i.
\end{equation}
Notice that this produces a $2\times 2$ matrix that squares to the identity.
From this, we see that the coefficient of $ws_i$ is zero
if and only if $w_{i+1}-w_i=m$. It is straightforward to see that
if $w\in \Res(n,m)$, then
\begin{equation}
  \label{wsinonzero}
  w_{i+1}-w_i=m\Leftrightarrow ws_i \notin \Res(n,m).
  \end{equation}
Therefore the reflection operators preserve the span of
$\Res(n,m)\subset F_{\awg}$, and hence the dual reflection operators
preserve $I_{n,m}$.

The statement that this defines an action of $\awg$
can also be proved algebraically.
\end{proof}

\section{Double Coinvariants}
\label{doubsec}
In this section we will state and prove our main results.

\subsection{Commuting variables}

We define an action of $DR_n$ 
on $\hbm_*(\Sp_{n,n+1})$.

\begin{defn}  
\label{def:xyaction}
    Define $\C[\epsilon]$-linear maps $\tilde{x}_i,\tilde{y}_i : \Aaf^U \rightarrow \Aaf^U$ 
    where $\tilde{x}_i$ is multiplication by the Chern class
\begin{equation}
\label{ydualeq}
  \tilde{x}_i\cdot w=(c_{w_i}\epsilon) w,\ \ 
  { c_i=\frac{n-1-2i}{2n}},
\end{equation}
under the restricted torus action 
\eqref{sptorus},
    and
\begin{equation}
    \label{eq:zidef}
    \tilde{y}_i=\tilde{z}_i-1,\ \ 
    \tilde{z}_i(f)=\Ad_{\rho^{-1}}(f) *_{n+1} (\rho^{-1} \shiftelt_i).
\end{equation}
We have the induced operators $x_i,y_i,z_i$ on 
$\hbm_*(\Fl)\cong \aafo$.
  \end{defn}

\begin{lemma}
\label{lem:actionform}    
Under the isomorphism
$\hbm_*(\Fl)\cong \Lambda_{(n-1)} \otimes R_n(\bx)$,
%from equation \eqref{noneqeq}, 
the map $x_i$ is given by usual multiplication by $x_i$, and 
\begin{equation}
  \label{eq:yimul}
  y_i(f)=(x_i h_1+\cdots+x_i^{n-1} h_{n-1})f.
\end{equation}    
\end{lemma}

\begin{proof}

We first check that $z_i$ (and therefore $y_i)$ commutes with 
the operators $x_j$ and $h_k$:
since the right action of $W$ satisfies
$(\_ \cdot w)x_i=x_{w_i} (\_ \cdot w)$, and ${ x_i=x_{i+n}}$,
we find that $z_i$ commutes with $x_i$.
We can also see that $\Ad_{\rho^{-1}}$ commutes with $h_k$ since it preserves the cyclically decreasing condition, and since
$\Ad_{\rho^{-1}}(A_w)=A_{\rho^{-1} w \rho}$ 
for all $w$, noting that conjugation by $\rho^{-1}$ preserves the Bruhat order. The right multiplication by $\rho^{-1} \psi_i$ commutes with $h_k$ since $h_k$ is defined as a left multiplication.

Let $z'_i$ denote the expression on the right hand side of \eqref{eq:yimul} plus $f$, so that we are proving $z'_i=z_i$. Since
$z'_i$ commutes with $x_j$ and $h_k$ as well
by Proposition \ref{prop:afsiso}, it suffices to check that they take the same values on the generator, $z'_i \Delta_n=z_i \Delta_n$.
Using the rule that $wz_i=z_{w_i}w$ and similarly for $z'_i$, it suffices to check this equation for $i=n$. 

In this case we have
\[z_n \Delta_n=
\Ad_{\rho^{-1}} (\Delta_n(\psi_n \rho^{-1}))=
(-1)^{n-1}\Ad_{\rho^{-1}} (\Delta_n)=
(-1)^{n-1}A_{\rho^{-1} w_0 \rho},\]
noting that $\psi_n\rho^{-1}=s_{n-1}\cdots s_1 \in S_n$, which acts on $\Delta_n$ by multiplying by the sign, which can be seen in terms of fixed points \eqref{eq:deltahat}. 
It therefore remains to show that
$z'_n \Delta_n=(-1)^{n-1}A_{\rho^{-1} w \rho}$.
For this, we claim that
\begin{equation}
\label{eq:xwrules}    
h_k x_n^k \Delta_n=
(-1)^k(A_{w^{(k)}}-A_{w^{(k-1)}}),\ \ 
w^{(k)}=(s_{k-1}\cdots s_{0})
( s_{k}\cdots s_1) w_0.
\end{equation}
Assuming this holds, we can insert 
\eqref{eq:xwrules} into the expression \eqref{eq:yimul}
defining $z'_n \Delta_n$. The result would then
cancel in pairs, leaving
$A_{w^{(n-1)}}=(-1)^{n-1} A_{\rho^{-1} w_0 \rho}$,
completing the proof.

To prove \eqref{eq:xwrules}, we first claim that
\begin{equation}
\label{eq:monkstep}
x_n^k \Delta_n=(-1)^k A_{s_{k}\cdots s_1 w_0}.
\end{equation}
This can be checked using the usual (non-affine) Monk rule \cite{monk1959geometry}, which in this context says that
\begin{equation}
    \label{eq:monkform}
(x_1+\cdots+x_r) A_{\sigma} =
\sum_{\substack{1\leq i \leq r<j\leq n \\
\inv(w)=\inv(\sigma t_{i,j})-1}}
A_{wt_{i,j}}.   
\end{equation}
Then since $x_1+\cdots +x_n$ acts by zero, we find that $x_n A_\sigma=-(x_1+\cdots+x_{n-1})A_\sigma$,
from which \eqref{eq:monkstep} follows by induction
and \eqref{eq:monkform} at $r=n-1$.

Finally, plugging \eqref{eq:monkstep} into the left side of \eqref{eq:xwrules} and using
 \eqref{eq:lamelts}, we find that
 \begin{equation}
\label{eq:lamstep}     
h_k x_n^k \Delta_n=
(-1)^k\sum_{\substack{w\in \cycdec{n}:\inv(w)=k}} 
A_w A_{s_k\cdots s_1 w_0}
\end{equation}
Then using the fact that $A_i A_{w_0}=0$ for $1\leq i \leq n$, we can check that the affine permutations
$w$ contributing to  \eqref{eq:lamstep} are
$s_{k-1}\cdots s_0$ and $s_{k-2}\cdots s_0 s_k$.

\end{proof}

We will also need the following fact,
which is Proposition 4.5. from Haiman \cite{Haiman02}:
\begin{prop}[Haiman]
\label{prop:haimaninjective}
Suppose that $f(\bx,\by) \in \mplusxy$. Then we have that
\begin{equation}
\label{eq:prophaiman}    
f(\bx,\by)\big|_{y_i=\lambda_{n-1} x_{i}^{n-1}+\cdots + \lambda_1 x_i}\in (e_1(\bx),...,e_n(\bx))
\end{equation}
as an ideal in $\C[\bx,\lambda_1,...,\lambda_{n-1}]$.
\end{prop}
\begin{remark}\label{rem:open-Haiman}
For our purposes, we note that
this has the same form as \eqref{eq:yimul} but with the symbols $\lambda_i$ in place of $h_i$.
%It is interesting that
%    In that setup,
%\eqref{eq:haimanchart}
Interestingly, this formula plays a very different role in Haiman's paper, in which the variables $\lambda_i$
%takes the form of a restriction map  from algebraic sections of the Procesi bundle to a certain open subset 
are the coordinates of an open chart
\[\left\{ (x^n,y-(\lambda_{n-1}x^{n-1}+\cdots+\lambda_1 x)\right\}\subset \hilbgeo^0_n \C^2\]
of the punctual Hilbert scheme of points in $\C^2$.
%The variables $\lambda_i$.
{
Notice also that these are the same as the coordinates
$c^{r,s}_{i,j}$ in the case where $\lambda$ is a strip,
as in Example \ref{ex:esstrip}.}

\end{remark}

We can now state our first theorem:  
\begin{thm}    
      \label{thma}
      The induced operators $x_i,y_j$ on $\hbm_*(\Fl)$
      commute, giving rise to an action of $\gf[\mathbf{x},\mathbf{y}]$.
Furthermore, this action satisfies the following properties:
      \begin{enumerate}[a)]
      \item \label{thmam} The elements of $\mplusxy$ act by zero,
        giving us an action of $DR_n$.
      \item \label{thmapres} The subspaces
        $j_*(\hbm_*(\Sp_{n,m})) \subset \hbm_*(\Fl)$
        are preserved, i.e. are submodules.
\item \label{thmaiso} 
  The map $DR_n\rightarrow j_*(\hbm_*(\Sp_{n,n+1}))$
  given by $f\mapsto f\cdot \Delta_n$ is an isomorphism.
\item \label{inj-homology} {  The inclusion  map \(j_*: \hbm_*(\Sp_{n,n+1})\to \hbm_*(\Fl)\) is injective.}
  
      \item \label{thmaact}
        There is an action of the extended affine Weyl group on $DR_n$ induced by the
        conjugation action, which is given by
\[ w x_i=x_{w_i}w,\quad   w y_i=y_{w_i} w,\quad
\sigma(1)=(-1)^{\sig},\quad \rotelt(1)=1+y_n,\]
where $w\in \extperms$ is any extended permutation, $\sig\in S_n$, and we have identified
the multiplication operators $x_{i+n}=x_i$, $y_{i+n}=y_i$.
\end{enumerate}
      \end{thm}

      \begin{proof}

It follows from \eqref{eq:yimul} that $x_i$ and $y_j$ commute.

To check part \ref{thmam}),
we must show that a non-constant multisymmetric
power sum, given by $p_{r,s}=x_1^ry_1^s+\cdots x_n^r y_n^s$
acts by zero. To do this, make the substitution \eqref{eq:yimul}
and collect the monomials in the $h_i$
to obtain a linear combination of $h_\mu$
with coefficients in $R_n(\bx)$.
Since the formula for $y_i$ begins in degree one, those coefficients have no constant term. Since the power sum is diagonally 
$S_n$-invariant and
the $S_n$ action preserves $h_\mu$, the coefficients must also be symmetric, and therefore zero.

Next, notice that the modified actions in \eqref{spacteq}
preserve $\Saf$, and all limit to the usual right action
modulo the relation $\epsilon=0$, so part
\ref{thmapres}) follows
from Proposition \ref{spactprop}.

For part \ref{thmaiso}), 
notice that 
$DR_n$ has dimension $(n+1)^{n-1}$ by \cite{Haiman02}, 
and 
$j_*(\hbm_*(\Sp_{n,m}))$ has dimension 
\emph{at most}
$(n+1)^{n-1}$ by
Proposition \ref{injsurjprop} items 
\ref{injpart}) and
\ref{freegenpart}).
It therefore suffices to show that the map
\begin{equation}
\label{eq:haimanchart}
    DR_n \rightarrow \Lambda_{(n-1)} \otimes R_n(\bx)
\end{equation}
determined by \eqref{eq:yimul} is an injection. 
This follows from 
Lemma \ref{lem:actionform} and
Proposition \ref{prop:haimaninjective}, substituting
$\lambda_i=h_i$.

Part \ref{inj-homology}) follows from the part \ref{thmaiso}) and the statement that \(\dim \hbm_*(\Sp_{n,n+1})=\dim DR_n\).
Finally, the relations in part \ref{thmaact}) hold equivariantly for the
modified actions,
and follow from definitions, as well as the twisting by the sign representation
in $R_n(\bx)\otimes \Lambda_{(n-1)}$.

\end{proof}

\subsection{Filtration by the descent order}
%coarsening of the Bruhat order by compositions}

We now describe a filtration on the homologies of the affine
flag variety and Springer fiber by compositions, which we
relate to the order on monomials in the $y$-variables
that produce the ``descent monomials'' described below.
For the rest of the paper, we will be concerned with the case $m=n+1$.

Recall the definition of the index 
$\indt(w)$ of an affine permutation from Definition \ref{def:ind}.
\begin{defn}
\label{def:sptau}    
Given a composition $\ba$
we define
$\hessleq(\ba)\subset \Sp_{n,n+1}$ to be the union of the cells
$Y^\circ_w$ where $w$ ranges over elements $w\in \Res(n,n+1)$
 which satisfy $\indt(w)\leq_{des} \ba$.
\end{defn}

The following lemma shows that 
 $\hessleq(\ba)$ is a closed subspace.
\begin{lemma}
  \label{desbrulebm}
        The descent order is compatible with the Bruhat order,
        \[u \leq_{bru} v \Rightarrow \indt(u) \leq_{des} \indt(v).\]
      \end{lemma}      
\begin{proof}
  First, consider the case $|\ba|=|\bb|$, where $\ba,\bb=\indt(u),\indt(v)$ so that
  $\min(u)=\min(v)$.
  Furthermore, by using $\rotelt$, we can see that it suffices to consider the
  case $\min(u)\cong 0\ (\moda n)$. In this case, $\indt$
  is the same as the composition corresponding to the left coset space
  in $S_n\backslash \awg$. It is known that $u \leq_{bru} v$ implies that
  $\ba \leq_{bru} \bb$, where the Bruhat order on compositions is the
  same as the order on the coset spaces by taking minimal representatives
  in $\extperms$ \cite{haglund2008nonsymmetric}.
  It follows immediately
  that $\ba \leq_{bru} \bb$ implies that $\ba \leq_{des} \bb$, proving this case.
  
  We then see that $u \leq_{bru} v$ implies
  that $|\ba|\leq |\bb|$, so it remains to consider the case $|\ba|<|\bb|$.
  Since $a\neq b$, we only need to prove that
  $\sort(\ba)\leq_{lex} \sort(\bb)$, as the tie breaking case in
  Definition \ref{desdef} will never come up.
  It is well known that
  \[u \leq_{bru} v \Rightarrow u' \leq_{bru} v'\]
  where $u',v'$ are the associated Grassmannian permutations, i.e.
  the permutations whose window notations have the same
  values as those of $u,v$, but in increasing order.
  Since $\indt(u')=\sort(\indt(u))$, it suffices to assume
  that $u,v$ are Grassmannian permutations.

  In the case of Grassmannian permutations, there is an explicit description
  of the Bruhat order in terms of
  the ``unit increasing monotone function'' $\Z\rightarrow \Z$ given
  by
  \[\rotelt_{w}(j)=\sum_{i=1}^n \max\left(0,\left\lceil \frac{j-w_i}{n}\right\rceil\right),\]
  see Theorem 6.3 of \cite{Bjrner1996AffinePO}. We make the following claim, which
  is straightforward to check using this description:
  given Grassmannian permutations with $u\leq_{bru} v$, if $u_1>v_1$, then
  there exists $i$ and $j$ such that
  $v_1 \leq j < i=u_1$ and
  $w=t_{i,j}u \leq_{bru} v$, where $t_{i,j}\in \awg$ is the affine transposition
  that exchanges $i$ and $j$. It follows easily that
  $\indt(w)_k \geq \indt(u)_k$ for all $k$, so of course we have
  $\indt(u) \leq_{des} \indt(w)$. But now inductively on 
  $|\bb|-|\ba|$,
  we may assume that $\indt(w)\leq_{des} \indt(v)$, proving that
  $\indt(u) \leq_{des} \indt(v)$. 
\end{proof}

%\[H\overline{\mathclap{ABC}}\]

We now define
\begin{defn}
  \label{fadef}
  For a composition $\ba$ of $n$, we 
define $F_{\ba} \hbm_*(\Sp_{n,n+1})$ to be the image of $\hbm_*(\hessleq(\ba))$ in 
$\hbm_*(\Sp_{n,n+1})$, and similarly for
$F_\ba \Saf \cong 
F_\ba \hbm^U_*(\Sp_{n,n+1}) \subset \Saf$.
\end{defn}

We will denote the associated graded component by\[G_\ba \hbm_*(\Sp_{n,n+1})=F_{\ba} \hbm_*(\Sp_{n,n+1})/F_{\ba'} \hbm_*(\Sp_{n,n+1}),\] 
where $\ba'<_{des}\ba$ is the largest element smaller than $\ba$.
It follows from Proposition \ref{prop:extstats} that $\indt(w)$ is always a descent composition for $w\in \Res(n,n+1)$,
so that
$G_{\ba}DR_n=\{0\}$ unless 
$\ba=\majt(\tau)$ for some $\tau$.

\begin{lemma}
\label{lem:geofilt}    
We have the following:
\begin{enumerate}[a)]
    \item \label{item:filtbaslem} The elements 
    $ [Y_w]\in \hbm_*(\Sp_{n,n+1})$
    for $w\in \Res(n,n+1)$ and $\indt(w)\leq_{des} \ba$
    are a vector space basis of $F_\ba \hbm(\Sp_{n,n+1})$. The corresponding equivariant classes freely generate $F_\ba \hbm^U_*(\Sp_{n,n+1})$ as a
$\C[\epsilon]$-module.
\item \label{item:geoinjlem} The map $\hbm_*(\hessleq(\ba))\rightarrow \hbm_*(\Sp_{n,n+1})$ is injective, and similarly in the equivariant case.
The kernel of the map 
$F_\ba \hbm^U_*(\Sp_{n,n+1})\rightarrow F_\ba \hbm_*(\Sp_{n,n+1})$ is $\epsilon F_\ba \hbm^U_*(\Sp_{n,n+1})$, and the corresponding map on the quotient is an isomorphism.
\item \label{item:chernfilt}
Each $F_{\ba} \hbm_*(\Sp_{n,n+1})$ is preserved by the action of Chern classes.
\item  \label{item:fplem} 
In the fixed point basis, we have 
\begin{equation*}
    \label{eq:saflem}
    F_\ba \Saf = \Saf \cap \bigoplus_{w:\indt(w)\leq \ba} \C[\epsilon^{\pm 1}] w.
\end{equation*}
\end{enumerate}

\end{lemma}

\begin{proof}
For any subset $A\subset \Res(n,n+1)$
which is an interval in the Bruhat
order, $w\in A, v\leq_{bru} w\Rightarrow v \in A$, the corresponding union of 
intersected Schubert cells is 
closed and paved by affine spaces.
It follows that both equivariant and nonequivariant Borel-Moore homologies are generated by the fundamental classes
$[Y_w]$ for $w\in A$, 
see \cite{Graham01}. Since the localization map is injective \cite{brion1997torus}, we have the injectivity
of part \ref{item:geoinjlem}), and also the statement of part \ref{item:filtbaslem}). Another way to see the injectivity is
to use the long exact sequence for BM homology of \(\Sp_{n,n+1}\)
and \(\hessleq(\ba)\) and the fact that all the spaces in our construction have no odd cohomology.
The kernel of the map in that item follows from Corollary 1 of the same reference.
The statement about Chern classes in part 
\ref{item:chernfilt}) follows since the Chern classes are pulled back from
$H^*_U(\Sp_{n,n+1})$ and $H^*(\Sp_{n,n+1})$.
Part \ref{item:fplem}) follows because if
\[f=\sum_{w} b_w(\epsilon) B_w=\sum_{w} a_w(\epsilon) w\]
and $w$ is a Bruhat-maximal element for which
$b_w(\epsilon)\neq 0$, then $a_w(\epsilon)\neq 0$.
\end{proof}

We can now state our second main result, which is Theorem B from the introduction.
\begin{thm}
  \label{thmb}
   Let $F_\ba DR_n$ be the image of $F_{\ba} \hbm_*(\Sp_{n,n+1})$ under the isomorphism $DR_n\cong \hbm_*(\Sp_{n,n+1})$
  from Theorem \ref{thma}, and let $G_\ba DR_n$ be the corresponding subquotient.
Then the following statements hold.
  \begin{enumerate}[a)]
  \item \label{item:thmfilt}
  We have that 
  \begin{equation}
      \label{eq:descfilt}
      F_\ba DR_n=\sum_{\ba' \leq_{des} \ba} \C[\bx] \by^{\ba'}\subset DR_n
  \end{equation}
  \item \label{item:thmbasis} If $\ba=\majt(\tau)$ for some $\tau\in S_n$, then the monomials
  \begin{equation}
\label{eq:basismons}      
\left\{y_1^{a_1}\cdots y_n^{a_n} x_{\tau_1}^{k_1}\cdots x_{\tau_n}^{k_n}:\bk \in \scheds(\tau)\right\}
  \end{equation}
  are a vector space basis of the quotient $G_\ba DR_n$. Otherwise, $G_\ba DR_n$ is the zero vector space.
      \item \label{item:thmquot} 
As a $\C[\bx]$-module, the quotient $G_{\ba} DR_n$ 
for $\ba=\majt(\tau)$
is isomorphic to the principal ideal $(f_\tau(\bx))\subset R_n(\bx)$, where
        \begin{equation}
        \label{eq:ftau}
f_\tau(\bx)=x_{\tau_1}\cdots x_{\tau_{n-l}}           \prod_{i=1}^n \prod_{j=i+\sched_i(\tau)+1}^{n} (x_{\tau_i}-x_{\tau_j}),
        \end{equation}
     and $l$ is the length of the final run of 
     $\tau$.
    \end{enumerate}   
\end{thm}
As a corollary, we have a basis of the anti-invariants
of $DR_n$ under a Young subgroup, and therefore an independent proof of the Shuffle Theorem.
 Recall that $\antisym_\mu$ is the anti-symmetrization operator with respect to a Young subgroup, given by
\begin{equation}
\label{eq:antisym}    
\antisym_\mu f(\bx,\by)=\sum_{\sigma \in S_\mu} \signrep(\sigma) f(\bx_\sigma,\by_\sigma)
\end{equation}
%first proved in \cite{carlsson2015proof}:
\begin{cor}  \label{cor:shf}
  For any composition $\mu$, 
the antisymmetrized monomials
  \begin{equation}
  \label{eq:symcor}
      \{\antisym_\mu \by^{\ba} \bx^{\bk}:(\ba,\bk)\in
      \schedpfl^{>}_{\mu}(n)\}
  \end{equation}
  are a basis of 
  $(DR_n \otimes \signrep)^{S_{\mu}}$. In particular, we have a new proof of the ``schedules'' version of the Shuffle Theorem, which is Theorem $1'$.
\end{cor}

In particular, 
to recover the Frobenius character of $DR_n$, we would take
\begin{equation}
\label{eq:cortofrob}    
(\omega \frob DR_n)\big|_{m_\mu}=
\sum_{(\ba,\bk)\in \rschedpfl{\mu}(n)} t^{|\ba|} q^{|\bk|}.
\end{equation}

\begin{proof}
 
First, Proposition \ref{prop:schedswitch} shows that
$\schedpfl(n)$ is closed under diagonally sorting adjacent entries with respect to an ordering
on pairs $(a_i,k_i)$,
in which the elements of 
$\schedpfl_{\mu}^{>}(n)$ are minimal for transpositions in $S_\mu$.
Thus if $f=\antisym_\mu \by^{\ba} \bx^{\bk}$
for $(\ba,\bk)\in \schedpfl(n)$, then either 
$f =0$, or 
$f=\pm \antisym_{\mu} \by^{\ba} \bx^{\bk}$
for some $(\ba,\bk)\in \schedpfl^{>}_{\mu}(n)$.
It follows that the elements in \eqref{eq:symcor} span
$(DR_n \otimes \signrep)^{S_\mu}$.
Then, using only the \emph{ungraded} Shuffle Theorem (Proposition 3.13 in \cite{Haiman02} at $q=1$, which computes $\frob DR_n$ as a sum of parking functions), we find that the dimensions agree, so that the set must also be linearly independent.
    
\end{proof}

The case of \(\mu=(n)\) in the above corollary is of special importance since the \(q,t\)-graded character of \((DR_n\otimes\signrep)^{S_n}\)
is a \(q,t\)-Catalan number. Indeed for a Dyck path \(\pi\), there is a unique permutation \(\sigma(\pi)\in S_n\) such that
\(P_\pi=(\pi,\sigma(\pi))\) is a parking function  and \(\mathrm{word}(P_\pi)=(n,n-1,\dots,1)\). Thus
for any Dyck path \(\pi\) we can define \(\area_i(\pi)=\area_i(\pi,\sigma(\pi))\),
\(\dinv_i(\pi)=\dinv_i(\pi,\sigma(\pi))\).
Thus antisymmetric functions:
\[f_\pi=     \antisym_{(n)} \big(
    \mathbf{y}^{\mathbf{area}(\pi)}_{\sigma(\pi)}
    \mathbf{x}^{\mathbf{dinv}(\pi)}\big)\]
as \(\pi\) ranges through the set of \(n,n+1\) Dyck paths,
form a basis of \((DR_n\otimes\signrep)^{S_n}\). 

Let \(A=(\CC[\mathbf{x},\mathbf{y}]\otimes \signrep)^{S_n}\)  be
the space of alternating polynomials. This space has
a basis that consists of non-zero sums \(\antisym_{(n)}\mathbf{x}^{\mathbf{v}}\mathbf{y}^{\mathbf{u}}\)
for \(v_1\ge v_2\ge \dots\ge v_n\). The functions \(f_\pi\)
from the above are of this form.
In Haiman's construction of the celebrated isospectral Hilbert scheme \cite{Haiman02} he uses the ideal \(J\) that is generated by
the elements of \(A\). Since it is shown that in \cite{Haiman02} that \(J/(\mathbf{x},\mathbf{y})J=(DR_n\otimes \signrep)^{S_n}\) we
obtain:
\begin{cor}\label{cor:min-gen} The set of alternating functions \(f_\pi\) as \(\pi\)
  ranges through the set of \((n,n+1)\) Dyck paths is a set of minimal
  generators for the ideal \(J\).
\end{cor}

\subsection{Proof of Theorem \ref{thmb}}

We begin with some lemmas. Recall the elements $\tilde{y}_i=\tilde{z}_i-1 : \Saf \rightarrow \Saf$, which descend to $y_i$ under 
the isomorphism $\Saf/\epsilon \Saf \cong DR_n$, as well as the
element $\tilde{\Delta}_n \in \Saf$ 
from Section \ref{sec:springact}.

\begin{lemma}\label{zinfalem}
  Let $\ba=\majt(\tau)$. 
  Then we have
\begin{equation}
  \label{zatrieq}
\tilde{\by}^{\ba}\tilde{\Delta}_n
  =
  \sum_{w\in \Res(n,n+1)} a_w(\epsilon)w \in \Saf,
  \end{equation}
  where $a_w(\epsilon)=0$ unless $\indt(w) \leq_{des} \ba$, and
  $a_w(\epsilon)\neq 0$ for all 
  $w \in \Res(\tau)$.
In particular, $\by^{\ba}$ defines
a nonzero element of $G_{\ba} \hbm_*(\Sp_{n,n+1})$.
\end{lemma}

\begin{proof}

First, since the descent order is compatible with 
the product order on integer vectors, 
we have that $\tilde{\bz}^{\ba}$ is a linear combination of terms $\tilde{\by}^{\ba'}$ with
$\ba'\leq_{des} \ba$ with leading coefficient equal to one. It therefore suffices to prove the lemma 
with $\tilde{\bz}^{\ba} \tilde{\Delta}_n$ in place of $\tilde{\by}^{\ba} \tilde{\Delta}_n$.

%Write $\tilde{\bz}^{\ba}$ for short in place of $\tilde{\bz}^{\ba} \tilde{\Delta}_n$, and let
Let us write the expansion
\begin{equation}
\label{eq:zaexpansion}
\tilde{\bz}^{\ba} \tilde{\Delta}_n=\sum_{w\in \Res(n,n+1)} 
b_w(\epsilon) w.
\end{equation}
We have the following rules determining 
$\tilde{\bz}^{\ba}\tilde{\Delta}_n$,
supposing $\ba$ is a descent composition:
\begin{enumerate}
\item \label{item:baserec} If $\ba=(0,...,0)$, then $\tilde{\bz}^{\ba}\tilde{\Delta}_n=\tilde{\Delta}_n$.
    \item \label{item:switchrec} If $a_i>a_{i+1}$, then $\tilde{\bz}^{\ba}\tilde{\Delta}_n=-(\tilde{\bz}^{\ba'} \tilde{\Delta}_n)*_{n+1} s_i $, where $\ba'=\ba \cdot s_i$.
    \item \label{item:rotrec} If $a_n>0$, then $\tilde{\bz}^{\ba}\tilde{\Delta}_n=
    \Ad_{\rho^{-1}}(\tilde{\bz}^{\ba'}\tilde{\Delta}_n)$,
    where $\ba'=(a_n-1,a_1,...,a_{n-1})$.
\end{enumerate}
In each case, we have that if $\ba$ is a descent composition, then $\ba'$ is a descent composition. We can use these rules to recursively determine $\tilde{\bz}^{\ba}\tilde{\Delta}_n$ for any descent composition $\ba$.

We proceed by induction on $\ba$, using the relations. In the base case
from part (\ref{item:baserec}), 
we have that $\Res(\tau)=S_n$ and 
$\tilde{\bz}^{\ba}\tilde{\Delta}_n$, so the claim follows from \eqref{eq:deltatilde}.

Otherwise, we must be in the case of item \ref{item:switchrec}) or \ref{item:rotrec}),
or both. Suppose first we have that
$a_i>a_{i+1}$ for some $i$, and let $\ba'=\ba\cdot s_i$ as in item \ref{item:switchrec}). Then $\ba'$ is always a descent composition, so that
$\ba'=\majt(\tau')$ for some $\tau'$.
We next claim that
%A combinatorial argument shows that
\begin{equation}
\label{eq:combswitch}    
\{w' s_i: w'\in \Res(\tau')\}\cap
\Res(n,n+1)=\Res(\tau).
\end{equation}
To see this, since
$\indt(ws_i)=\indt(w)\cdot s_i$, we have that the left side is contained in the right.
The reverse follows from the statement that
if $w\in \Res(n,m)$ and $\inv(ws_i)\leq \inv(w)$, then $ws_i\in \Res(n,m)$.

Next, since we already know that
the nonzero coefficients $b_w(\epsilon)$
from \eqref{eq:zaexpansion} occur
for $w\in \Res(n,n+1)$, 
they must all be in $\Res(\tau)$.
The statement that the coefficients are nonzero follows from \eqref{spactfawg} and
\eqref{wsinonzero} for $m=n+1$.
The case of item \ref{item:rotrec})
can be proved similarly.

We now have that $\tilde{\by}^\ba \tilde{\Delta}_n$ is an element of
$F_\ba \Saf$ using part \ref{item:fplem})
of Lemma \ref{lem:geofilt}.
The final statement follows since 
$\tilde{\by}^{\ba} \tilde{\Delta}_n$
is nonzero in $G_\ba \Saf$, and
maps to $\by^\ba \Delta_n\in \hbm_*(\Sp_{n,n+1})$.

\end{proof}

We will see in the next two lemmas that
$G_\ba \Saf$ is generated by applying Chern operators $\tilde{x}_i$ to $\tilde{y}^{\ba} \tilde{\Delta}_n$. By Lemma \ref{zinfalem} and Lemma \ref{lem:geofilt} item \ref{item:chernfilt}),
we can make the following definition,
which is an affine version of
Definition \ref{def:chisets}.

\begin{defn}
Let $\ba=\majt(\tau)$. We define a map
$\tilde{\chi}_\tau:\C[\bx,\epsilon]
\rightarrow F_\ba\Saf$ 
of $\C[\bx,\epsilon]$ modules by
\[g(\bx,\epsilon)\mapsto
g(\tilde{\bx}_\tau,\epsilon) \tilde{\by}^{\ba} \tilde{\Delta}_n\]
We let $\bar{\chi}_\tau(\bx,\epsilon)$ be the composition of $\tilde{\chi}_\tau$ with the projection $F_\ba \Saf\rightarrow G_\ba \Saf$.
We denote the kernel of $\bar{\chi}_\tau$ by
$\bar{I}_\tau\subset \C[\bx,\epsilon]$.
%and the image by $\bar{M}_\tau$
\end{defn}

\begin{lemma}
A polynomial $g(\bx,\epsilon)$ is in $\bar{I}_\tau$
if and only if
\begin{equation}
    \label{eq:itaurels}
g(c_1(w\tau^{-1})\epsilon,...,
c_n(w\tau^{-1})\epsilon,\epsilon)=0,
\end{equation}    
for all $w\in \Res(\tau)$,
where $c_i(w)=c(w_i)$ are the coefficients 
in \eqref{zatrieq}.
\end{lemma}

\begin{proof}
In the fixed point basis, we have that
%$\tilde{\chi}_\tau:\C[\bx,\epsilon]
%\rightarrow F_\ba\Saf$ given by
\[\tilde{\chi}_\tau(g(\bx,\epsilon))
=\sum_{w} a_w(\epsilon) g(c_1(w\tau^{-1})\epsilon,...,c_n(w\tau^{-1})\epsilon,\epsilon)w,\]
where $a_w(\epsilon)$ 
are the coefficients in \eqref{zatrieq}.
By Lemma \ref{lem:geofilt} item
\ref{item:fplem}), we find that an element
\[\sum_{w} b_w(\epsilon) w \in F_\ba \Saf\]
maps to zero in $G_\ba \Saf$ if and only if $b_w(\epsilon)=0$ for all $w\in\Res(\tau)$.
The lemma follows from the fact that $a_w(\epsilon)\neq 0$
for all $w\in \Res(\tau)$ by Lemma \ref{zinfalem}.

\end{proof}

\begin{lemma}
\label{lem:equivindep}
Let $\ba=\majt(\tau)$.
Then the elements 
$\tilde{\bx}_\tau^{\bk} \tilde{\by}^{\ba} \tilde{\Delta}_n$
determine a $\C[\epsilon]$-basis of 
$G_\ba \Saf$ for $\bk\in \scheds(\tau)$.
%$\bx_\tau^{\bk} \by^{\ba}\Delta_n$ define a $\C$-basis
\end{lemma}

\begin{proof}

We first show that
the $\bx^{\bk}$ are linearly independent over 
$\C[\epsilon]$ in $\C[\bx,\epsilon]/\bar{I}_\tau\cong G_\ba \Saf$.
Recall the bijection between 
$\Res(\tau)$ and $\hesscomb(\tau)$
from Proposition \ref{prop:hessbij}.
Using \eqref{eq:reshess} and \eqref{ydualeq},
we find that $\bar{I}_\tau$ is equal to the image of $I_{\hesscomb(\tau)}$
under the invertible 
linear change of variables
\[x_i\mapsto (d_i\epsilon -x_i)/n,\ \ 
d_i=\frac{n-1+na_i-\maj(\tau)}{2n},\]
arising from the discrepancy 
between the restriction of the action of $U$ to 
$\gFl\subset \Fl$ and that of $\C^*$.
Independence now follows from Proposition
\ref{prop:schedindep}.

We now check that they are a basis by showing that the $\C[\epsilon]$-module $\langle \tilde{\bx}_\tau^{\bk} \tilde{\by}^{\ba} \tilde{\Delta}_n\rangle\subset G_\ba \Saf$ spanned by those vectors has the same Hilbert series as $G_\ba \Saf$.
By the previous paragraph, we have that
$\langle \tilde{\bx}_\tau^{\bk} \tilde{\by}^{\ba} \tilde{\Delta}_n\rangle$ is free 
with the generators being a basis. For the second,
Proposition \ref{prop:spbasis} and
Lemma \ref{lem:geofilt} item
(\ref{item:filtbaslem}) imply that the elements
$B_w$ 
%from \eqref{eq:defbclass}
for $w\in \Res(\tau)$
determine a $\C[\epsilon]$ basis of $G_\ba \Saf$, with
item \ref{cvpropdeg}) of
Proposition \ref{prop:spbasis} 
%items 
%\ref{item:filtbaslem}) and 
%\ref{item:geoinjlem}) 
giving the degrees of those elements. The statement now follows from the identity
\begin{equation}
\label{eq:grdimsaf}
\sum_{w\in \Res(\tau)}   
q^{-\dim_{n+1}(w)}=
q^{-n(n-1)/2}
\sum_{\bk \in \scheds(\tau)} q^{|\bk|},
\end{equation}
which is due to 
%the degree-preserving bijection
%$\Res(\tau)\leftrightarrow \scheds(\tau)$ 
%from 
Proposition \ref{prop:hessbij}, which equates the 
graded count of each basis.

\end{proof}

\begin{cor}
  \label{cor:gabasis}
Let $\ba=\majt(\tau)$.
Then the elements 
$\bx_\tau^{\bk} \by^{\ba}\Delta_n$ 
for $\bk \in \scheds(\tau)$
define a $\C$-basis
of the quotient module $G_\ba \hbm_*(\Sp_{n,n+1})$.
\end{cor}

\begin{proof}

By Lemma \ref{lem:geofilt} items
\ref{item:filtbaslem}) and 
\ref{item:geoinjlem}), we have that
$G_\ba \Saf$ is free over $\C[\epsilon]$, and that $G_\ba \hbm_*(\Sp_{n,n+1}) \cong G_\ba \Saf \otimes_{\C[\epsilon]}\C= G_\ba \Saf/\epsilon G_\ba \Saf$.
The statement now follows since the $x_i,y_i$ are induced by
$\tilde{x}_i,\tilde{y}_i$ from Definition \ref{def:xyaction}

%Since we would have the right number of elements in each graded component, it would be a basis.

\end{proof}

We can now prove Theorem \ref{thmb}.

\begin{proof}
Let $F'_\ba DR_n$ be the filtration by the descent order defined on the right hand side of \eqref{eq:descfilt}, and similarly for the subquotient $G'_\ba DR_n$. 
We prove that
$F_\ba DR_n=F'_\ba DR_n$
inductively with respect to the descent order on $\ba$, assuming that the two filtrations are equal for all $\ba'<_{des} \ba$. 

For the base case $\ba=(0,...,0)$,
we have that $F_\ba DR_n=F'_\ba DR_n=R_n(\bx)$
using the definition of $DR_n$, 
item \ref{item:classsp}) of Proposition \ref{prop:spbasis},
and the fact that
$\Res(\tau)=S_n$ for $\tau$ the identity permutation.

Now assume inductively that $F'_{\ba'} DR_n=F_{\ba'} DR_n$ for $\ba'<_{des} \ba$ for some
$\ba$. By Lemma \ref{zinfalem},
we have that $\by^{\ba} \in F_{\ba} DR_n$,
and therefore $F'_\ba DR_n \subset F
_\ba DR_n$
by Lemma \ref{lem:geofilt} part \ref{item:chernfilt}). Putting these two together, we have an inclusion $G'_\ba DR_n \subset G_\ba DR_n$, and it suffices to show that it is an equality.

In the case where $\ba$ is not a descent composition, we have that
$G_\ba DR_n=\{0\}$ since $\indt(w)$ is always a descent composition for $w\in \Res(n,n+1)$,
and $G'_\ba DR_n=\{0\}$ by Proposition
\ref{prop:allen}.
Thus, we may assume that $\ba=\majt(\tau)$, 
for some $\tau$.
By Corollary \ref{cor:gabasis},
there is a basis of $G_\ba DR_n$
whose elements are contained in
$G'_\ba DR_n$, so the two are equal.
This argument has also proved item \ref{item:thmbasis}).

Item \ref{item:thmquot}) follows since $g_\tau(\bx)=\tilde{g}_\tau(\bx,0)\big|_{x_i=x_{\tau_i}}$, 
where
  \begin{equation}
\label{haggendef}
\tilde{g}_h(\bx,\epsilon)=
q(\bx,\epsilon)
\prod_{i=1}^{n} \prod_{j=h(i)+1}^n (x_i-x_j-\epsilon),
\end{equation}
and $q(\bx,\epsilon)$ is the polynomial appearing in the proof Proposition \ref{prop:schedindep}.
We can see that $\chi_{S_n}(\tilde{g}_\tau(\bx,\epsilon))$ vanishes at precisely the fixed points in $\hesscomb(\tau)$ under
$\chi_{S_n}$ from Definition \ref{eq:chiset} using argument similar to that proposition, and item \ref{item:hesscombb}) from Proposition \ref{prop:hesscells}. We then have that the 
ideal $(\tilde{g}_\tau(\bx,\epsilon)) M_{S_n}
\subset M_{S_n}$ is isomorphic to $M_{\Res(\tau)}$,
noting that $M_{S_n} \cong H^*_{\C^*}(\gFl)$. 
We saw that $M_{\Res(\tau)}\cong G_\ba \Saf$ is free in the proof of Lemma \ref{lem:equivindep}, so we find that the image of $(\tilde{g}_\tau(\bx,\epsilon))M_{S_n}$ in 
$M_{S_n}/(\epsilon) M_{S_n} \cong R_n(\bx)$
is $(g_\tau(\bx))$.

%equation \eqref{haggendef} and $q(\bx,\epsilon)$
%is from the proof of Lemma \ref{lem:equivindep}.
%the proof of Lemma \ref{lem:gabasis}.

\end{proof}

\section{Geometry of the Hessenberg paving}

\label{sec:geohess}

In this section we provide geometric explanations for 
the algebraic arguments we used to prove our main results.
We consider the subspace
\begin{equation}
\label{eq:defhesspav}    
\hesspav(\tau)=\bigcup_{w\in \Res(\tau)} Y^\circ_w=
\hessleq(\ba)-\bigcup_{\bb <_{des} \ba} \hessleq(\bb)
\end{equation}
for $\ba=\majt(\tau)$.
We prove that $\hesspav(\tau)$ is isomorphic to an affine bundle over a base space $\hessgeo(\tau)$,
which is a smooth subvariety of the 
regular nilpotent Hessenberg variety $\hessgeo(N,h_\tau)$.
It will then follow that 
$\hbm_*(\hesspav(\tau))\cong G_\ba \hbm_*(\Sp_{n,n+1})$,
the latter being the algebraic filtration
from the previous section, and that both are isomorphic to $\hbm_*(\hessgeo(\tau))$ with a grading shift, stated in Theorem \ref{thm:geo}.

To guide reader through this section we assemble the spaces and maps in one diagram, this diagram is a geometric counterpart of the
diagram from the beginning of the Section \ref{sec:affine-perm-park}.

\[\begin{tikzcd}[column sep=1.5em]%[cramped]
	\Sp_{n,n+1}^{\grsup} && \Sp_{n,n+1} && \hesspav(\tau) && \\
	\\
	\bigsqcup_{\mu}\ghilbgeo_0(\C^2)^+_{F_\mu} && \bigsqcup_{\mu}\pfghilbgeo(\C^2)^+_{F^{\bullet}_\mu} && \pfghilbgeo(\C^2)^+_{F^{\bullet}_{\mu(\tau)}} && { \ZZ[t,q^{\pm 1}]} \\
	\\
	\bigsqcup_\ell\hessgeo(\ell) && \bigsqcup \hessgeo(\tau) && \hessgeo(\tau)
	\arrow[two heads, from=1-1, to=3-1, "\gext\circ\psi"]
	\arrow[two heads, from=1-3, to=1-1]
	\arrow[two heads, from=1-3, to=3-3]
	\arrow[hook', from=1-5, to=1-3]
	\arrow[two heads, from=1-5, to=3-5,"\gext\circ\psi",,"\CC^{\maj(\tau)}"']
	\arrow[from=1-5, to=3-7,"t^{\maj(\tau)}q^{\delta}P_{q^{-1}}"]
	\arrow[two heads, from=3-1, to=5-1, "\varphi"]
	\arrow[two heads, from=3-3, to=3-1]
	\arrow[two heads, from=3-3, to=5-3]
	\arrow[hook', from=3-5, to=3-3]
	\arrow[from=3-5, to=3-7]
	\arrow[two heads, from=3-5, to=5-5,"\varphi","\CC^{\delta-\maj(\tau)-\mathrm{sch}(\tau)}"']
	\arrow[two heads, from=5-3, to=5-1]
	\arrow[from=5-5, to=3-7,"t^{\maj(\tau)}P_q"']
	\arrow[hook', from=5-5, to=5-3].
\end{tikzcd}\]
In the diagram we use \(\delta=n(n+1)/2\), \(\mathrm{sch}(\tau)=\sum_i \mathrm{sch}_i(\tau)\), \(\tau\in S_n\), and
\(P_q(X)=\sum_i q^i \dim \overline{H}_{2i}(X)\) is the Poincar\'{e} polynomial for Borel-Moore homology. The spaces
\(\pfghilbgeo(\C^2)^+_{F^{\bullet}_\mu}\) are the flag versions of the spaces \(\bigsqcup_{\mu}\ghilbgeo_0(\C^2)^+_{F_\mu}\),
these spaces are defined in section~\ref{sec:par-hilbert}.

The vertical maps in the diagram are affine fibrations and we indicate the fibers on the diagram. All the spaces in the diagram
are equipped with \(\CC^*\)-action and
to relate to the
diagram at the start of Section \ref{sec:affine-perm-park} we extract the \(\CC^*\) fixed locus of the spaces in diagram and
set \(\mu={ (1,...,1)}\) in the leftmost column of the diagram from the start of Section \ref{sec:affine-perm-park}.

%We then move on to the full flag case, which makes use of a variety called the \emph{parabolic flag Hilbert scheme},
%which was defined and shown to be smooth in \cite{carlsson2017parabolic}.

%A key example is Lemma \ref{lem:grasshilb}, which explains the  Hessenberg and 
%Schubert-type conditions from Definition \ref{def:hesstau}.
%In particular, Proposition~\ref{prop:geo-comb} describes the descent filtration from Theorem~\ref{thmb}. In Sections~\ref{sec:dualASF}-\ref{sec:hessenbergs} we prove
%Propositions
%~\ref{prop:Hess}, \ref{prop:ftau}, and \ref{prop:geo-comb},
%which together imply the statement of Theorem C from the introduction.

\subsection{The Grassmannian case}
\label{sec:grass}

We start with a version of \eqref{eq:defhesspav} in the case of the affine Grassmannian, which turns out to coincide with a 
G\"{o}ttsche's result (Proposition \ref{prop:gottsche})
in the context of the punctual Hilbert scheme. 

Let $\Sp^{\grsup}_{n,m}=\Sp^{\grsup}_\gamma \subset \Gr$ be the Grassmannian affine Springer fiber, as defined in 
\eqref{eq:defaffspgr} for 
$\gamma=\gamma_{n,m}$. 
The cells are identified with 
$m$-restricted affine permutations which are 
\emph{Grassmannian}, meaning that they are in sorted order in window notation, given by 
$\Res^{<}_{(n)}(n,m)$ using the notation of Section \ref{sec:affine-permutations}. Recall that the corresponding open Schubert cell is denoted by 
$X^{\grsup,\circ}_w$ and
$Y^{\grsup,\circ}_w=X^{\grsup,\circ}_w \cap \Sp^\grsup_{n,n+1}$. As in Section 
\ref{sec:geometryFl}, we will use a superscript of $g$ to denote the 
affine Grassmannian case.

We describe the Grassmannian version of
\eqref{eq:defhesspav}, and in particular fix $m=n+1$. Then the permutations 
$\tau$ for which the Grassmannian restricted permutations have nonempty intersection with $\Res(\tau)$ are those for which
$\majt(\tau)$ is in increasing order. There is exactly one such permutation $\tau$ for each run composition $\mu$, which can be identified by starting with the maximal length permutation
$(n,...,1)$, and sorting in increasing order
along the blocks of $\mu$. The resulting set will be denoted $\Res(\mu)$, { and we will similarly write $\maj(\mu)=\maj(\tau)$, 
and $\sched_i(\mu)=\sched_i(\tau)$}.

We define 
\begin{equation}
\hesspav(\mu)=\bigcup_{w\in \Res(\mu)}
Y^{\grsup,\circ}_w \subset \Sp^\grsup_{n,n+1}.
\end{equation}
We will prove that $\hesspav(\mu)$ is isomorphic to a vector bundle over a smooth locus in a certain
partial Hessenberg variety. 
%In fact, we will see that
%both spaces turn out to be familiar ones that have been studied by G\"{o}ttsche in the context
%of the punctual Hilbert scheme of points on $\mathbb{C}^2$ \cite{gottsche1990punctual}.

%\subsection{Jacobian varieties}
%\label{sec:grassjac}

We first
reformulate the definition of the affine Grassmannian in a way that connects 
it with the compactified Jacobian variety.
We start by introducing the following local rings:
\[\calO=\CC[[\eps^n]]\subset R=\CC[[\eps^n,\eps^m]]\subset \tilde{R}=\CC[[\eps]]\]
which is identified with \(\calO\) from Section~\ref{sec:geometryFl} 
by setting \(t=\eps^n\).
Similarly, we have $\mathcal{K}=\mathbb{C}((\eps^n)) \subset \tilde{F}$ for $\tilde{F}=\mathbb{C}((\eps))$.
%Then writing
%\(\tilde{R}[\eps^{-1}]=\mathbb{C}((\eps))\) 
%to denote the field  of formal Laurent series, we have

By identifying $\eps^{i}$ with the basis vector $e_i$ for $0\leq i \leq n-1$, we have
\begin{equation}
\label{eq:splat}
\tilde{R}=\oplus_{i=0}^{n-1}\eps^i\calO,\ \ 
\tilde{F}=\oplus_{i=0}^{n-1}\eps^i\calK.\end{equation}
%Using \eqref{eq:splat},
We may then describe the affine Grassmannian 
as \( \agrass \) as the space of
$\mathcal{O}$-submodules of $\tilde{F}$
with the property that
%from Section~\ref{sec:geometryFl} 
%as the moduli space of sublattices 
\begin{equation}
\label{eq:deflatgrass}
\codim_{\tilde{R}}\tilde{R}\cap L-\codim_{L}\tilde{R}\cap L=(m-1)(n-1)/2.
%
%L\subset \tilde{R}[\eps^{-1}],\ \ \eps^n L\subset L,
%\ \ L\otimes_{\tilde{R}} \tilde{R}[\eps^{-1}]=\tilde{R}[\eps^{-1}]
\end{equation}
In particular, the codimension is finite,
which implies that
$L\otimes_{\mathcal{O}} \mathcal{K}=\tilde{F}$.
%\(\mathrm{ind}(L)=0\), where
%\[\mathrm{ind}(L)=
%  \[\codim_{\tilde{R}}\tilde{R}\cap L-\codim_{L}\tilde{R}\cap L=0.\]
and \(R\in \agrass\).
Then
\(\Sp^{\grsup}_{n,m} \subset \agrass\) is 
identified with the locus of lattices 
that are additionally preserved by multiplication by \(\eps^m\), and so is an $R$-module.
It was shown in \cite{Laumon06} that \(\Sp^{\grsup}\) 
is homeomorphic to 
the  local factor of the compactified Jacobian of the curve singularity \(x^m=y^n\).

For each \(L\in \Gr\) using this description,
let \(\Gamma(L)\subset \ZZ\) be the semi-module of lower degrees:
\[\Gamma(L)=\{\ldeg(f): 
f\in L\},\]
where $\ldeg(f)$ is degree of the lowest nonvanishing coefficient,
and let  
\(\ldeg(L)=\min(\Gamma(L))\in \Z\). 
Then $\Gamma(L)$ is a submodule 
of $\mathbb{Z}$ over the semigroup $\langle n\rangle$
which has $n$ unique generators
$(k_1,...,k_n)$, which we assume are written in increasing order so that $k_1=\ldeg(L)$. 
Let $\ba=\indt(L)$
be the composition for which 
$a_i=\lfloor (k_i-k_1)/n\rfloor$.
%which agrees with the composition
%$\indt(w)$ from Section \ref{sec:affine-permutations}
%when $w$ is a Grassmannian 
%affine permutation 
%corresponding to $L$.
We similarly have the total index 
$\ind(L)=|\indt(L)|$ which satisfies
\[\ind(L)=|\ZZ_{>\ldeg(L)}\setminus \Gamma(L)|.\]
By the combinatorial Proposition~\ref{prop:extstats}
for \(L\in \hesspav(\mu)\), \(\ind(L)=\maj(\mu)\).

%In a similar direction, 
We now describe a map from the $\Gr$ to 
a space of submodules of a power series ring 
rather than Laurent series. Instead of the condition
from \eqref{eq:deflatgrass}, we consider 
Grassmannian varieties will have finite codimension:
given an $S$-module \(N\), let
\(\Grr^d_S(N)\) denote the collection
of \(S\)-submodules \(M\subset N\) such that
\(\dim_{\C} N/M=d\), and let 
\(\Grr_S(N)=\bigcup_d \Grr^d_S(N)\).
We will also use
\(\grquot{S}{N}{N'}\) to denote the
collection of 
\(S\)-modules \(M\) such that \(N'\subset M\subset N\).

Now suppose $L\in \affgrass$. 
Then there is a unique element 
\(v_L\in L\) such that
\begin{equation}
\label{eq:inLam}
  v_L=\eps^{\ldeg(L)}+\sum_{\gamma\in \ZZ_{>\ldeg(L)}\setminus \Gamma(L)} a_\gamma \eps^\gamma,
  \end{equation}
and every element \(v\in L\) is divisible by \(v_L\). 
Thus can  we define the \emph{class map}
\begin{equation}
    \label{eq:classmap}
    cl: \Gr \rightarrow \grquot{\mathcal{O}}{\tilde{R}}{\mathcal{O}},
    %\Sp^\grsup_{n,m}\to \grquot{R}{\tilde{R}}{R},
    %Grr_R(\tilde{R},R),
\ \ cl(L)=\{v/v_L:v\in L\}.
%L/\clmin(L).
\end{equation}
Indeed, \(L/v_L\) contains element \(1\in\tilde{F}\) and it is the element of \(L/v_L\) of the lowest degree, hence
\(\mathcal{O}\cdot \subset L/v_L\) and \(L/v_L\subset \tilde{R}\)

% The class map is discontinuous, but is
% continuous on the collection of
% lattices for which $\indt(L)=\ba$ for a fixed composition $\ba$.
% We provide the details below.

Restricting $cl$ to the Springer fiber \(\Sp_{n,m}^{\grsup}\)
%$\Sp^g_{n,m}$ 
determines a map to
%$cl:\Sp^g_{n,m}\rightarrow 
$\grquot{R}{\tilde{R}}{R}\subset \grquot{\mathcal{O}}{\tilde{R}}{\mathcal{O}}$. Indeed, \(1\in L/v_L\) and
since \(L\in \Sp^{\grsup}_{n,m}\) is preserved by multiplication by elements of \(R\), \(R\cdot 1
\subset L/v_L\).
The index of
$L$ corresponds to $\majt(\mu)$ for $L\in \hesspav(\mu)$,
we have a  map:% {\color{red} of varieties (right?)}
\begin{equation}
\label{eq:clspmu}    
cl: \hesspav(\mu)\rightarrow \grquot{R}{\tilde{R}}{R}.
\end{equation}
Below we show that the fibers of this map 
are { isomorphic to} $\CC^{\maj(\mu)}$ and the map is
a continuous map on each  \(Y^{g,\circ}_w\). We also explain below a
construction of stabilization \(\Sp_{n,m}^{g,\aux}\) that leads to a stable version
of the map \eqref{eq:clspmu} that is continuous.  

\begin{prop} For \(L\in  \Sp^{\grsup}_{n,m}\) we have
  \(\ind(L)=\conduc/2-\ldeg(L),\) \(\conduc=(n-1)(m-1)\).
\end{prop}
\begin{proof}
  Let \(\hilbgeo^N_0(x^n-y^m)\) be the punctual Hilbert scheme
  \(N\) points on \(x^n-y^m\), that is the moduli space of
  ideals \(I\subset R\), such that \(\dim R/I=N\).
  Since \(\eps^{\conduc}\tilde{R}\subset R \) the multiplication by
  \(\eps^N\) yields an embedding \(i_N:\Sp^{\grsup}_{n,m}\to \hilbgeo^N_0(x^n-y^m)\)
  for sufficient large \(N\), see for example \cite{ORS}.

  Let us denote the lower order degree
  of an ideal \(I\in \hilbgeo_0^N(x^n-y^m)\) by \(\ldeg'(I)\).
  We set \(\Gamma(I)\) to be the set of the \(\eps\) degrees of element in \(I\).
  By construction we have  \(\ldeg'(i_N(L))=\ldeg(L)+N\)
  and \(\Gamma(i_N(L))=N+\Gamma(L)\). If we define
  \(\ind(I)\) as \(|\ZZ_{\ldeg'(I)}\setminus \Gamma(I)|\) then
  \(\ind(i_N(I))=\ind(I)\).

  Finally, let us observe that for any \(I\in \hilbgeo_0^N(x^n-y^m)\)
  for sufficiently large \(N\) we have \(\ind(I)+\ldeg'(I)=N-\conduc/2\). Thus the sum \(\ind(L)+\ldeg(L)\)
  is constant as we vary \(L\) in \(\Sp^{\grsup}_{n,m}\). To compute
  the constant we use \(R\in \Sp^{\grsup}_{n,m}\).
\end{proof}

The map \(cl\) sends lattices \(L\) with \(\ind(L)=d\) to
\(\mathcal{G}r_R^d(\tilde{R}/R)\). 
We can think about
an element of \(\grr_R^d(\tilde{R}/R)\)
as \(L\in \grr_{R}(\tilde{R})\) such that \(R\subset L\). Hence we have a multiplication map:
\[\mathrm{act}: \tilde{R}^\times\times \grr_R^d(\tilde{R}/R) \to 
\Sp^{\grsup}_{n,n+1},\quad (f,L)\to \eps^{\conduc/2-d}\cdot f\cdot L \]
where \(\tilde{R}^\times\) is group of units of \(\tilde{R}\). The
space \(\Sp^{\grsup}_{n,n+1}\) is paved by affine affine spaces
\(\Sp^{\grsup}(\Gamma)\) which consist of \(L\) such that \(\Gamma(L)=\Gamma\). For each \(w\in \Res_{(n)}^{<}(n,n+1)\) let 
\(\Gamma(w)\) be a semigroup with \(\Gamma_{n,n+1}\) generators
\(w_1,\dots,w_n\) then \(Y^\circ_w=\Sp^{\grsup}(\Gamma(w))\).

% The map
% \(cl\) is a section of the last map, since
%\begin{equation}\label{eq:sec-cl}\mathrm{act}\circ 1\times cl=1.
%\end{equation}

\begin{prop}
  The map \(cl:\Sp^{\grsup}(\Gamma)\to \agrass_{R}(\tilde{R}/R) \)
  is a smooth affine fibration with affine fibers that are
{ isomorphic to}  \(\CC^d\), \(d=|\ZZ_{>\min(\Gamma)}\setminus \Gamma|\).
\end{prop}
\begin{proof} Let \(r=\min(\Gamma)\),
  and let us define \(G(\Gamma)\subset \tilde{R}^\times\) to be
  the set of elements of the form \(1+\eps^{-r}\sum_{\gamma\in\ZZ_{>r}\setminus\Gamma}a_\gamma\eps^\gamma\),
  \(a_\gamma\in \C\). By construction \(G(\Gamma)=\C^d\).

  The map \(\mathrm{act}: G(\Gamma)\times \grr_{R}(\tilde{R}/R)\to \Sp_{n,n+1}^{\grsup}\) is a local isomorphism. The inverse 
  of this map is defined by \(L\mapsto (v_L/\eps^r,cl(L)). \)
\end{proof}

The map \(cl\) is continuous on \(\Sp^{\grsup}_{n,m}(\Gamma)\) but
it is enough for studies of BM homology of \(\Sp^{\grsup}_{n,m}\)
since \(\Sp^{\grsup}_{n,m}(\Gamma)\) have no odd-dimensional homology.
However, one could construct a stabilized version of the map
\(cl\) that is continuous over the connected components of
\(\grr_{R}(\tilde{R}/R)\). We outline the details of the stabilization
construction below. However, we are not using the stabilized version
of the map to keep notations and exposition simple.

Indeed, we can consider an auxiliary space \(\Sp^{\grsup,\aux}_{n,m}\) that consists of pairs \((L,v)\) where \(L\in \Sp^{\grsup}_{n,m}\)
and a vector \(v \in L\cap \eps^{-N}\tilde{R}\), \(N=2mn\) such that
the coefficient of \(v\) in front of \(\eps^{\ldeg{L}}\) is \(1\) and
the coefficients  in front of \(\eps^{k}\), \(k>2N\) are zero.
The projection map \(\Sp^{\grsup,\aux}_{n,m}\to \Sp^{\grsup}_{n,m}\) is
a smooth affine fibration. The desired stabilized map
\(cl:\Sp^{\grsup,\aux}_{n,m}\to \grr_{R}(\tilde{R}/R)\) is defined
by \((L,v)\mapsto L/v\).

\subsection{Duality map}
\label{sec:ring-modul-dual}

We now give a dual description of lattices in another Grassmannian space, which corresponds to the strange looking difference in the combinatorial description
of \eqref{eq:latmap}. 
%hich will be used to generate a paving of $\Sp_{n,n+1}$
%by bundles over subvarieties satisfying a Hessenberg-type condition in the next subsection. 
The constructions in this section are 
closely related to the results in the Appendix of \cite{PandharipandeThomas10}.

Given relatively prime $(n,m)$, let
\[\conduc=(n-1)(m-1),\quad c=\eps^{\conduc}\]
The element $c$, called the \emph{conductor},
has the property that   
\(c\tilde{R}\subset R\) and it is the smallest element of
\(R\) with this property.
The quotient \(R/c\tilde{R}\) will be denoted by \(\rbar\).

Suppose \(M\) is a finitely generated \(R\)-module. Then we define the dual \(R\)-module by
\[\mathbb{D}(M)=\Ext^1_R(M,R).\]
It turns out that if  \(M\in \Gr_R(\tilde{R}/R)\) then  \(\mathbb{D}(\tilde{R}/M)\) is naturally an \(\rbar\)-submodule of \(\rbar\). To see this, we compute the
\(\mathbb{D}(\tilde{R}/R)\), since \(\tilde{R}/M\) is the quotient of
\(\tilde{R}/R\).
First, we  apply \(\Hom_R(-,R)\) to the short exact sequence:
\begin{equation}\label{eq:R/tilR} 0\to R\to \tilde{R}\to \tilde{R}/R\to 0.\end{equation}
Since \(\tilde{R}/R\) is a torsion module, we have \(\Hom_R(\tilde{R}/R,R)=0\). Thus we get:
\[0\to\Hom_R(\tilde{R},R)\xrightarrow{i} \Hom_R(R,R)\to \Ext^1(\tilde{R}/R,R)\to \Ext^1(\tilde{R},R)\to 0\]

The inclusion \(\nu:R\to \tilde{R}\) is the normalization map 
%\comment{Is the normalization related to $c$?}
and the \(R\)-module
\(\tilde{R}\) is the push-forward: \(\tilde{R}=\nu_*(\tilde{R})\).
%\comment{I'm missing something, how do we take $\nu_*(\tilde{R})$ for $\nu:R\rightarrow \tilde{R}$? Assuming related to the normalization}.
Thus by the
adjuction for \(\Ext^*_R\)  we have \(\Ext^1_{R}(\nu_*(\tilde{R}),R)=\nu_*(\Ext^1_{\tilde{R}}(\tilde{R},\tilde{R}))=0\)
since \(\nu^*(R)=\tilde{R}\).

The same adjunction argument implies \(\Hom_R(\tilde{R},R)=\tilde{R}\) and the image of the inclusion \(i\) is
\(c\tilde{R}\subset R\). Indeed, the element \(\phi\in \Hom_R(\tilde{R},R)\)
is uniquely defined by \(\phi(1)\in R\), since \(\phi(1)=0\) implies that
 \(\phi\in \Hom(\tilde{R}/R,R)=0\). Moreover, the set \(\deg_\eps(\phi(\tilde{R}))\)
is equal to \(\mathbb{Z}_{\ge d}\), \(d=\deg_\eps(\phi(1))\). Indeed, for any \(x\in \tilde{R}\)
there are \(z,z'\in R\) such that \(0=z\phi(x)-z'\phi(1)=
\phi(zx-z')\). Since, \(r=\tilde{R}\) is torsion free \(R\)-module, \(\phi\) is injective  and \(\deg_\eps(\phi(x))=
d+\deg_\eps(z')-\deg_\eps(z)=d+\deg(x).\)
Thus we conclude \(\Ext^1_R(\tilde{R}/R,R)\simeq \rbar\) as \(R\)-module.

 Finally,
 let us identify \(\CC^*\)-equivariant structure on  \(\Ext^1(\tilde{R},R)\). The \(\CC^*\)-action on
 \(\tilde{R}\) is defining \(\CC^*\)-weight of \(\eps\) to be \(1\).
Thus the \(R\) and \(\tilde{R}/R\) acquire \(\CC^*\)-equivariant structure as well.

 We have the equality  of the virtual \(\CC^*\)-representations:
 \([\Ext^1_R(\tilde{R}/R,R)]=[(\tilde{R}/R)^\vee].\)
 Here and everywhere below we use \(M^\vee\) for the dual \(\CC^*\) module.
 On the other hand \([(\tilde{R}/R)^\vee]=[\rbar\{1-\conduc\}]\), where \(M\{k\}=M\otimes \chi^k\) with \(\chi\simeq \CC\) being the
 tautological \(\CC^*\)-representation. The last equality follows from the combinatorial identity for subsets of \(\Z\)
 \[-\Z_{\ge 0}\setminus \Gamma_{n,m}=\Gamma_{n,m}\setminus \Z_{\ge\conduc}-\conduc+1.\]
 Finally, we arrive at
 \begin{equation}\label{eq:wt-sht-1}
 \Ext^1_R(\tilde{R}/R,R)\simeq \rbar\{1-\conduc\}.\end{equation}

Next we apply \(\Hom_R(-,R)\) to the short exact sequence:
\begin{equation}\label{eq:sesMM}
  0\to M/R\to \tilde{R}/R\to \tilde{R}/M\to 0.\end{equation}

The modules in the sequence are \(R\)-torsion hence we obtain the short
exact sequence:
\[0\to \Ext^1(\tilde{R}/M,R)\to \Ext^1(\tilde{R}/R,R)\to \Ext^1(M/R,R). \]

Let us denote the map on the moduli space of \(R\)-modules that sends \(M\) a submodule of \(\tilde{R}\) to the quotient
module \(\tilde{R}/M\) by \(Q\).
Then by  combining the previous constructions with the involutive properties of the the duality functor we obtain:
\begin{prop}\label{prop:dq1}
  The  map \(\mathbb{D}\circ Q\) yields an isomorphism
  \[\mathbb{D}\circ Q:\Grr_R(\tilde{R}/R)\to \Grr_R(\rbar)\]
  \end{prop}
  \begin{proof}
    By applying \(\Hom(-,R)\) to the short exact sequence of  \(R\)-modules
    \begin{equation}\label{eq:sesMR}
      0\to M\xrightarrow{\varphi} \tilde{R}\to \tilde{R}/M\to 0,\end{equation}
    we get a short exact sequence:
    \[0\to c\tilde{R}\simeq \Hom_R(\tilde{R},R) \xrightarrow{\varphi^\vee}\Hom(M,R)\to \Ext^1(\tilde{R}/M,R)\to 0,\]
    and \(\Ext^1(M,R)=0\).

    The map \(\varphi\) is the natural inclusion map, that could be seen  by applying \(\Hom_R(-,R)\) to the diagram of maps
    \[0\to R\xrightarrow{i_M} M\to M/R\to 0,\]
    we get an injective map \(i^\vee_M:\Hom_R(M,R) \to\Hom_R(R,R)=R\). In particular, if \(M=\tilde{R}\) then by the discussion above
    we get \(\Hom_R(\tilde{R},R)=\tilde{R}\) and the map \(i^\vee_{\tilde{R}}\) is the inclusion of \(c\tilde{R}\) inside \(R\). Finally, we observe that
    \(\varphi\circ i_M=i_{\tilde{R}}\) as thus \(i^\vee_M\circ \varphi^\vee=i_{\tilde{R}}^\vee\).

    Now we see that \(\Ext^1_R(\tilde{R}/M,R)=\Hom_R(M,R)/c\tilde{R}\).
    The curve \(\mathrm{Spec}(R)\) is Gorenstein hence the duality functor \(\Hom_R(-,R)\) is involutive on the
    derived category of \(R\)-modules.  
    
    In more detail, let us construct the inverse \(\Phi\) to the map \(\mathbb{D}\circ Q\) by 
    observing that a module \(K\in \Grr_R(\rbar)\) yields an \(R\)-submodule \(K'\subset R\) such that \(c\tilde{R}\subset K'\) and 
    \(K'\) maps to \(K\) by the projection \(R\to \rbar\).  Thus we can apply \(\Hom_R(-,R)\) to the short exact sequence:
    \[0\to \tilde{R}\to K'\to K'/c\tilde{R}\to 0.\] 
    The results is an exact sequence:
    \[0\to \Hom_R(K',R)\xrightarrow{\psi} \Hom_R(\tilde{R},R)\to \Ext^1(K'/c\tilde{R},R)\to 0\]
     Since \(\Hom_R(\tilde{R},R)=\tilde{R}\), we obtain an element of \(\Grr_R(\tilde{R})\). Moreover, the inclusion 
     map \(K'\to R\) is sent by the map \(\psi\) to \(1\in \tilde{R}=\Hom_R(\tilde{R},R)\) thus we actually obtain an element of 
     \(\Grr_R(\tilde{R}/R)\).

     Finally, let us observe that if \(K=\Ext^1_R(\tilde{R}/M,R)\) then \(K'=\Hom_R(M,R)\). Moreover, since \(M\) is torsion free we get 
     that \(\Hom_R(K',R)=M\) and that shows \(\Phi\circ \mathbb{D}\circ Q= \mathrm{Id}\). The argument for \( \mathbb{D}\circ Q\circ \Phi= \mathrm{Id}\) is analogous.

  \end{proof}

By composing the isomorphism
\(\mathbb{D}\circ Q : \Grr_R(\tilde{R}/R)
\rightarrow \Grr_R(\rbar)\) 
from Proposition \ref{prop:dq1}
with the class map \eqref{eq:classmap},
we obtain a function
\begin{equation}
\label{eq:extdef}
\gext: \Sp^g_{n,m}\to \Grr_R(\rbar),\quad \gext(L)=\mathbb{D}(Q(cl(L)).
\end{equation}
This function is continuous when restricted
to each $\hesspav(\mu)$. We summarize this in the following proposition.
%but it is a map of varieties 
%on each connected component
%%of \(\Grr_R(\overline{R})\) because $cl$
%is. Its fibers may be used to construct a paving of 
%$\Sp_{n,m}'$. 
%We define a function $\Upsilon^g:$
%Recall the jump composition $\mu=\jumpcomp_n(L)$ 
%associated to a lattice $L\in \latl(n,m)$ from Section \ref{sec:latts}.
\begin{prop}
\label{prop:jacmapgr}    
We have a continuous map {of varieties} 
$\gext :
\hesspav(\mu)\rightarrow \Grr_R(\bar{R})$
whose fiber is an affine space of 
dimension
$\maj(\mu)$.
%$=\sum_{i=1}^k (k-i)\mu_i$,
%where $k$ is the number of 
%parts of $\mu$. 
%Its image is the space of submodules $M$ for which the number of
%elements in the set
%\[\left\{j\in \Gamma_{n,n+1}-\Gamma(M):
%n(n-1-i)\leq j \leq n(n-i)-1 \right\}\] 
%is $\mu_1+\cdots+\mu_{k-i}$ for $1\leq i \leq k-1$.
The image of one of the Schubert cells 
$Y^{\grsup,\circ}_w \subset \hesspav(\mu)$
is the attracting set of $\extcomb(w)$ for the 
compatible action of $U$ on $\Grr_R(\rbar)$,
where $\extcomb$ is the combinatorial map from Definition \eqref{def:extcomb}
restricted to \(\Res_{(n)}^<(n,n+1)\).

%$w \in \Res^{<}(\mu)$.
%$\jumpcomp_n(L)=\mu$.
%the number of elements of
%$\Gamma(L)$ between $n(i-1)$ and $ni-1$ is
%$a_i$ where $\ba=\majt(\mu)$.
\end{prop}

%\subsection{Cell decomposition for the compactified Jacobian}
\subsection{Map to the punctual Hilbert scheme}
\label{sec:cell-decomp-comp}

%$\mathcal{S}'_{n,m,w}$

%To connect this with the previous section, notice that we have an alternate description
%$\CC[x,y]/\mathfrak{m}^{n-1}\cong \rbar$.
%The isomorphism is induced by the map
Define a homomorphism
$\evmap_{n,m}:\C[x,y]\rightarrow \tilde{R}$
which sends $f(x,y) \in \C[[x,y]]$ to
$f(\eps^n,\eps^{m})$. Then
%We then check that
\begin{lemma}
\label{lem:gcdker}    
The following preimages are given by
\begin{equation}
\label{eq:rbarhilb}
\evmap_{n,n+1}^{-1} ((\eps^\conduc))=\mathfrak{m}^{n-1},\ \ 
\evmap_{n,n+1}^{-1} ((\eps^{\conduc+n}))=\mathfrak{m}^n
 %   \mathfrak{m}^{n-1}=\langle x^ay^b: an+b(n+1) \ge \conduc\rangle,
\end{equation}
where $\conduc$ is the conductor and $\mathfrak{m}=(x,y)$ is the maximal ideal in $\C[[x,y]]$. In particular, $\evmap_{n,n+1}$ determines an isomorphism $\rbar\cong \C[[x,y]]/\mathfrak{m}^{n-1}$.
%which is the key property of the conductor.
\end{lemma}
\begin{proof}
The preimage is a monomial ideal, so
the first statement amounts to checking that
\begin{equation}
\mathfrak{m}^{n-1}=\langle x^ay^b: an+b(n+1) \ge \conduc\rangle,
\end{equation}
and similarly for the second case.
\end{proof}
%which follows from the 
%\end{equation}
%the inverse is determined by sending
%$M$ to the result of replacing $\eps^i$ for $i<f$ by $x^ay^b$, where $(a,b)$ are the unique nonnegative integers satisfying $an+b(n+1)=i$.
We then have an injective map 
$\psi:\Grr_R(\rbar)\rightarrow
\hilbgeo_0(\C^2)$ determined by sending
a submodule satisfying 
$c\tilde{R} \subset M\subset R$ 
to its preimage in $\C[[x,y]]$, recalling that $\hilbgeo_0(\C^2)$
denotes the punctual Hilbert scheme of all codimensions 
defined in Section \ref{sec:hilbcells}. This map intertwines
the action of $U$ with the one on $\hilbgeo_0(\C^2)$
from Proposition \ref{prop:gottsche}, and so carries attracting
sets to attracting sets. We can describe the image of
$\hesspav(\mu)$ in terms of \eqref{eq:grhilbpartition}.
\begin{lemma}
\label{lem:hesshilbmu}
The image of $\hesspav(\mu)$ under
$\psi \circ \extmap$
is $\ghilbgeo_0(\C^2)^+_{F_\mu}$,
where $F_\mu(u)$
is the expression from \eqref{eq:Fmu} before Proposition \ref{prop:idealcars}.
 \end{lemma}

\begin{proof}
%Usin
%The attracting sets
Using Proposition \ref{prop:jacmapgr}, we find that attracting sets are mapped to attracting sets. The result then follows from 
Proposition \ref{prop:gottsche}. To complete the argument we need to compute the images under \(\psi\circ\extmap\) of the \(U\)-fixed points.
This can be done by computing the \(U\)-virtual characters of the corresponding modules since \([\Hom_R(S,R)]-[\mathbb{D}(S)]=[S]^\vee\).
On the other hand \(\Hom_R(\mathbb{D}(\tilde{R}/M))=0\) for \(M\in \Grr_R(\tilde{R}/R)\) and \([\tilde{R}/M]^\vee\) for \(U\)-fixed
modules \(M\) can be matched with \([I_\lambda]\)  for a unique \(I_\lambda\in \ghilbgeo_0(\C^2)^{U}\).
%    describe the image.
\end{proof}

We give a second description of the attracting sets
%\[\hilbgeo_0(\C^2)_\mu=\hilbgeo_0(\C^2)_{\phi(z)}^{\C^*}\]
which relates them with the Hessenberg varieties of Section \ref{sec:hesscohomology}. 
Let $F^{std}_{\bullet}$
denote the standard flag in $\C^n$ 
given by
$F^{std}_i=\langle e_{1},...,e_i\rangle$, 
and let
$F^{opp}_{i}=\langle e_{n-i+1},...,e_n\rangle$ denote the opposite one. 
Let $N$ be the upper triangular
Jordan block matrix \eqref{eq:Ndef}.
Then we have
\begin{lemma}
\label{lem:grasshilb}
Let $F(z)$ and the coefficients $a_i$ 
correspond to a nonempty \(\mathcal{G}(\CC^2)_F\) in
Proposition \ref{prop:iarrobino}, 
and define $\ell=(\ell_1,...,\ell_n)$ by 
$\ell_i=a_{i-1}$ for any $n$ greater than the degree of $F(z)$.
We have an isomorphism 
%on each connected component
\[\phi:\ghilbgeo(\CC^2)_{F}
\rightarrow \hessgeo(\ell),\] 
where $\hessgeo(\ell)$ consists
of flags $(W_1\subset \cdots \subset W_n)$ of subspaces of 
$\C^n$ satisfying
    \begin{equation}
\label{eq:hesshilbcond}        
\dim(W_i)=\ell_i,\ \ 
N W_i\subset W_{i+1},\ \ 
W_i \subset F^{opp}_i,
\end{equation}
    and $N$ is the upper triangular 
    nilpotent matrix \eqref{eq:Ndef}.
\end{lemma}

\begin{proof}
Let us identify $\gr_{n-1} \C[x,y]$ with $\C^n$
by setting $e_i=x^{n-i}y^{i-1}$.
The map
is given by sending a 
homogeneous ideal $I$ to the flag
\begin{equation}
\label{eq:hesshilbmap}    
W_i=x^{n-i} (\gr_{i-1} I) \subset 
\gr_{n-1} \C[x,y] =\C^n.
\end{equation}
The fact that $xI \subset I$ shows that this collection of vector spaces is indeed nested, while the definition of $W_i$ shows that it is contained in $F_i^{opp}$.
The Hessenberg type condition
$N W_i\subset W_{i+1}$ is satisfied because $yI\subset I$. The inverse function is straightforward to describe, and it is clear that these are maps of varieties.

\end{proof}

In the previous lemma we assume that \(\ell_i\ge 0\) and \(\ell_{i+1}\ge \ell_i\). Thus if \(\ell_i=\ell_{i+1}\) for some \(i\)
then \(\hessgeo(\ell)\) is not a subvariety of a partial flag variety in the form it is presented in the lemma.
On other hand
we can realize $\hessgeo(\ell)$ as a 
subvariety of the partial Hessenberg variety
$\hessgeo_\mu(N,h_\mu)$, where $h_\mu$ is the 
partial version of the Peterson Hessenberg function
defined in Section \ref{sec:hesscohomology}. 
\begin{prop}
\label{prop:gottschehess}    
We have that $\hesspav(\mu)$ is isomorphic to 
a locally trivial affine bundle
over a smooth base, described by
\[\hessgeo(\ell)=\hessgeo_\mu(N,h_\mu)\cap \mathcal{Z}(\mu)\]
where $h_\mu$ is the parabolic Peterson Hessenberg function, and
\begin{equation}
\mathcal{Z}(\mu)=
\left\{(V_{d_1}\subset \cdots \subset V_{d_k})\in \gFl_\mu:
V_{d_i} \subset F^{opp}_i
\right\}
\end{equation}
for $d_i=\mu_1+\cdots+\mu_i$ and \(\ell_i=a_{i-1}\) are the coefficients of the expansion of
\(F_\mu\) defined by \eqref{eq:Fmu}.
The rank is equal to 
$n(n+1)/2-\sum_{i=1}^n \sched_i(\mu)$.
\end{prop}

\begin{proof}
{
{
By Propositions \ref{prop:iarrobino} and
 \ref{prop:jacmapgr} and Lemma 
\ref{lem:hesshilbmu},
we find that $\hesspav(\mu)$ is a fibration over
$\ghilbgeo_0(\C^2)^+_{F_\mu}$, which is isomorphic to 
the given base by Lemma \ref{lem:grasshilb}.}
   The rank of the bundle can be computed
    by checking that the desired formula
agrees with the sum of
$\maj(\mu)$ and the expression
for the rank in item \ref{item:iarrorank}
of Proposition \ref{prop:iarrobino}.}

\end{proof}

The cell decomposition from Proposition
\ref{prop:gottsche} agrees with the intersection of the parabolic Schubert cells with $\hessgeo(\ell)$, as both are the Bialynicki-Berula decompositions coming from the same torus action.

\begin{ex}
Let $F(z)={z}^{3}+{z}^{2}+2z+1$
and $n=5$, so that $\ell=(0,0,2,3,5)$, $\mu=(2,1,2)$,
and $h=h_\mu=(3,3,5,5,5)$.
There are two nontrivial subspaces $W_3,W_4$
in Lemma \ref{lem:grasshilb} of dimensions 2 and 3.
We have that
\begin{equation}
\label{eq:hessellsched}    
\hessgeo(\ell)=\hessgeo_\mu(N,h) \cap \fltaugeo(\mu)
\end{equation}
%where $\mu=(2,1,2)$, 
%and $h=(3,3,5,5,5)$,
%and 
where $\fltaugeo(\mu)\subset G/P_\mu$
is the partial variety of flags $(V_2 \subset V_3)$ in $\mathbb{C}^5$ satisfying $V_2\subset F^{opp}_{3}=\langle e_3,e_4,e_5\rangle$, and 
$V_3 \subset F^{opp}_4=\langle e_2,e_3,e_4,e_5\rangle$, identifying $V_2=W_3$, and $V_3=W_4$.

We use this to describe a cell $\hessgeo(\mu)$. Let 
$\sigma=(5,4,3,2,1)$,
and consider the parabolic Schubert cell
\begin{equation}
A_{\sigma,\mu}=
 \left(\begin {array}{ccccc} a_{1,5}& a_{1,4} & a_{1,3}&0&1\\ a_{2,5} &a_{2,4}&a_{2,3}
&1&0\\a_{{3,5}}&a_{{3,4}}&1&0&0
\\0&1&0&0&0\\ 1&0&0&0&0
\end {array} \right)
% \left( \begin {array}{ccccc} 0&0&0&0&1\\
% 0&0&0&1&0
%\\ 0&0&1&a_{{3,2}}&a_{{3,1}}\\ 0&1&
%a_{{4,3}}&a_{{4,2}}&a_{{4,1}}\\ 
%1&0&a_{{5,3}}&a_{{5,2
%}}&a_{{5,1}}\end {array} \right) 
    \end{equation}
    so that the span of the first two columns is $V_2$, and the first three span $V_3$.
Intersecting this cell with $\fltaugeo(\mu)$ amounts to setting $a_{1,5},a_{1,4},a_{1,3},a_{2,5},a_{2,3}$ 
equal to zero. The condition that 
$A_{\sigma,\mu}^{-1}N A_{\sigma,\mu}\in H_h$
from Definition \ref{def:hessvariety} gives the remaining equations $a_{3,4}=a_{2,3}$, and $a_{3,5}=-a_{2,3}a_{3,4}$.

In terms of the ideal, this describes a cell in
$\ghilbgeo^5(\mathbb{C}^2)_{F}$
given by substituting
$a_{3,4}=a_{2,3}=a$, and $a_{3,5}=-a^2$ into
%\[I=(x^3,x^2+a_{3,5}y^2,xy+a_{2,3}y^2,xy^2+a_{2,3}y^3,
%xy^3,y^4),\]
\[(x^2+a_{3,5}y^2,xy+a_{3,4}y^2,xy^2+a_{2,3}y^3,xy^3,y^4)\]
which is the preimage under $\varphi$.
%for which we have 
One can check that the these relations produce 
an ideal with codimension 
$\dim_{\mathbb{C}}(\mathbb{C}[x,y]/I)=5$, while 
the dimension drops for other values.
We show in Lemma \ref{lem:schedcells}
below that the full flag case of this 
Hessenberg variety may be paved by 
affines in a similar way.

\end{ex}

%We now obtain
%\begin{prop}
%\label{prop:jacpav}    
%The preimage of 
%$\gext^{-1} \gr^{-1} (\hilbgeo(\C^2,\mathfrak{m}^{n-1})^{\C^*})$
%is a vector bundle over $\hessgeo'(\ellvec)$.

%\end{prop}

%\begin{proof}
    
%\end{proof}

\subsection{Duality morphism in the flagged case}

\label{sec:dualASF}

We extend this construction to the flagged case, recalling the 
definition of $\hesspav(\tau)$ from \eqref{eq:defhesspav}. 
%Recall that 
%$\Sp_{n,n+1}$ parametrizes chains 
We have an alternative description of the affine flag variety $\Fl$ can as the collection of chains
$L_\bullet=(\cdots \supset L_i \supset \cdots)$ of lattices
$L_i\subset \tilde{F}$ satisfying
\begin{equation}\label{eq:chain}
\codim_{\tilde{R}}(L_i \cap \tilde{R})-\codim_{L_i}(L_i\cap \tilde{R})=i, \ \ 
L_{i+n}=\eps^n L_{i}.
\end{equation}
The affine Springer fiber $\Sp_{n,n+1}$ consists of flags which additionally satisfy $\eps^{n+1} L_i \subset L_i$,
in other words are $R$-modules for $R=\C[[\eps^n,\eps^{n+1}]]$.

We now define a flag version of the Jacobian and the $\gext$
map.
\begin{defn}
\label{def:flagjac}    
We define $\mathcal{F}l^d_R(N/N')$ 
to be the collection of 
flags of $R$-modules $(M_0\supset M_{1}\supset \cdots)$ such that
\begin{equation} 
M_i \subset N,\ \ 
\dim(N/M_{i})=i+d,\ \ 
    M_0 \supset N',\ \ 
    M_{i+n}=\eps^n M_{i}.
    \end{equation}
\end{defn}
Given a flag $L_\bullet \in \Fl$, we define a new flag
$L'_\bullet$ by
\begin{equation}
\label{eq:defclflag}
L'_{i}=\left\{v/v_{L_0}:v\in L_{i}\right\},
\end{equation}
where $v_L$ is the minimal element from \eqref{eq:classmap},
so that $L'_0=cl(L_0)$.
Then the flagged version of the $\gext$ map is given by
\begin{equation}
\gext: \Sp_{n,n+1}\rightarrow \mathcal{F}l^d_R(R/c\tilde{R}),\ \ 
\gext(L_\bullet)=M_\bullet,\ \ 
M_{i}=\gext(L'_{-i})
\end{equation}
where $d=\comaj(\tau)$.

We have an action of $U=\C^*$ on $\mathcal{F}l(R/c\tilde{R})$ by sending $z\cdot \eps^j= z^j\eps^j$.
Then the following statement is an analogue
of Proposition~\ref{prop:jacmapgr} in the flag setting:
\begin{prop}
\label{prop:extmapflag}
  The map $\gext$ yields a 
  %surjective
  map of varieties
%  
 % \[\mathbb{D}\circ Q': \Sp_{n,n+1}^0\to \mathcal{F}l(R,\rbar'\to\rbar).\]
 $\hesspav(\tau)\rightarrow 
 \mathcal{F}l^d_R(\tilde{R}/cR)$
 where $d=\comaj(\tau)$.
 The fiber is an affine space of dimension $\maj(\tau)$, 
 and the image of $Y^{\circ}_{w}$ for $w\in \Res(\tau)$ is the attracting set for the action of $U$.
\end{prop}

\begin{ex}
Let $w=(-1,4,3)$, which is one of two elements in 
$\Res(\tau)=\{(-1,3,4),(-1,4,3)\}$ 
for $\tau=(2,3,1)$. The corresponding 
three-dimensional cell 
in $\Sp_{3,4}$ is parametrized by
\begin{multline}\cdots \subset  (\eps^{-2}+a_{2,3} \eps^{-1}+a_{2,4},
\eps^{2},\eps^{3}) \subset\\
(\eps^{-2}+a_{2,4}-a_{2,3}a_{3,4},\eps^{-1}+a_{3,4},\eps^3) \subset (\eps^{-2},\eps^{-1},1) \subset \cdots 
\end{multline}
Here each entry is described by its generators over $\mathcal{O}=\C[[\eps^3]]$, with the first one corresponding to $L_0$. 
We have that $v_{L_0}=\eps^{-2}+a_{2,3} \eps^{-1}+a_{2,4}$,
so the corresponding sequence $L'_\bullet$ is given by
 \begin{multline}
\cdots \supset (1,\eps^4,\eps^5)\subset (1,\eps+(a_{3,4}-a_{2,3})\eps^2,\eps^5)\subset (1,\eps,\eps^2)\subset \cdots
 \end{multline}
 Notice that there is one degree of freedom, so that the rank
 of the fiber is given by $\maj(\tau)=\dim(Y^\circ_w)-1=2$.
 The image under the $\gext$ map is then
 \begin{multline}
(\eps^3,\eps^4,\eps^8)\supset 
(\eps^3+(a_{2,3}-a_{3,4})\eps^4,\eps^7,\eps^8) \\
\supset
(\eps^6,\eps^7,\eps^8)
\supset \cdots \in \mathcal{F}l^1_R(R/c\tilde{R})
 \end{multline}
 where the sequence begins at $M_0$,
whose codimension in $R$ is $\comaj(\tau)=1$.
\end{ex}

\begin{proof}
  The argument is in line with the proof of Proposition~\ref{prop:dq1}. Indeed, the analogue of  \eqref{eq:R/tilR}
  is the sequence:
  \[0\to R\to \eps^{-n}\tilde{R}\to \eps^{-n}\tilde{R}/R\to 0.\]
  By applying \(\Hom_R(-,R)\) to  sequence and arguing the same way as above we obtain \(\Ext^1_R(\eps^{-n}\tilde{R},R)=0\) and
  \begin{equation}\label{eq:wt-sht-2}
  \Ext^1_R(\eps^{-n} \tilde{R}/R,R)\{\conduc+n-1\}\simeq R/\eps^nc\tilde{R}\simeq \rbar'.\end{equation}

  The inclusion map of \(\Ext^1_R(\Lambda_{-i},R)\) inside \(\rbar'\) is constructed from the
  result of application of \(\Hom_R(-,R)\)  to the analogue of \eqref{eq:sesMM}:
  \[0\to M/R\to \eps^{-n}\tilde{R}/M\to \eps^{-n}\tilde{R}/R\to 0.\]

  Thus we have shown that \(M_{-i}=\mathbb{D}(\eps^{-n}\tilde{R}/L_{-i})\) has a natural inclusion inside \(\rbar'\). The inclusions between \(\Lambda_\bullet\)
  induce the \(R\)-morphisms between the modules \(M_{-i}\) and these morphisms satisfy the defining  conditions for the space  \(\mathcal{F}l(R,\rbar'\to\rbar)\).

  Finally, for showing that \(\mathbb{D}\circ Q'\) is an isomorphism we need an  analogue of the sequence~\eqref{eq:sesMR}:
  \[0\to M\to \eps^{-n}\tilde{R}\to \eps^{-n}\tilde{R}/M\to 0.\]
  Just as in the proof of Proposition~\ref{prop:dq1}, we can apply \(\Hom_R(-,R)\)  to the above sequence to 
  prove that \(\Ext^1_R(\eps^{-n}\tilde{R}/M,R)=\Hom_R(M,R)/c\tilde{R}\). 
    %\comment{The ``$t$'' in red is just a stray character, yes?}
  The involutive property of the duality
  implies the desired statement because \(\Ext^1_R(M,R)=0\).

\end{proof}

\subsection{The Parabolic flag Hilbert scheme}
\label{sec:par-hilbert}
We now extend the map $\psi$ to the punctual 
Hilbert scheme. To do this, we first recall a flagged version of $\hilbgeo(\C^2)$, introduced in
\cite{carlsson2017parabolic}:
\begin{defn}
\label{def:pfhilb}
The \emph{parabolic flag Hilbert scheme} is the variety of flags
\begin{equation}
\label{eq:pfhilb}
\pfhilbgeo^{d,n}(\C^2)=\left\{(I_0 \supset \cdots \supset I_{n}):
I_j \in \hilbgeo^j(\C^2),\ I_{n}\supset xI_0\right\}
\end{equation}
\end{defn}
\begin{prop}[\cite{carlsson2017parabolic}, Theorem 4.1.6]
\label{prop:pfhsmooth}
    The parabolic flag Hilbert scheme 
    $\pfhilbgeo^{d,n}(\C^2)$
    is a smooth quasi-projective variety of dimension $2d+n$.
\end{prop}

\begin{proof}
    We have a map $\pfhilbgeo^{d,n}\hookrightarrow \hilbgeo^{N}(\C^2)$
    for $N=(n+1)d+n(n+1)/2$ defined as follows. For each flag of ideals, let
    \begin{equation}J=\sigma^*(I_{0})+x
    \label{eq:blendedieal}
    \sigma^*(I_{1})+\cdots +x^n \sigma^*(I_{n})
    \end{equation}
    where $\sigma^*(I)$ is the result of substituting $x=x^{n+1}$. The smoothness is checked by showing that the image is in the fixed point locus of of the action of $\mathbb{Z}/(n+1)\mathbb{Z}$ on $\hilbgeo^N(\C^2)$ which multiplies $x$ by the $(n+1)$-th root of unity. 
\end{proof}

We have an action of $T=\C^*\times \C^*$ 
on $\pfhilbgeo$ by acting on each factor $I_j$ as in Section \ref{sec:hilbcells}. 
Its fixed points are given by $I_{\parflag}=(I_{\lambda^0}\supset \cdots 
\supset I_{\lambda^{n}})$ for 
$\parflag\in\pfhilbcomb^d(n)$. The image 
$J$ of $I_{\parflag}$ in $\hilbgeo^N(\C^2)$ in the proof is the monomial ideal corresponding to  \emph{blended partition} 
$\blendlam=\blendpar(\parflag)$ defined so that the number of boxes with $x$-coordinate equal to $k$ in 
$\blendlam$
is the number of boxes with $x$-coordinate equal to $i$ 
in $\lambda^j$, where $k=i(n+1)+j$ so that 
$j\in \{0,...,n\}$ is the remainder.
%For instance, if 
%$\parflag=((2, 1)\subset(2, 2)\subset(2, 2, 1))$, then
%$\blendlam=(6,5,1)$.

Let $\ppfhilbgeo(\C^2)$ denote the punctual subspace
of $\pfhilbgeo(\C^2)$
%denote the punctual version
for which $I_j \in \philbgeo(\C^2)$ for each $j$.
We have a map 
$\psi:\mathcal{F}l^d_R(R/c\tilde{R})
\rightarrow \ppfhilbgeo(\C^2)$ given by
\begin{equation}
\label{eq:psiflag}
\psi(M_\bullet)=(I_0 \supset \cdots \supset I_{n}),
\ \ I_{j}=\evmap_{n,n+1}^{-1}(M_j).
\end{equation}
This is compatible with
the action of $U=\C^*_{n,n+1}$ on $\pfhilbgeo(\C^2)$.
\begin{prop}
\label{prop:exthilbflag}
The composition $\psi\circ \gext$ defines a map of varieties
\[\psi\circ\gext: \quad \hesspav(\tau)\rightarrow  \pfhilbgeo^d(\C^2)\]
where $d=\comaj(\tau)$.
The image of $Y^\circ_w$ for $w\in \Res(\tau)$ is the 
attracting set to $\parflag=(\psi\circ\extcomb)(w)$
for the action of $U$, and the fiber is an affine space of dimension $\maj(\tau)$.
\end{prop}

\begin{proof}
By Proposition \ref{prop:extmapflag} and the compatibility of
the action of $U$ with $\psi$, the image of each
$Y^\circ_w$ must be in the attracting set. On the other hand, if $I_\bullet$ is in the attracting set of
$\parflag=(\psi \circ\extcomb)(w)$ for 
$w\in \Res(\tau)$, then each $I^j$
is in the attracting set of $I_{\lambda^j}$. Since 
$B_{\lambda^j}(z^n,z^{n+1})$ has degree at most
$n^2-1$ for those partitions, 
we find that $I_j$ has initial terms of
$x^iy^j$ for all $i+j\geq n$, so that
$\mathfrak{m}^n\subset I_j$ for all $j$.
By Lemma \ref{lem:gcdker}, there is a
unique element
$M_j \in \Grr_R^{d+j}(R/c\tilde{R})$ whose image under 
$\psi$ is $I_j$ for each $j$. These unique elements must define a flag, and since $I_n \supset x I_0$, this flag may be extended to $M_\bullet \in \mathcal{F}l_R^d(R/c\tilde{R})$ by taking 
$M_{i}=\eps^n M_{i-n}$ for $i\geq n$.
Since $M_\bullet$ is in the attracting set 
of $\extcomb(w)$ for $U$,
it must be in the image of some element of
$\hesspav(\tau)$ under $\gext$ by Proposition 
\ref{prop:extmapflag} again. The final statement 
follows from that proposition, and since 
$\psi$ is injective on these elements.

\end{proof}

We now study these attracting sets, by 
proving analogues of Iarrobino and G\"{o}ttsche's results. 
Let $\pfghilbgeo(\C^2)$ be the space of flags of homogeneous ideals $I_j \in \ghilbgeo(\C^2)$, and denote the
components by
%\begin{equation} and
$\pfghilbgeo(\C^2)_{F^\bullet}$ in which
$I_j \in \ghilbgeo_{F^j}$, and similarly
for $\pfghilbgeo(\C^2)^+_{F^\bullet}$.

Given chain of ideals \(I_j\) from $\pfghilbgeo(\C^2)^+_{F^\bullet}$ we construct the chain homogeneous ideals
\(\mathrm{lt}(I_i)\) from \(\pfghilbgeo(\C^2)^+_{F^\bullet}\) by applying lower term extraction to the elements
of the ideal. Thus we obtain an algebraic retraction map:
\[ \varphi: \pfghilbgeo(\C^2)^+_{F^\bullet}\to \pfghilbgeo(\C^2)_{F^\bullet},\]
and will study the fibers of this map below.

Let $C^{a,b}_{\parflag}$
denote the attracting set for the action of $\C^*_{a,b}$ on $\pfhilbgeo(\C^2)$ as in the description of $C^{a,b}_{\lambda}$ from Proposition \ref{prop:escells}, and let
$G_{\parflag}^{a,b}=C^{a,b}_{\parflag} \cap \pfghilbgeo(\C^2)$ be homogeneous cell.
\begin{lemma}
\label{lem:pfhcells}
If $n>\deg(F^j(z))$ for all $j$, 
then the attracting sets $C^{n,n+1}_{\parflag}$
and $G^{n,n+1}_{\parflag}$ 
for $\hs_z(\lambda^j)=F^j(z)$
determine cell decompositions of 
$\pfghilbgeo(\C^2)^+_{F^{\bullet}}$ and 
$\pfghilbgeo(\C^2)_{F^\bullet}$ respectively.
They are realized as affine subspaces of the 
cell $C^{a,b}_{\blendlam}$
for $(a,b)=(n,(n+1)^2)$ as follows:
%is an affine space 
\begin{enumerate}[a)]
    \item The attracting cell
    $C^{n,n+1}_{\parflag}$ 
    %can be realized 
%    as the subspace of 
is determined by setting
$c^{r,s}_{i,j}=0$ in $C^{a,b}_{\blendlam}$
for  unless $r-i$ is a multiple
of $(n+1)$.
\item The homogeneous cell
$G^{n,n+1}_{\parflag}\subset C^{n,n+1}_{\parflag}$ is 
the result of setting $c^{r,s}_{i,j}=0$ unless 
$(r-i)+(n+1)(s-j)=0$.
\end{enumerate}

\end{lemma}

\begin{proof}
The attracting set
$C^{n,n+1}_{\parflag}$ is identified with the intersection of $C^{a,b}_{\blendlam}$ with
$\pfhilbgeo(\C^2)$ under \eqref{eq:blendedieal}, noting that the torus actions are compatible.
Now notice that both $\pfhilbgeo(\C^2)$ and 
$\pfghilbgeo(\C^2)$ can be described as the invariants of subgroups of $T$ 
in $C^{a,b}_{\blendlam}$,
and so are determined by conditions on the coordinates by the $T$-equivariance 
statement from Proposition \ref{prop:escells}.

To see that the union of $C_{\parflag}^{n,n+1}$
is $\pfghilbgeo(\C^2)^+_{F^\bullet}$, suppose that $I_\bullet$ is in the cell. Then each 
$I_j$ is in the attracting cell to $\lambda^j$, so by Proposition \ref{prop:gottsche} it lies in
$\ghilbgeo(\C^2)^+_{F^j}$. On the other hand, if 
$I_\bullet \in 
\pfghilbgeo(\C^2)^+_{F^\bullet}$, then each 
$I_j\in \ghilbgeo(\C^2)^+_{F^j}$, so by 
G\"{o}ttsche's result again we have that it is in
$C_{\lambda^j}$ for some collection of partitions
$\lambda^j$ which must necessarily form a flag,
and we find that $I_j \in C^{n,n+1}_{\parflag}$.
The argument is similar for 
$G^{n,n+1}_{\parflag}$.
\end{proof}

Applying Proposition \ref{prop:exthilbflag}, we must have $\dim(Y^{\circ}_w)=\maj(\tau)+\dim(C_{\parflag}^{n,n+1})$ for any
$w\in \Res(\tau)$. In light of Lemma 
\ref{lem:pfhcells} and Proposition \ref{prop:extstats}, the corresponding statement
in terms of parking functions is a
combinatorial identity
\begin{equation}
\label{eq:identitycells}    
\dinv(P)+\area(P)+
|\mathcal{C}_{\parflag}^{n,n+1}|=
n(n-1)/2,
\end{equation}
where $\parflag=(\psi\circ \extcomb)(P)$ and as in the non-flag case,
$\mathcal{C}_{\parflag}^{n,n+1}$ is the set of parameters of $C_{\parflag}^{n,n+1}$ as described
in the lemma. On the other hand, we will see in the next section that we have an identity
\begin{equation}
\label{eq:identityhess}    
\sum_{P\in \cars(\tau)} q^{|\mathcal{G}_{\parflag}^{n,n+1}|}=
\prod_i [\sched_i(\tau)]_q,\ \ 
\parflag=(\psi\circ \extcomb)(P).
\end{equation}
\begin{ex}
    For the final parking function 
    $P\in \cars(\tau)$ for $\tau=(3,1,2,5,4)$
    from Figure \ref{fig:cars}, we have 
    $\dinv(P)=2$ and $\area(P)=\maj(\tau)=5$.
The blended partition is given by
$\blendlam=(20,9,7,5,4)$,
    and coefficients
    $\mathcal{C}_{\parflag}^{5,6}$ are
$\{c^{2,4}_{20,0}, c^{3,4}_{9,1}, c^{1,4}_{7,2}\}$, giving a value of 10 in \eqref{eq:identitycells}.
\end{ex}

\begin{prop}
\label{prop:bundleflag}
We have that $\pfghilbgeo(\C^2)^+_{F^\bullet}$
is an affine bundle 
over a smooth base
$\pfghilbgeo(\C^2)_{F^\bullet}$.
The rank is
${n(n+1)/2}-\sum_i \sched_i(\tau)-\maj(\tau)$.
\end{prop}
\begin{proof}

By Lemma \ref{lem:iarrobinounion}, there is a change of variables $y=cx$ taking each
$I_j$ for $I_\bullet\in \pfghilbgeo(\C^2)_{F^\bullet}^+$
into one for which 
$I_j \in C_{\lambda^j}^{n,n+1}$, where 
$\lambda^j$ is the partition with a normal pattern. In fact, the collection of such values of $c$ is generic, and there is one carrying
$I_\bullet$ into
$C_{\parflag}^{n,n+1}$, where every $\lambda^j$ is normal, which in fact is open. 
We have seen that $C_{\parflag}^{n,n+1}\cong \mathbb{A}^k\times G_{\parflag}^{n,n+1}$ for such flags in Lemma \ref{lem:pfhcells}, so the first statement follows. 

The base is smooth because it is an invariant subspace of a smooth complex variety under an action of $\C^*$.
The remaining statement comes from checking the combinatorial identity that $|\mathcal{C}_{\parflag}^{n,n+1}|\setminus
|\mathcal{G}_{\parflag}^{n,n+1}|$ is the expression for the rank for the flag with normal pattern (or in fact, for all of them).
\end{proof}

\subsection{Relation with the 
Hessenberg varieties}

We describe the base space
$\pfghilbgeo(\C^2)_{F^\bullet}$
as a sublocus in a Hessenberg variety,
so that
the permutations satisfying $P_\sigma\in \fltaugeo(\tau)$ consists of the set
$\fltaucomb(\tau)$ from Definition
\ref{def:fltaucomb}. 

\begin{defn}
Let $\tau\in S_n$ have run composition 
$\mu=(\mu_1,...,\mu_k)$.
Let $\fltaugeo(\tau)\subset \gFl_n$ 
to be the subset of flags $(V_1\subset \cdots \subset V_n)$ satisfying
\begin{equation}
\label{eq:fltaudefi}
V_j \subset F^{opp}_{n-i} \mbox{ whenever
$j\leq \mu_1+\cdots+\mu_{k-i}$}    
\end{equation}
for $1 \leq i \leq k-1$.
We define 
\begin{equation}
    \label{eq:defhesstaugeo}
    \hessgeo(\tau)=\hessgeo(N,\hf_\tau)\cap \fltaugeo(\tau)
\end{equation} where
$N$ is the upper triangular nilpotent matrix with one Jordan block \eqref{eq:Ndef}, and $h_\tau$ is the Hessenberg function described from \eqref{eq:hftau}.
\end{defn}

In fact, $\mathcal{Z}(\tau)$ is a type of
Schubert variety called a tower of Grassmannians, 
which is always smooth.

\begin{ex}
    For $\tau=(6,4,5,3,2,1)$, we would have
$\mu=(1,3,2)$. The intersection of the top-dimensional Schubert cell
would be parametrized by
\begin{equation}
\label{eq:fltauschubex}    
A_{\sigma}=
 \left( \begin {array}{cccccc} 0&0&0&0&a_{{1,2}}&1\\ 0
&a_{{2,5}}&a_{{2,4}}&a_{{2,3}}&1&0\\ a_{{3,6}}&a_{{3,
5}}&a_{{3,4}}&1&0&0\\ a_{{4,6}}&a_{{4,5}}&1&0&0&0
\\ a_{{5,6}}&1&0&0&0&0\\ 1&0&0&0&0&0
\end {array} \right)
\end{equation}
The extra zeros in the upper left corner
are the result of removing the first 
$\mu_1=1$ columns from row 2, and the first $\mu_1+\mu_2=4$ columns from row 1.
\end{ex}

Recall the sequence of Hilbert series
$F^\bullet_\tau(z)$ from Section \ref{sec:hilbcomb}. We have
\begin{lemma}
\label{lem:hesshilb}    
We have an isomorphism 
$\pfghilbgeo(\C^2)_{F_\tau^\bullet}\cong \hessgeo(\tau)$.
\end{lemma}

\begin{proof}

We produce a map
$\Phi:\pfghilbgeo(\C^2)_{F_\tau^\bullet}\rightarrow 
\mathcal{F}l_n$ which extends the one from Lemma
\ref{lem:grasshilb}, and check that the image
lies in $\hessgeo(\tau)$.
Given $\tau$, let $\bc=(c_1,...,c_n)$ be the result of sorting $\majt(\tau)$ in reverse order, so that $c_i$ is the number of spaces from the diagonal of the boxed numbered $i$ in any parking function $P\in \cars(\tau)$. We define $\Phi(I_\bullet)=
(V_1\subset \cdots \subset V_n)$
where $V_n=\gr_{n-1} \C[x,y]$ by setting
\begin{equation}
\label{eq:Phimap}
V_i=x^{c_i} \gr_{n-c_i-1}(I_{n-\tau_i}). 
\end{equation}

    The nested condition for the ideals \(I_{\bullet}\) along the runs of $\tau$
    and the fact 
    $I_j \supset xI_i$ for all $i,j$
    implies that the $V_i$ are indeed nested. 
    The Hessenberg condition for the function $h_\tau$ is satisfied since
    $I_{n-\tau_i}\supset yI_{n-\tau_j}$ whenever
    $c_i=c_{j+1}$ and $\tau_j<\tau_i$. The image is in $\fltaugeo(\tau)$ since every element of $V_i$ is a multiple of $x^{c_i}$ by \eqref{eq:Phimap}, and it can be seen that the inverse construction is well defined, making it a bijection. Using Iarrobino's change of variables as in the proof of Proposition \ref{prop:bundleflag},
    we can check that both are maps of varieties using the in the coordinates $\mathcal{C}_{\parflag}^{n,n+1}$
    for the top-dimensional cell, and the ones 
    defined below on $\hessgeo(\tau)$.
\end{proof}

Putting this together gives our third main result, which was Theorem C from the introduction. 
\begin{thm}
\label{thm:geo}
For each $\tau\in S_n$, the space 
$\hesspav(\tau)$ is
a smooth quasi-projective variety isomorphic to a locally trivial affine bundle with base
$\hessgeo(\tau)$, which is a smooth projective variety. The rank of the bundle is
${ n(n+1)/2}-\sum_i \sched_i(\tau)$.
\end{thm}
We then have
\begin{cor}\label{cor:closure}
For any $\tau\in S_n$, we have an isomorphism \[G_{\majt(\tau)}\hbm_*(\Sp_{n,n+1})\cong
\hbm_*(\hessgeo(\tau))\] 
up to a grading shift.
  \end{cor}

\begin{proof}
We have that $G_{\majt(\tau)}\hbm_*(\Sp_{n,n+1})
\cong \hbm_*(\hesspav(\tau))$
by the long exact sequence in Borel-Moore homology,
and the fact that $\hesspav(\tau)$ has only even dimensional homology. The grading shift 
in $\hbm_*(\hesspav(\tau))\cong \hbm_*(\hessgeo(\tau))$
is the rank from Theorem \ref{thm:geo}.
\end{proof}

We now explain how to use this theorem to recover statements from Theorem \ref{thmb}.
\begin{lemma}
\label{lem:schedcells}
The image of $G^{n,n+1}_{\parflag}$
where $\parflag=(\psi\circ \extcomb)(w)$
under $\Phi$
is the intersected Schubert cell
$X^\circ_\sigma \cap \hessgeo(\tau)$, 
for $\sigma=\reshess_\tau(w)$.
Its coordinates may be described by { setting} 
$a_{1,\tau_j}=0$ for all $\tau_j$ not in the final run of $\tau$ from the set $\mathcal{A}^h_\sigma$ described in \eqref{eq:hesscoords}.

\end{lemma}

\begin{proof}
Let $\fltaugeo_t(\tau)\subset \gFl_n$ denote 
the collection of flags satisfying \eqref{eq:fltaudefi}
for $1\leq i \leq t$ so that $\fltaugeo(\tau)=\fltaugeo_{k-1}(\tau)$ where $k$ is the number of runs of $\tau$.
We claim that
\begin{equation}
\label{eq:wtauztau}
    \fltaugeo_1(\tau)\cap \hessgeo(h_\tau)=
\fltaugeo(\tau)\cap \hessgeo(h_\tau)=
\hessgeo(\tau).
\end{equation}
On the other hand, the intersection on the left hand 
side turns out to
intersect $\hessgeo(h_\tau)$ transversely.
%making it advantageous for describing cells. 
Notice that this coincides with 
Lemma \ref{lem:eq:hessleq} by taking flags coming from permutation matrices.

To prove \eqref{eq:wtauztau}, it suffices to show that
$\fltaugeo_{t-1}(\tau) \cap \hessgeo(h_\tau) =\fltaugeo_t(\tau) \cap \hessgeo(h_\tau)$
for all $2\leq t \leq k-1$, for then the statement follows by induction.
For this, suppose $(V_1\subset \cdots \subset V_n)\in 
\fltaugeo_{t-1} \cap \hessgeo(h_\tau)$,
and let us show that
$V_j \subset F^{opp}_{n-t}$ for all
$j\leq \mu_1+\cdots+\mu_{k-t}$, establishing
\eqref{eq:fltaudefi} for $i\leq t$.
First, by the Hessenberg condition,
we have $N V_j \subset V_{h_\tau(j)}$.
%
%Then by , 
%for $j\leq \mu_1+\cdots+\mu_{k-t+1}$, we have
%$V_j \subset F^{opp}_{n-t+1}$. By th
On the other hand, we can easily check from 
the definition of $h_\tau$ 
given in \eqref{eq:hftau}
that
\begin{equation}
\label{eq:hftauineq}    
h_\tau(j)
\leq \mu_1+\cdots+\mu_{k-t+1}
\end{equation}
whenever $j\leq \mu_1+\cdots+\mu_{k-t}$,
so that $V_{h_\tau(j)}\subset F^{opp}_{n-t+1}$.
Putting these together, we find that 
$N V_j\subset F^{opp}_{n-t+1}$, which 
implies that
$V_j \subset F^{opp}_{n-t}$ since $N$ is injective on $V_j$.

Now notice that the relations in
\eqref{eq:hesscellform} take the form of
expressing certain coordinates in 
a Schubert cell $A_\sigma$ appearing in rows 2 and greater in terms of other variables.
On the other hand, the new relations describing
$\fltaugeo_1(\tau)$ amount to setting coordinates
in first row equal to zero, so the remaining ones give
a coordinate system.

\end{proof}

\begin{ex}
Let $\tau=(2,4,1,3)$ and let us consider the top-dimensional cell $A_\sigma$ for $\sigma=(4,3,2,1) \in \hesscomb(\tau)$.
In coordinates, the Schubert cell has the form
\begin{equation}
A_{\sigma}=
     \left( \begin {array}{cccc} a_{1,4}&a_{1,3}&a_{{1,2}}&1\\ a_{{2,
4}}&a_{{2,3}}&1&0\\ a_{{3,4}}&1&0&0
\\ 1&0&0&0\end {array} \right)
\end{equation}
Intersecting with $\hessgeo(h_\tau)$ gives 
the relation
$a_{2,4}=a_{1,3}+a_{1,2}\left( a_{{3,4}}-a_{{2,3}} \right)$
which forces $NV_1\subset V_3$,
whereas the conditions coming from $\fltaugeo_1(\tau)$ allows us
to eliminate the two upper-left coordinates
$a_{1,3}=a_{1,4}=0$.

In terms of the ideal, we have
\[I_0=(x^2+a_{{3,4}}xy+a_{{2,4}}{y}^{2},
xy+a_{{2,3}}{y}^{2},
x{y}^{2}+a_{{1,2}}{y}^{3},{y}^{3}) \]
subject to the above relation 
$a_{2,4}=a_{1,2}(a_{3,4}-a_{2,3})$.
This appears to have three degrees of freedom,
but we may actually eliminate the $a_{1,2}y^3$
using $y^3$, so there are only two. Since $a_{2,4}$ is determined by the others, we obtain a two-dimensional affine space, determining a cell in the projective plane 
describing the space of the $I_0$.

\end{ex}

\begin{cor}\label{cor:ftau}
  The variety \(\hessgeo(\tau)\) has Poincar\'{e} polynomial
  \[\sum_i \dim H^i(\hessgeo(\tau)) v^i=\prod_{j=1}^n 
  [\sched_j(\tau)]_{v^2}\]
  Moreover, the pushforward of the inclusion
  $i_*:H^*(\hessgeo(\tau))\rightarrow 
  H^*(\mathcal{F}l)$ using Poincar\'{e} duality is injective with image \((f_\tau(x))\subset H^*(\mathcal{F}l)\), where $f_\tau(x)$ is the polynomial from \eqref{eq:ftau}
  in the Chern classes.
\end{cor}

\begin{proof}
The Poincar\'{e} polynomial follows from taking the dimensions of the cells in the paving
and using Propositions \ref{prop:extstats} and \ref{prop:hessbij}.
For the second statement, we can check that the first factor in \eqref{eq:ftau} is the Euler class of 
$\mathcal{Z}_1(\tau) \supset \mathcal{Z}(\tau)$ from the proof of Lemma \ref{lem:schedcells}.
  The Hessenberg variety 
  \(\hessgeo(N,h_\tau)\) is the zero locus of the section
  \(N\) of the vector bundle whose fiber over  \(V_\bullet\) is \(\oplus_{i}\Hom(V_i/V_{i-1},V_{h_{  \tau}(i)+1})\), and we can see that its Euler class is the second factor. By Lemma \ref{lem:schedcells},
  we can see that the varieties intersect transversely, 
  so the Euler class of the intersection is the product of the two factors.

\end{proof}

\bibliographystyle{plain}

%\bibliography{expanded.bib}
\end{document}